%% file: IEEE_v12_One_Col_final_Aybat_arxiv.tex
\documentclass[12pt,draftcls, onecolumn]{IEEEtran}

\usepackage{enumerate,setspace,multirow,xcolor,cite,microtype}
\usepackage{graphicx,amsmath,amssymb,amsfonts,url, algorithm, algorithmic, etoolbox, amsthm}
\usepackage{array,subcaption}
\usepackage{dsfont,booktabs} 

\DeclareMathOperator*{\sgn}{sgn}

\def\admm{\textbf{ADMM}}
\def\sadmm{\textbf{SADMM}}
\def\fprod#1{\left\langle#1\right\rangle}
\def\prox#1{\mathbf{prox}_{#1}}

\def\T{\mathsf{T}}

\def\blambda{\lambda} 

\newtheorem{theorem}{Theorem}
\newtheorem{remark}[theorem]{Remark}
\newtheorem{definition}{Definition}
\newtheorem{lemma}[theorem]{Lemma}
\newtheorem{corollary}[theorem]{Corollary}

\,

\newcommand{\LCal}{\mathcal{L}}

\newcommand{\be}{\begin{equation}}
\newcommand{\ee}{\end{equation}}
\newcommand{\ba}{\begin{array}}
\newcommand{\ea}{\end{array}}
\newcommand{\bpm}{\begin{pmatrix}}
\newcommand{\epm}{\end{pmatrix}}

\AtBeginEnvironment{align}{\setcounter{subeqn}{0}}
\newcounter{subeqn} %

\input defs.tex
\begin{document}
%
\title{Distributed Linearized Alternating Direction Method of Multipliers for Composite Convex Consensus Optimization}


\author{\IEEEauthorblockN{N. S. Aybat\IEEEauthorrefmark{1},
Z. Wang\IEEEauthorrefmark{1},
T. Lin\IEEEauthorrefmark{2},
and S. Ma\IEEEauthorrefmark{2}}\\
\IEEEauthorblockA{\IEEEauthorrefmark{1}IME Department, The Pennsylvania State University, PA, USA. \\Emails: {\tt nsa10@psu.edu}, {\tt zxw121@psu.edu}}\\
\IEEEauthorblockA{\IEEEauthorrefmark{2}Department of SEEM, The Chinese University of Hong Kong, Hong Kong. \\Emails: {\tt linty@se.cuhk.edu.hk}, {\tt sqma@se.cuhk.edu.hk}.}%
\thanks{Research of N. S. Aybat was partially supported by NSF grants CMMI-1400217 and CMMI-1635106. Research of S. Ma was supported in part by the Hong Kong Research Grants Council General Research Funds Early Career Scheme (Project ID: CUHK 439513).}}

%



\IEEEtitleabstractindextext{%
\vspace*{-4mm}
\begin{abstract}
Given an {undirected} graph $\cG=(\cN,\cE)$ of agents $\cN=\{1,\ldots,N\}$ connected with edges in $\cE$, we study how to compute an optimal decision on which there is consensus among agents and that minimizes the sum of agent-specific private convex composite functions $\{\Phi_i\}_{i\in\cN}$ while respecting privacy requirements, where $\Phi_i\triangleq \xi_i + f_i$ belongs to agent-$i$. Assuming only agents connected by an edge can communicate, we propose a distributed proximal gradient method DPGA for {consensus optimization over both unweighted and weighted static (undirected) communication networks}. In one iteration, each agent-$i$ computes the prox map of $\xi_i$ and gradient of $f_i$, and this is followed by local communication with neighboring agents. We also study its stochastic gradient variant, SDPGA, which can only access to noisy estimates of $\grad f_i$ at each agent-$i$. This computational model abstracts a number of applications in distributed sensing, machine learning and statistical inference. We show ergodic convergence in both sub-optimality error and consensus violation for DPGA and SDPGA with rates $\cO(1/t)$ and $\cO(1/\sqrt{t})$, respectively.
\vspace*{-2mm}
\end{abstract}

}

\maketitle

\IEEEdisplaynontitleabstractindextext

%
\IEEEpeerreviewmaketitle

\section{Introduction}
\label{sec:intro}

Let $\cG=(\cN,\cE)$ denote a \emph{connected} undirected graph of $N$ computing nodes where nodes $i$ and $j$ can communicate information only if $(i,j) \in \cE$.  Each node $i\in\cN\triangleq\{1,\ldots,N\}$ has a \emph{private} (local) cost function $\Phi_i:\reals^n\rightarrow\reals\cup\{+\infty\}$ such that \vspace*{-1mm}
\begin{equation}
\label{eq:F_i}
\Phi_i(x)\triangleq \xi_i(x) + f_i(x),
\end{equation}
where $\xi_i$ is a possibly \emph{non-smooth} convex function, and $f_i$ is a \emph{smooth} convex function.
\begin{assumption}
\label{assumption-0}
$\xi_i: \mathbb{R}^n \rightarrow \mathbb{R}\cup\{+\infty\}$, and $f_i: \mathbb{R}^n \rightarrow \mathbb{R}$ are closed convex functions such that $f_i$ is differentiable on an open set containing $\dom \xi_i$ with a Lipschitz continuous gradient $\grad f_i$, of which Lipschitz constant is $L_i$; and the prox map of $\xi_i$,
\begin{equation}
\label{eq:prox}
\prox{\xi_i}(x)\triangleq\argmin_{y \in \reals^n} \left\{ \xi_i(y)+\tfrac{1}{2}\norm{y-x}^2 \right\},
\end{equation}
is \emph{efficiently} computable for $i\in\cN$.
\end{assumption}
In this paper, we study a distributed consensus problem~\cite{Tsitsiklis84_1T}; in particular, we consider solving a multi-agent consensus optimization problem of minimizing the sum of privately known composite convex functions in \eqref{eq:F_i} satisfying Assumption~\ref{assumption-0}: \vspace*{-2mm}
\begin{equation}
\label{eq:central_opt}
F^*\triangleq\min_{x \in \mathbb{R}^n}\sum_{i\in\cN} \Phi_i(x).
\end{equation}
We consider the setting where only local information exchange is allowed, i.e., there is no central node such that the data can be collected, and only neighboring nodes can exchange data; and we focus on the following equivalent formulation:
\begin{equation}
\label{eq:dist_opt}
\min_{\bx}\Big\{ F(\bx)\triangleq\sum_{i\in\cN}\Phi_i(x_i):\ x_i=x_j,\ \forall~(i,j)\in\cE\Big\}.
\end{equation}
where $\bx=[x_i]_{i\in\cN}\in\mathbb{R}^{n|\cN|}$. 
We call $\bar{\bx} = [\bar{x}_i]_{i\in \cN}$, $\epsilon$-feasible
if the consensus violation satisfies $\max_{(i,j)\in\cE}\big\{\norm{\bar{x}_i-\bar{x}_j}\big\} \leq \epsilon$
and $\epsilon$-optimal if $\big|\sum_{i\in\cN}\Phi_i(\bar{x}_i)-F^*\big| \leq
\epsilon$.

This computational setting, i.e., decentralized consensus optimization, appears as a generic model for various applications in signal processing, e.g., \cite{chang2015multi,lesser2003distributed,ling2010decentralized,ravazzi2015distributed,schizas2008consensus}, machine learning, e.g.,~\cite{forero2010consensus,mcdonald2010distributed,yan2013distributed} and statistical inference, e.g.,~\cite{lopuhaa1991breakdown, mateos2010distributed}. Clearly, \eqref{eq:central_opt} can also be solved in a ``centralized" fashion by communicating all the private functions $\Phi_i$ to a \emph{central} node, and solving the overall problem at this
node. However, such an approach can be very expensive both from communication
and computation perspectives when compared to the distributed algorithms which are far more scalable to increasing problem data and network sizes. In particular, suppose $(A_i,b_i) \in \reals^{m \times (n+1)}$ and $\Phi_i(x) = \norm{A_i x - b_i}^2 +\lambda \norm{x}_1$ for some given $\lambda>0$ for $i\in\cN$ such that $m\ll n$ and $N\gg 1$. Hence, \eqref{eq:central_opt} is a very large scale LASSO problem with \emph{distributed} data. 
To solve~\eqref{eq:central_opt} in a centralized fashion, the data
$\{(A_i,b_i): i \in \cN\}$ needs to be communicated to the central node. This can be
prohibitively expensive, and may also violate privacy constraints -- in case some node $i$
 does not want to reveal the details of its \emph{private} data.
Furthermore, it requires that the central node has large enough memory to
be able to accommodate all the data. On the other hand, at the expense of slower
convergence, one can completely do away with a central node, and seek for
\emph{consensus} among all the nodes on an optimal decision
using ``local" decisions communicated by the neighboring nodes.
These
considerations in mind, we
propose decentralized algorithms that can compute
solutions to~\eqref{eq:central_opt} using only local computations without explicitly requiring
the nodes to communicate the functions $\{\Phi_i: i \in \cN\}$; thereby,
circumventing all privacy, communication and memory issues.

The contribution of the paper can be summarized as follows: 1) we propose a proximal gradient alternating direction method of multipliers (PG-ADMM) and its stochastic gradient variant SPG-ADMM to solve composite convex problems; we only assume that the prox map of $\xi_i$ can be computed efficiently, while other 
ADMM based algorithms are efficient when either $\Phi_i=\xi_i+f_i$ or both $\xi_i$ and $f_i$ have \emph{simple} prox maps that can be computed efficiently; {2) we establish that when the gradient is deterministic PG-ADMM is equivalent to primal-dual algorithms for saddle-point problems proposed in~\cite{chambolle2015ergodic,condat2013primal} -- hence, SPG-ADMM extends these algorithms to noisy gradient setting;} 3) we show ergodic convergence of both (expected) suboptimality and consensus violation bounds for PG-ADMM with the rate of $\cO(1/t)$, and for SPG-ADMM with the rate of $\cO(1/\sqrt{t})$; 4) we implement PG-ADMM and SPG-ADMM on 
consensus formulations of \eqref{eq:central_opt} for unweighted and weighted static communication networks -- this gives rise to two different \emph{node-based} distributed algorithms: DPGA and DPGA-W and their stochastic gradient variants SDPGA and SDPGA-W -- and we examine the effect of the underlying network topology on their convergence rate. 5) The proposed algorithms DPGA, DPGA-W, SDPGA and SDPGA-W are fully distributed, i.e., the agents are \emph{not} required to know any global parameters depending on the entire network topology, e.g., the second smallest eigenvalue of the Laplacian; instead, we only assume that agents know who their neighbors are. Using only local communication, our \emph{node-based} distributed algorithms require less communication burden and memory storage compared to edge-based distributed algorithms. 6) Proposed algorithms consist of a \emph{single loop}, i.e., there are no outer and inner iteration loops; therefore, they are easy and practical to be implemented over distributed networks.

To sum up, there are many practical problems where one can compute the prox map for $\xi_i$
efficiently; however, computing the prox map for $\Phi_i=\xi_i+f_i$ is not easy. The methods proposed in this paper can compute an $\epsilon$-optimal $\epsilon$-feasible solution in
$\cO(\epsilon^{-1})$ iterations \emph{without} assuming bounded $\grad f_i$ for any $i\in\cN$; each iteration of these methods requires computing $\prox{\xi_i}$ and $\grad f_i$ for $i\in\cN$, and one or two communication rounds among the neighbors -- hence, $\cO(\epsilon^{-1})$ communications per node in total.\vspace*{-3mm}
\begin{table*}[t!]
\setlength\extrarowheight{1pt}
  \centering
  {
    \vspace*{-0.25 cm}
    \caption{Comparison of our method with the previous work}
    \label{tab:prev}
    \resizebox{\textwidth}{!}{
    \begin{tabular}{|l|l|l|l|l|l|l|}
      \hline
      \multirow{2}{*}{\textbf{Reference}} & \multirow{2}{*}{\textbf{assumption on} $\Phi_i$} & \multirow{2}{*}{operation / iter.} &iter \# for & iter \# for & comm. steps& Single Loop\\
      & & & $\epsilon$-feas. & $\epsilon$-opt. & $\epsilon$-opt. & (\textbf{simple})\\ \hline
      Duchi et al. (2012)~\cite{Duchi12} & convex, Lipschitz cont. & subgrad., projection& unknown & $\mathcal{O}(1/\epsilon^2)$ &$\cO(1/\epsilon^2)$ & yes \\
      \hline
      Nedic \& Ozdaglar (2009)~\cite{nedic2009distributed} & convex & subgrad. &$\mathcal{O}(1)$ & $\mathcal{O}(1/\epsilon^2)$ & $\mathcal{O}(1/\epsilon^2)$& yes \\
      \hline
      Wei \& Ozdaglar (2012)~\cite{wei2012_1} & strictly convex & $\prox{\Phi_i}$& unknown & $\cO(1/\epsilon)$&$\cO(1/\epsilon)$ & yes\\ \hline
      Makhdoumi \& Ozdaglar (2014)~\cite{makhdoumi2014broadcast} & convex & $\prox{\Phi_i}$ & $\cO(1/\epsilon)$& $\cO(1/\epsilon)$& $\cO(1/\epsilon)$& yes \\ \hline
      Wei \& Ozdaglar (2013)~\cite{wei20131} & convex & $\prox{\Phi_i}$ & $\cO(1/\epsilon)$ & $\cO(1/\epsilon)$ & $\cO(1/\epsilon)$ & yes \\ \hline
      \multirow{2}{*}{Jakovetic et al. (2014)~\cite{jakovetic2011fast} (D-NC)} & smooth convex& \multirow{2}{*}{$\grad{\Phi_i}$} &\multirow{2}{*}{$\mathcal{O}(1/\sqrt{\epsilon})$} &  \multirow{2}{*}{$\mathcal{O}(1/\sqrt{\epsilon})$} &\multirow{2}{*}{ $\mathcal{O}(\log(1/\epsilon)/\sqrt{\epsilon})$} & \multirow{2}{*}{no} \\
      & bounded $\grad \Phi_i$ & & & & &\\ \hline
      \multirow{2}{*}{Chen \& Ozdaglar (2012)~\cite{chen2012fast}} & composite convex $\Phi_i=\xi+ f_i$& \multirow{2}{*}{$\prox{\xi}$, $\grad f_i$}& \multirow{2}{*}{$\cO(1/\sqrt{\epsilon})$} & \multirow{2}{*}{$\cO(1/\sqrt{\epsilon})$} & \multirow{2}{*}{$\cO(1/\epsilon)$}& \multirow{2}{*}{no} \\
      & bounded $\grad\gamma_i$ & & & & &\\ \hline
      Shi et al. (2015)~\cite{shi2015proximal} (PG-EXTRA) & composite convex $\Phi_i=\xi_i+ f_i$& $\prox{\xi_i}$, $\grad \gamma_i$ &$\mathcal{O}(1/\epsilon^2)$ & $\mathcal{O}(1/\epsilon^2)$~KKT & $\cO(1/\epsilon^2)$ & {yes} \\ \hline
      Aybat et al. (2015)~\cite{icml2015_aybat15} (DFAL) & composite convex $\Phi_i=\xi_i+ f_i$& $\prox{\xi_i}$, $\grad f_i$ &$\mathcal{O}(1/\epsilon)$ & $\mathcal{O}(1/\epsilon)$ & $\cO(1/\epsilon)$ & {no} \\ \hline
       Our work (DPGA, DPGA-W) & composite convex $\Phi_i=\xi_i+f_i$& $\prox{\xi_i}$, $\grad f_i$ &$\mathcal{O}(1/\epsilon)$ & $\mathcal{O}(1/\epsilon)$ & $\cO(1/\epsilon)$ & {yes} \\ \hline
    \end{tabular}
  }}
  \vspace*{-0.5cm}
\end{table*}
\subsection{Notation}
Throughout the paper, $\norm{\cdot}$ {denotes the Euclidean or the spectral norm depending on its argument, i.e., for a matrix $A$, $\norm{A}=\sigma_{\max}(A)$}. Let $d\in\integers_+$. $\mathbf{1}_{d}\in\reals^{d}$ is the vector of all ones. Let $\mathbb{S}_d$ denote the set of symmetric matrices in $\reals^{d\times d}$, $\otimes$ denotes the Kronecker product, and $I_d$ is the $d\times d$ identity matrix. Given a graph $\cG=(\cN,\cE)$, $\Omega\in\reals^{|\cN|\times|\cN|}$ denotes the graph Laplacian. Given a 
positive definite matrix $Q\in\mathbb{S}^n$, i.e., $Q\succ\textbf{0}$, define $\norm{v}_Q\triangleq\sqrt{v^\top Q v}$ for all $v\in\reals^n$.\vspace*{-2mm}
\subsection{Related work}
\label{sec:related_work}
A number of different distributed 
algorithms have been proposed to solve \eqref{eq:dist_opt} -- see Table~\ref{tab:prev} that displays some recent work. Duchi et al.~\cite{Duchi12} proposed a dual averaging algorithm to
solve~\eqref{eq:central_opt} in a distributed fashion over $\cG$ 
when each $\Phi_i$ is convex.
This algorithm computes $\epsilon$-optimal solution 
in
$\mathcal{O}(1/\epsilon^2)$ iterations; however, they do not provide
any guarantees on the consensus violation
$\max\{\norm{\bar{x}_i-\bar{x}_j}_2:\
(i,j)\in\cE\}$. Nedi\'{c} and Ozdaglar~\cite{nedic2009distributed}
developed a subgradient method with constant step size
$c>0$ for distributed minimization of \eqref{eq:central_opt} where
the network topology is time-varying. 
Setting the subgradient stepsize $c=\cO(\epsilon)$ in their method guarantees to compute a solution $\bar{\bx}=[\bar{x}_i]_{i\in\cN}$ such that its consensus violation $\max\{\norm{\bar{x}_i-\bar{x}_j}_2:\
(i,j)\in\cE\}\leq \epsilon$ within $\cO(1)$ iterations; and its suboptimality is bounded from above as $\sum_{i\in\cN}\Phi_i(\bar{x}_i)-F^*\leq\epsilon$ within $\cO(1/\epsilon^2)$ iterations; however, since the step size is
constant, neither suboptimality nor consensus errors are guaranteed to decrease further. Although these algorithms are for more general problems and assume mere convexity on each $\Phi_i$, this generality comes at the cost of $\mathcal{O}(1/\epsilon^2)$ complexity bounds, and they also tend to be very slow in practice. On the other extreme, under much stronger conditions: assuming each $\Phi_i$ 
is \emph{smooth} and has \emph{bounded} gradients, 
Jakovetic et al.~\cite{jakovetic2011fast} developed a fast distributed gradient method D-NC
with $\mathcal{O}(\log(1/\epsilon)/\sqrt{\epsilon})$ convergence rate in communication rounds.
For the quadratic loss, which is one of the most commonly used loss functions, \emph{bounded} gradient assumption does \emph{not} hold. In terms of distributed applicability, D-NC requires all the nodes $\cN$ to agree on a doubly stochastic weight matrix $W\in\reals^{|\cN|\times|\cN|}$; it also assumes that the second largest eigenvalue of $W\in\reals^{|\cN|\times|\cN|}$ is known globally among all the nodes -- this is not attainable for very large scale fully distributed networks. D-NC is a two-loop algorithm: for each outer loop $k$, each node computes their gradients once, and it is followed by $\cO(\log(k))$ communication rounds. 
In the rest, 
we briefly discuss those algorithms that balance the trade-off between the iteration complexity and the required assumptions on $\{\Phi_i\}_{i\in\cN}$.

Wei \& Ozdaglar~\cite{wei2012_1,wei20131}, and
recently Makhdoumi \& Ozdaglar~\cite{makhdoumi2014broadcast} proposed distributed ADMM
algorithms that
can compute an $\epsilon$-optimal  and $\epsilon$-feasible
solution in $\mathcal{O}(1/\epsilon)$ prox map evaluations for each $\Phi_i$. 
These algorithms have superior iteration complexity compared to the subgradient methods discussed above. That said, there are many practical problems where one can compute 
$\prox{\xi_i}$
efficiently; but, computing the prox map for $\Phi_i=\xi_i+f_i$ is
\emph{not} easy -see Section \ref{sec:numerical} for an example. One can
overcome this limitation of
ADMM by locally splitting variables, i.e., setting
$\Phi_i(x_i,y_i)\triangleq\xi_i(x_i)+f_i(y_i)$, and adding a constraint $x_i=y_i$ in
\eqref{eq:dist_opt}. However, this approach more than \emph{doubles} local memory requirement; 
and in order for
ADMM to be efficient, 
 the prox maps for \emph{both} $\xi_i$ and $f_i$ 
still must be simple. 


When node functions $\Phi_i$ are composite convex, i.e., $\Phi_i=\xi+f_i$, assuming that
the non-smooth term $\xi$ is the \emph{same} at all nodes,
and $\grad f_i$ is \emph{bounded} for all $i\in\cN$, Chen \& Ozdaglar~\cite{chen2012fast} proposed an inexact
proximal-gradient method, which exploits the function structure, for distributed minimization of
\eqref{eq:central_opt} over a time-varying network topology. Their method also consists of two loops; it can compute $\epsilon$-feasible and
$\epsilon$-optimal solution 
in $T=\cO(1/\sqrt{\epsilon})$
iterations which require $k$ communication rounds with neighbors during the $k$-th iteration for each $1\leq k\leq T$ -- hence, leading to $\sum_{k=1}^{T}k=\cO(\epsilon^{-1})$
communications per node in total. Note that there are also many practical problems where nodes in the network have different non-smooth components in their objective and/or have different preference when choosing non-smooth regularizers.
In contrast, our methods 
allow node specific non-smooth functions $\xi_i$, do not
assume bounded $\grad f_i$ for any $i\in\cN$, and are still able to compute
an $\epsilon$-optimal $\epsilon$-feasible solution in
$\cO(\epsilon^{-1})$ iterations. 

Recently, Shi et al.~\cite{shi2015proximal} and Aybat et al.~\cite{icml2015_aybat15} proposed proximal gradient based distributed algorithms that can solve \eqref{eq:central_opt} over a static connected network when $\Phi_i=\xi_i+f_i$ as in \eqref{eq:F_i}.
In~\cite{shi2015proximal}, the proximal gradient exact first-order algorithm (PG-EXTRA) is proposed;
PG-EXTRA is an extension of 
the algorithm EXTRA~\cite{shi2015extra} to handle the non-smooth terms $\{\xi_i\}_{i\in\cN}$. They showed $\cO(1/t)$ convergence for the ergodic average of \emph{squared} residuals for the consensus and KKT violations. That said, we consider their rate result as $\cO(1/\sqrt{t})$ because their result on consensus violation, $\frac{1}{t}\sum_{k=1}^t\norm{U\bx^k}^2=\cO(1/t)$, only guarantees $\norm{U\bar{\bx}^t}=\cO(1/\sqrt{t})$, where $\bx^k=[x^k_i]_{i\in\cN}\in\reals^{n|\cN|}$ denotes PG-EXTRA iterate at iteration $k\geq 1$, $\bar{\bx}^t\triangleq\sum_{k=1}^t\bx^k/t$ denotes their ergodic average, and $U\in\reals^{n|\cN|\times n|\cN|}$ such that null space of $U$ only contains $\mathbf{1}\in\reals^{n|\cN|}$, i.e., $U\bx=\mathbf{0}$ implies $x_1=\ldots=x_N$. PG-EXTRA is a node-based distributed algorithm, and each node $i\in\cN$ stores \emph{four} different copies of local variable $x_i$ at each iteration $k$, say $x_i^{k+3/2}, x_i^{k+1}, x_i^{k+1/2}$ and $x_i^{k}$, and requires  \emph{two} rounds of local communications to be able to compute next iterates -- one can reduce the information exchange to one round per iteration at the expense of increasing the storage 
at node $i$ from $4n$ to $(4 + d_i)n$. In terms of 
applicability, PG-EXTRA requires all the nodes $\cN$ to agree on two \emph{symmetric} mixing matrices: $W,\tilde{W}\in\reals^{|\cN|\times|\cN|}$ such that $\tilde{W}\succ\mathbf{0}$ and $\frac{W+I}{2}\succeq \tilde{W}\succeq W$ -- see Assumption~1 in~\cite{shi2015proximal}. The gradient stepsize, $c$, is the \emph{same} for all nodes, and should satisfy $c<\sigma_{\min}(\tilde{W})/L_{\max}$. If $W$, $\tilde{W}$ are chosen such that $W\succeq\textbf{0}$ (this requires coordination among all the agents), and $\tilde{W}=\frac{W+I}{2}$, then $c$ can be chosen $c\in(0,\frac{1}{L_{\max}})$, which is independent of the global topology, and only depends on $L_{\max}\triangleq\max_{i\in\cN}L_i$ -- some max-consensus algorithm is needed to compute $L_{\max}$, and this may not be feasible for very large scale fully distributed networks.

Aybat et al.~\cite{icml2015_aybat15} proposed a distributed first-order augmented Lagrangian (DFAL) algorithm to solve \eqref{eq:central_opt}, where each $\Phi_i$ is a composite convex function as in \eqref{eq:F_i}. Assuming $\xi_i$ is bounded below by a norm, i.e., $\xi_i(.)\geq\norm{.}$, and it has a uniformly bounded subdifferential for each $i\in\cN$, they showed that any limit point of DFAL iterates is optimal; and for any $\epsilon>0$, an $\epsilon$-optimal and $\epsilon$-feasible solution can be computed within $\cO(\log(\epsilon^{-1}))$ DFAL iterations, which require $\cO(\frac{\sigma_{\max}^{1.5}(\Omega)}{d_{\min}}~\epsilon^{-1})$ gradient computations and communications per node in total, where
$d_{\min}$ is the degree of the smallest degree node.
Based on our tests and the results reported in \cite{icml2015_aybat15} the algorithm DFAL performs very well in practice; however, due to its double-loop structure, 
distributed implementation requires a more complex network protocol. Specifically, checking the subgradient stopping criterion for inner iterations requires evaluating a logical conjunction over $\cG$, which may not be easy for large networks.

Table~\ref{tab:str_commu} summarizes the storage and communication requirements for the algorithms discussed above. We will illustrate their practical performance in Section~\ref{sec:numerical}. 
After we started writing this paper, we became aware of other recent work~\cite{bianchi2014stochastic,chang2015multi,Ling15_1J,hong2015stochastic} for solving \eqref{eq:central_opt} over a connected graph $\cG$. These methods are very closely related to our proximal gradient ADMM~(PG-ADMM), and are based on linearized ADMM method. Suppose $\Phi_i(x)=\xi_i(x)+f_i(A_i x)$. When compared to our smooth convexity assumption on $\{f_i\}_{i\in\cN}$, Chang et al.~\cite{chang2015multi} showed the convergence of their distributed method under a far more stringent assumption: $f_i$ is \emph{strongly} convex with a Lipschitz continuous gradient for $i\in\cN$. Moreover, under this stronger assumption, they were able to show linear convergence rate \emph{only} when the non-smooth terms are absent, i.e., $\xi_i\equiv 0$, and $A_i$ has full column rank for all $i\in\cN$. Finally, their distributed algorithm requires the global knowledge of $\sigma_{\min}(\Omega+2W)$ of the graph $\cG$, where $W$ is the adjacency matrix. 
On the other hand, the algorithms we propose in this paper are fully distributed, i.e., the agents do not require the knowledge of some global parameters depending on the entire network topology; instead, we only assume that agents know who their neighbors are. Ling et al.~\cite{Ling15_1J} were able to show the convergence of their distributed method without strong convexity when penalty parameter is chosen sufficiently large; however, no rate has been shown for this setting -- again, as in~\cite{chang2015multi} determining whether the parameter is large enough requires the global knowledge of $\sigma_{\min}(\Omega)$. The algorithm in~\cite{Ling15_1J} is similar to our DPGA algorithm, and in contrast to sublinear rate result shown in this manuscript for the convex setting, Ling et al. were able to establish a rate result only under strong convexity assumption. Bianchi et al.~\cite{bianchi2014stochastic} also proposed a distributed algorithm based on linearized ADMM for solving \eqref{eq:central_opt}, where each node computes proximal gradient steps; the authors proved its convergence and also showed almost sure convergence for its randomized version, where a random set of agents become active and compute their proximal gradient steps and broadcast the recently updated local variables to their neighbors. However, no rate has been shown for neither the deterministic nor the randomized versions. The methods in~\cite{bianchi2014stochastic} run on edge-based formulations of the decentralized problem;
consequently, information exchange, computational effort and memory requirement are far more expensive than node-based algorithms proposed in this paper. {While we are finalizing our paper, we become aware of the work~\cite{hong2015stochastic}, 
where the authors also develop a distributed algorithm
based on linearized ADMM for solving \eqref{eq:central_opt} over both random and static networks and they attain similar rate results to ours. For the static network setting, their algorithm achieves $\cO(1/t)$ rate using deterministic gradient and $\cO(1/\sqrt{t})$ rate using the stochastic gradient; however, in contrast to our results, these rates are established assuming \emph{bounded} domain for all $\xi_i$ (for both deterministic and stochastic gradient settings); explicit bounds for suboptimality and infeasibility are not separately provided; and when the gradient is noisy, their algorithm does not have a compact characterization using only primal local decisions (see Theorem~4.2 and Algorithm~1 in~\cite{hong2015stochastic}) -- even if the network is static, in case the gradient is noisy, one needs to use Algorithm~1 which requires updating edge-variables and explicitly computing the dual variables, while our algorithm SDPGA using stochastic gradient is in a compact form updating \emph{only} primal node-variables, does not explicitly compute the dual iterates and still achieves $\cO(1/\sqrt{t})$ rate without assuming compact domain for any $\xi_i$.}

{
The focus of our paper is on synchronous computation over undirected static communication topology; that said, there has been work considering more general settings, e.g., see~\cite{tsianos2012push,nedic2015distributed} for distributed optimization on directed graphs, and~\cite{lobel2011distributed} for computation over random networks.}

\begin{table}[htbp]
\centering
\vspace{-1mm}
\caption{\small Comparison on data amount stored and communicated per node per communication round}
\resizebox{0.85\columnwidth}{!}{
\begin{tabular}{l c c c c c c}
\toprule
 & DPGA & DPGA-W & ADMM& SADMM & DFAL & PG-EXTRA  \\
\midrule
Storage & $3n$ & $3n$  & $3n$ & $7n$ & $4n$ & $4n$ or $(4+d_i)n$ \\
Communication & $n$ & $2n$ & $2n$ & $4n$ & $n$ & $2n$ or $(n)$  \\
\bottomrule
\end{tabular}
}
\label{tab:str_commu}
\vspace{-2mm}
\end{table}
To sum up, unlike two-loop methods, e.g., \cite{icml2015_aybat15,jakovetic2011fast,chen2012fast}, DPGA algorithms proposed here have only single-loop, and they are \emph{very easy} to implement - see Fig.~\ref{alg:DPGA-II} and Fig.~\ref{alg:DPGA-I}. These surprisingly simple algorithms can compute an $\epsilon$-feasible and $\epsilon$-optimal solution to \eqref{eq:central_opt} within $\cO(1/\epsilon)$ communication rounds among neighboring nodes for all $\epsilon>0$ (this is the best rate known for our setting) under much weaker assumptions on $\xi_i$ and $f_i$ compared to DFAL and with much simpler set of instructions compared to all the algorithms discussed above. To the best of our knowledge, in terms of storage and communication requirements per communication round, and convergence rate in terms of communication rounds, DPGA achieves the best guarantees known in the literature for problem~\eqref{eq:central_opt} when $\Phi_i$ is as in \eqref{eq:F_i} for $i\in\cN$.
\section{Proximal Gradient ADMM and its connections} 
Let $\cN=\{1,\ldots,N\}$, and $\{\Phi_i\}_{i\in\cN}$ be a collection of composite convex functions satisfying Assumption~\ref{assumption-0}.
Let $g:\reals^{n_y}\rightarrow\reals\cup\{+\infty\}$ be a possibly non-smooth convex function, and $f:\reals^{n_x}\rightarrow\reals$ denote $f(\bx)=\sum_{i\in\cN}f_i(x_i)$, where $n_x\triangleq\sum_{i\in\cN}n_i$ and $\bx = [x_i]_{i\in\cN}\in\reals^{n_x}$. 
We adopt the notation $\bx = [x_i]_{i\in\cN}=(x_i; x_{-i})\in\reals^{n_x}$ with $x_i\in\reals^{n_i}$ and $x_{-i} = [x_j]_{j\in\cN\setminus\{i\}}\in \reals^{n_x-n_i}$ to denote a vector where $x_i$ and $x_{-i}$ are treated as variable and parameter sub-vectors of $\bx$, respectively; and $\grad_{x_i}f(\bx)\in\reals^{n_i}$ denotes the sub-vector of $\grad f(\bx)\in\reals^{n_x}$ corresponding to components
of $x_i\in\reals^{n_i}$. In this section we consider the following problem: \vspace*{-1mm} 
{\small
\begin{align}
\label{prob}
\min_{\bx, y} \quad F(\bx, y)\triangleq g(y)+\sum_{i=1}^N \Phi_i(x_i) \quad
\text{s.t.} \quad A_i x_i + B_i y = b_i:\ \lambda_i, \quad i\in\cN,
\end{align}}%
where $A_i\in\reals^{m_i\times n_i}$, $B_i\in\reals^{m_i\times n_y}$, $b_i\in\reals^{m_i}$ for all $i\in\cN$, and $\lambda_i\in\reals^{m_i}$ denotes the dual variable corresponding to the $i$-th constraint. We assume that 
$\{\xi_i\}_{i\in\cN}$ and $g$ have easy prox maps. {Given $\{\gamma_i\}_{i\in\cN}\subset\reals_{+}$}, define the augmented Lagrangian function as \vspace*{-2mm}
{\small
\[\LCal_\gamma(\bx,y,\blambda) = g(y) + \sum_{i=1}^N \xi_i(x_i) + \phi_\gamma(\bx,y,\blambda),\vspace*{-2mm}\]}%
where $\blambda=[\lambda_i]_{i\in\cN}$, and $\phi_\gamma$ denotes the smooth part of the augmented Lagrangian, i.e., \vspace*{-1mm}
{\small
\[\phi_\gamma(\bx,y,\blambda) = \sum_{i=1}^N f_i(x_i) + \sum_{i=1}^N \lambda_i^\top (A_i x_i + B_i y - b_i) + \frac{1}{2}\sum_{i=1}^N {\gamma_i}~\|A_i x_i + B_i y - b_i\|^2.\]}%
Note setting {$\gamma_i=0$ for all $i\in\cN$}, $\LCal_\gamma$ becomes the standard Lagrangian function $\LCal(\bx,y,\blambda)$.

Consider the algorithm PG-ADMM stated below for solving \eqref{prob}: for $k \geq 0$ compute \vspace*{-1mm}
{\small
\begin{subequations}\label{GADM-extension}
\begin{align}
x_i^{k+1} & \triangleq  \prox{c_i\xi_i} \Big(x_i^k - c_i\nabla_{x_i}\phi_\gamma(x_i^k,y^k,\blambda^k) \Big), \quad i\in\cN \label{eq:PGADMM-x}\\
y^{k+1}  & \triangleq  \text{$\argmin_{y}$} \ \LCal_\gamma(x^{k+1},y,\blambda^k) \label{eq:PGADMM-y}\\
\lambda_i^{k+1} & \triangleq  \lambda_i^k + {\gamma_i}~(A_i x_i^{k+1} + B_i y^{k+1} - b_i), \quad i\in\cN \label{eq:PGADMM-lambda}
\vspace*{-2mm}
\end{align}
\end{subequations}}%
where $c_i>0$ is the gradient step size for $i\in\cN$, which should be related to $L_i$ and $\norm{A_i}$. In Section~\ref{sec:deterministic}, we study the convergence properties of \eqref{GADM-extension} given a deterministic first-order oracle which returns $\grad f_i$; and we also consider the effect of using stochastic first-order oracle, which returns noisy observations of $\grad f_i$, on the convergence rate.

\begin{definition}
Given $f:\reals^n\rightarrow\reals\cup\{+\infty\}$ such that $\dom f$ is open, and $\grad f$ exists on $\dom f$. $G:\reals^n\times\Xi\rightarrow\reals^n$ is called a stochastic first-order oracle (SFO) for $\grad f$ if there exists $\sigma\in[0,+\infty)$ such that for all $x\in\dom f$,
$\mathbb{E}_\nu \big[  G(x,\nu) \big] = \grad f(x), \quad \text{and} \quad \mathbb{E}_\nu \big[ \norm{G(x,\nu) - \grad f(x)}^2 \big] \leq \sigma^2$,
where $\nu\in\Xi$ is a random variable following a certain distribution.
\end{definition}
\begin{definition}
\label{eq:def:SFO-II}
Let $G_i$ denote an SFO for $\grad f_i$ for $i\in\cN$ with common parameter $\sigma$. Let {$\bar{G}_i$} be an SFO for $\grad_{x_i} \phi_{\gamma}$ 
defined as
{$\bar{G}_i(x_i,\nu_i; x_{-i}, y, \blambda) = G_i(x_i,\nu_i) + A_i^\T \left(\lambda_i + \gamma_i (A_ix_i + B_i y -b_i)\right)$.}
\end{definition}
\begin{assumption}\label{assumption-2} $\{\nu_i^k\in\Xi: i\in\cN,\ k\in\integers_+\}$ is a set of 
i.i.d. random variables.\end{assumption}

Under Assumption~\ref{assumption-2}, consider the algorithm SPG-ADMM stated below for solving \eqref{prob}: \vspace*{-1mm}
\begin{align}
\label{GADM-stoc-extension}
x_i^{k+1} & \triangleq  \prox{c_i^k\xi_i} \Big(x_i^k - c^k_i{\bar{G}_i}(x_i^k,\nu_i^k;x_{-i}^{k},y^k,\blambda^k) \Big), \; i\in\cN,\quad \eqref{eq:PGADMM-y},\ \eqref{eq:PGADMM-lambda}
\end{align}
for $k \geq 0$, where $c_i^k>0$ is the stochastic-gradient step size for $i\in\cN$ at the $k$-th iteration.
\begin{remark}
In the extreme case that $\xi_i=0$ for $i\in\cN$ and $|\cN|=1$, PG-ADMM \eqref{GADM-extension} and SPG-ADMM \eqref{GADM-stoc-extension} reduce to G-ADMM and SG-ADMM that take gradient steps for $x_i$-subproblems and have been studied in \cite{Ma-Zhang-EGADM-2013} and \cite{Gao-Jiang-Zhang-2014}. Specifically, \cite{Gao-Jiang-Zhang-2014} proves the $O(1/t)$ convergence rate of G-ADMM and $O(1/\sqrt{t})$ convergence rate of SG-ADMM. Our PG-ADMM and SPG-ADMM can be viewed as extensions of G-ADMM and SG-ADMM where general $\xi_i$'s are allowed. Our proofs of convergence rate results for PG-ADMM and SPG-ADMM in Section~\ref{sec:deterministic} are 
inspired by \cite{Gao-Jiang-Zhang-2014}.
\end{remark}
\noindent \emph{Preliminaries and Simple Identities:}
In our analysis, we used two well-known identities frequently: \vspace*{-5mm}
{\small
\begin{eqnarray}
 (v_{2}-v_{1})^\top Q(v_{3}-v_{1}) &=& \frac{1}{2}\left(\|v_{2}-v_{1}\|_Q^{2}+\|v_{3}-v_{1}\|_Q^{2}-\|v_{2}-v_{3}\|_Q^{2}\right), \label{identity-3}\\
 (v_1-v_2)^{\top}Q(v_3-v_4) &=& \frac{1}{2}\left(\|v_1-v_4\|_Q^{2}-\|v_1-v_3\|_Q^{2}\right)+\frac{1}{2}\left(\|v_2-v_3\|_Q^{2}-\|v_2-v_4\|_Q^{2}\right). \label{identity-4}
\end{eqnarray}
}%
For simplicity, we adopted the following notation to denote the stacked vectors or tuples:
$\bx=[x_1^\top,\ldots,x_N^\top]^\top$, $\bx^k=[{x_1^k}^\top,\ldots,{x_N^k}^\top]^\top$, and $\bx^*=[{x_1^*}^\top,\ldots,{x_N^*}^\top]^\top$. Let
$u=[\bx^\top y^\top]^\top$, $\lambda=[\lambda_1^\top\ldots \lambda_N^\top ]^\top$, and $w=[u^\top \lambda^\top]^\top$; and let $u^k$, $u^*$, $\lambda^k$, $\lambda^*$, $w^k$, $w^*$ be defined similarly. For the sake of simplifying the notational burden, we also adopted $(\bx,y,\lambda)\triangleq[\bx^\top,y^\top,\lambda^\top]^\top$. Let $m\triangleq\sum_{i\in\cN}m_i$, $n_x\triangleq\sum_{i\in\cN}n_i$, and $F:\reals^{n_x+n_y}\rightarrow\reals\cup\{+\infty\}$ such that $F$ denotes the objective function of problem \eqref{prob}, i.e., $F(u)\triangleq g(y)+ \sum_{i=1}^N  \Phi_i(x_i)$.
\begin{definition}
\label{def:optimal_set}
We denote the set of optimal primal-dual pairs for \eqref{prob} as $\chi^*$, i.e., $w^*=(\bx^*,y^*,\lambda^*)\in\chi^*$ if and only if $w^*$ is a saddle point of the Lagrangian function $\cL$ corresponding to \eqref{prob}. {Moreover, for $i\in\cN$, define $\chi_i^*$ as the projection of $\chi^*$ onto $x_i$-coordinate.}
\end{definition}
{We analyze 
SPG-ADMM, stated 
in~\eqref{GADM-stoc-extension}, under Assumptions~\ref{assumption-2} and \ref{assumption-1}; and specialize these results for PG-ADMM stated in \eqref{GADM-extension}.} The following assumption is made in the rest of the discussions.
\begin{assumption}\label{assumption-1} The optimal primal-dual pair set $\chi^*$ for \eqref{prob} is non-empty, i.e., there exists $(\bx^*,y^*,\lambda^*)\in\chi^*$ such that $\cL(\bx,y,\lambda^*)\geq\cL(\bx^*,y^*,\lambda^*)\geq \cL(\bx^*,y^*,\lambda)$ for all $\bx$, $y$, $\lambda$. \end{assumption}
According to the first-order optimality conditions for \eqref{prob}, solving \eqref{prob} is equivalent to finding $(\bx^*,y^*,\lambda^*)\in\chi^*$
such that the following holds: \vspace*{-1mm}
{\small
\begin{equation}\label{kkt}
\left\{
\begin{array}{l}
- A_i^\top\lambda_i^* - \nabla f_i(x_i^*) \in \partial\xi_i(x_i^*), \quad i\in\cN, \\
- \sum_{i=1}^N B_i^\top\lambda_i^* \in \partial g(y^*), \\
A_i x_i^* + B_i y^* - b_i = 0, \quad i\in\cN.
\end{array}
\right.
\end{equation}}%
\subsection{Convergence Rate of PG-ADMM and SPG-ADMM} 
\label{sec:deterministic}
Note PG-ADMM and SPG-ADMM produce the same iterate sequence with probability 1 when $\sigma=0$ in Definition~\ref{eq:def:SFO-II}. Therefore, we will first analyze SPG-ADMM, and then derive the bounds for PG-ADMM by sharpening the SPG-ADMM bounds for the case $\sigma=0$.
After some constant terms 
are discarded, SPG-ADMM for solving \eqref{prob} can be stated as: for $k\geq 0$, \vspace*{-1mm}
{\small
\begin{subequations}\label{spg-admm-update}
\begin{eqnarray}
x_i^{k+1} &\triangleq& \prox{{c_i^k}\xi_i} \Big( x_i^k - {c_i^k}\Big[G_i(x_i^k,\nu_i^k) + {\gamma_i} A_i^\top\big(A_i x_i^k + B_i y^k - b_i+{\frac{1}{\gamma_i}}\lambda_i^k\big)\Big]\Big), \quad i\in\cN, \label{update-x}\\
y^{k+1} &\triangleq& \argmin_y \ g(y) +\tfrac{1}{2}\sum\limits_{i=1}^N {\gamma_i}~\left\|  A_i x_i^{k+1} + B_i y - b_i + \frac{1}{\gamma_i}\lambda_i^k\right\|^2,  \label{update-y}\\
\lambda_i^{k+1} &\triangleq& \lambda_i^k +~ {\gamma_i} \left( A_i x_i^{k+1} + B_i y^{k+1} - b_i\right), \quad i\in\cN. \label{update-lambda}
\end{eqnarray}
\end{subequations}
}%
The first-order optimality conditions for \eqref{update-x}-\eqref{update-y} are given respectively as follows \vspace*{-1mm}
{\small
\begin{subequations}\label{alg-opt-conditions}
\begin{align}
\partial\xi_i(x_i^{k+1}) & \ni \frac{1}{c_i^k}\left(x_i^k - x_i^{k+1}\right)- \left[G_i(x_i^k,\nu_i^k)+A_i^\top \left(\lambda_i^k+\gamma_i \left(A_i x_i^k + B_i y^k - b_i\right)\right)\right], \quad i\in\cN \label{opt-x} \\
\partial g(y^{k+1}) &\ni -\sum\limits_{i=1}^N B_i^\top\left(\lambda_i^k + \gamma_i \left( A_i x_i^{k+1} + B_i y^{k+1} - b_i \right)\right). \label{opt-y}
\end{align}
\end{subequations}
}%
Since $\xi_i$ and $g$ are convex, 
using the subgradients above 
with \eqref{update-lambda}, we have \vspace*{-1mm}
{\small
\begin{align}
\xi_i(x_i)- \xi_i(x_i^{k+1}) &\geq \left( x_i-x_i^{k+1} \right)^\top\Bigg[ \frac{1}{c_i^k}\left(x_i^k - x_i^{k+1}\right) \nonumber\\ &\quad -\Bigg(G_i(x_i^k,\nu_i^k)+ A_i^\top\lambda_i^{k+1}+ \gamma_i A_i^\top \Big(A_i (x_i^k-x_i^{k+1}) + B_i (y^k-y^{k+1})\Big)\Bigg) \Bigg] , \label{opt-x-lambda} \\
g(y)-  g(y^{k+1})  &\geq  - \left( y - y^{k+1} \right)^\top \sum\limits_{i=1}^N B_i^\top\lambda_i^{k+1}. \label{opt-y-lambda}
\end{align}
}%
for any $x_i\in\reals^{n_i}$ and $y\in\reals^{n_y}$.
\begin{definition}
For some given $\{c_i^k\}_{k\geq 0}\subset\reals_{++}$, let $Q_i^k\triangleq I_{n_i}-\gamma_ic_i^kA_i^\top A_i$, and $\delta_i^{k}\triangleq G_i(x_i^k,\nu_i^k) - \grad_i f_i(x_i^k)$ denote the error in the noisy gradient generated by SFO for $i\in\cN$. Define $D \triangleq \max_{i \in \cN} \sup_{x_i, x_i' \in \dom(\xi_i)  }  \norm{x_i - x_i'}$ and {$D^*(\bx)\triangleq \max_{i \in \cN} \sup_{x'_i \in \chi^*_i  }  \norm{x'_i - x_i}$ for $\bx\in\reals^{n_x}$.}
\end{definition}
{We analyze the convergence of SPG-ADMM when either $D<+\infty$ or $D^*(\bx^0)<\infty$.} 
\begin{lemma}\label{technical-lemma}
Starting from given $w^0$ such that {$-\sum_{i=1}^NB_i^\top\lambda_i^0\in\partial g(y^0)$}, Let 
$\{w^k\}_{k\geq 1}$ be the SPG-ADMM iterate sequence generated as in \eqref{update-x}-\eqref{update-lambda}. 
If $\{c_i^k\}_{k\geq 0}$ is chosen such that $c_i^k\leq \left(L_i + {\gamma_i}\norm{A_i}^2+\mathds{1}(\delta_i^k)(1+\sqrt{k})\right)^{-1}$ for all $i\in\cN$ and $k\geq 0$, then for any $\lambda$ and $u^*=(\bx^*,y^*)\in\chi^*$, the following inequality holds for all $k\geq 0$,
{\small
\begin{align}
&F(u^*) - F(u^{k+1}){+\sum_{i=1}^N\fprod{\lambda_i,~b_i-A_ix_i^{k+1}-B_iy^{k+1}}}\nonumber\\
&+ \frac{1}{2}\sum_{i=1}^N \Bigg[ \frac{1}{c_i^k}\left\| x_i^*-x_i^k \right\|_{Q_i^k}^2 - \frac{1}{c_i^k}\left\| x_i^*-x_i^{k+1} \right\|_{Q_i^k}^2+ \gamma_i \left\| B_i (y^k - y^*) \right\|^2 - \gamma_i \left\|B_i (y^{k+1} - y^*) \right\|^2\nonumber \\
& + \frac{1}{\gamma_i}\left\| \lambda_i - \lambda_i^k\right\|^2 - \frac{1}{\gamma_i}\left\| \lambda_i - \lambda_i^{k+1} \right\|^2 + 2(x^*_i-x_i^k)^\top\delta_i^k+\frac{\norm{\delta_i^k}^2}{1+\sqrt{k}}\Bigg] \nonumber \\
\geq\ & \frac{1}{2}\sum_{i=1}^N \Bigg[ {\norm{x_i^{k+1}-x_i^k}_{\bar{Q}_i^k}^2+ \gamma_i\norm{B_i(y^{k+1}-y^{k})}^2+\frac{1}{\gamma_i}\norm{\lambda_i^{k+1}-\lambda_i^{k}}^2\Bigg]}, \label{inequality-lemma}
\end{align}
}%
where 
$\mathds{1}(\delta_i^k)$ is $1$ if $\norm{\delta_i^k}>0$, and equal to $0$ otherwise, i.e., when $\delta_i^k=0$; and $\bar{Q}_i^k\triangleq\frac{1}{c_i^k}Q_i^k-\big(L_i+\mathds{1}(\delta_i^k)(1+\sqrt{k})\big)$ for $k\geq 0$ and $i\in\cN$.
\end{lemma}
\begin{proof}
Note that \eqref{opt-x-lambda} implies \vspace*{-3mm}
{\small
\begin{align}
& \xi_i(x_i) - \xi_i(x_i^{k+1}) \nonumber\\
& \geq \frac{1}{c_i^k}\left( x_i-x_i^{k+1}\right)^\top Q_i^k\left(x_i^k - x_i^{k+1}\right) - \left( x_i-x_i^{k+1}\right)^\top \Big( G_i(x_i^k,\nu_i^k) + A_i^\top\lambda_i^{k+1} + \gamma_i A_i^\top B_i (y^k-y^{k+1}) \Big) \nonumber \\
& = \frac{1}{2c_i^k}\left\| x_i-x_i^{k+1} \right\|_{Q_i^k}^2 + \frac{1}{2c_i^k}\left\| x_i^k - x_i^{k+1} \right\|_{Q_i^k}^2 - \frac{1}{2c_i^k}\left\| x_i-x_i^k \right\|_{Q_i^k}^2 -\left( x_i-x_i^{k+1} \right)^\top{(\nabla f_i(x_i^k)+\delta_i^k)}
\nonumber \\
& \quad-\left( x_i-x_i^{k+1} \right)^\top A_i^\top\lambda_i^{k+1} - \gamma_i \left( A_i x_i - A_i x_i^{k+1} \right)^\top \left( B_i y^k - B_i y^{k+1}\right), \label{eq:xi-subgradient-ineq}
\end{align}
}%
where the equality follows from identity \eqref{identity-3} with $Q=Q_i^k$, and from the definition of $\delta_i^k$. Since each $f_i$ is convex with Lipschitz continuous $\grad f_i$ for all $x_i\in\dom\xi_i$, we have $0 \leq f_i(x_i)-f_i(x_i^k)-\grad f_i(x_i^k)^\top(x_i-x_i^k)\leq \frac{L_i}{2}\norm{x_i-x_i^k}^2$; hence
{\small
\begin{align}
 \Big( x_i -x_i^{k+1} \Big)^\top\nabla f_i(x_i^k) & =  \left( x_i-x_i^k \right)^\top\nabla f_i(x_i^k) + \left( x_i^k-x_i^{k+1} \right)^\top\nabla f_i(x_i^k) \nonumber \\
& \leq  f_i(x_i) - f_i(x_i^k) + f_i(x_i^k) - f_i(x_i^{k+1}) + \frac{L_i}{2}\left\| x_i^k-x_i^{k+1} \right\|^2 \nonumber\\
& =  f_i(x_i) - f_i(x_i^{k+1}) + \frac{L_i}{2}\left\| x_i^k-x_i^{k+1} \right\|^2. \label{eq:grad-f-ineq}
\end{align}
}%
Moreover, it also trivially follows from the definition of $\mathds{1}(\delta_i^k)$ that
{\small
\begin{align}
\left( x_i-x_i^{k+1} \right)^\top\delta_i^k \leq \left( x_i-x_i^k \right)^\top\delta_i^k+\frac{1}{2(1+\sqrt{k})}\norm{\delta_i^k}^2 +\frac{\mathds{1}(\delta_i^k)(1+\sqrt{k})}{2}\norm{x_i^k-x_i^{k+1}}^2. \label{eq:delta-ineq}
\end{align}
}%
Finally, applying identity \eqref{identity-4} with $Q=I_{m_i}$, $v_1=A_ix_i-b_i$, $v_2=A_ix^{k+1}_i-b_i$, $v_3=-B_iy^{k+1}$, and $v_4=-B_iy^k$, we have \vspace*{-3mm}
{\small
\begin{align}
\big(  A_i &x_i - A_i x_i^{k+1} \big)^\top \left( B_i y^k - B_i y^{k+1}\right) \label{eq:AB-ineq}\\
=  & \; \frac{1}{2}\left(\| A_i x_i + B_i y^k - b_i\|^{2}- \| A_i x_i + B_i y^{k+1} - b_i\|^{2}\right)+ \frac{1}{2}\left(\| A_i x_i^{k+1} + B_i y^{k+1} - b_i \|^{2} - \|  A_i x_i^{k+1} + B_i y^k - b_i\|^{2}\right), \nonumber \\
= & \; \frac{1}{2}\left(\| A_i x_i + B_i y^k - b_i\|^{2} - \| A_i x_i + B_i y^{k+1} - b_i\|^{2}\right)-\frac{1}{2}{\Big(\norm{B_i(y^k-y^{k+1})}^2+\frac{2}{\gamma_i}(\lambda_i^{k+1}-\lambda_i^k)^\top B_i(y^k-y^{k+1})\Big).} \nonumber
\end{align}
}%
Therefore, using \eqref{eq:grad-f-ineq}, \eqref{eq:delta-ineq} and \eqref{eq:AB-ineq} within \eqref{eq:xi-subgradient-ineq}, it follows that
{\small
\begin{align}
\xi_i(&x_i) - \xi_i(x_i^{k+1}) + f_i(x_i) - f_i(x_i^{k+1}) + \left( x_i-x_i^{k+1} \right)^\top A_i^\top\lambda_i^{k+1} +(x_i-x_i^k)^\top\delta_i^k+\frac{\norm{\delta_i^k}^2}{2(1+\sqrt{k})} \label{eq:subgrad-ineq} \\
\geq  &\frac{1}{2c_i^k}\left\| x_i-x_i^{k+1} \right\|_{Q_i^k}^2 - \frac{1}{2c_i^k}\left\| x_i-x_i^k \right\|_{Q_i^k}^2+ \frac{1}{2c_i^k}\left\| x_i^k - x_i^{k+1} \right\|_{Q_i^k}^2 -\frac{L_i+\mathds{1}(\delta_i^k)(1+\sqrt{k})}{2}\left\| x_i^k-x_i^{k+1} \right\|^2\nonumber\\
& + \frac{\gamma_i}{2}\| A_i x_i + B_i y^{k+1} - b_i\|^{2}- \frac{\gamma_i}{2}\| A_i x_i + B_i y^k - b_i\|^{2} + {\frac{\gamma_i}{2}\norm{B_i(y^k-y^{k+1})}^2+(\lambda_i^{k+1}-\lambda_i^k)^\top B_i(y^k-y^{k+1})}, \nonumber
\end{align}
}%
{Since
$0<c_i^k\leq \big(L_i+ \gamma_i\sigma_{\max}^2(A_i)+\mathds{1}(\delta_i^k)(1+\sqrt{k})\big)^{-1}$, by definition $\bar{Q}_i^k\succeq\mathbf{0}$ for all $k\geq 0$ and $i\in\cN$.}
{Moreover, for all $k\geq 0$, we have $\sum_{i=1}^N\fprod{B_i(y^k-y^{k+1}),~\lambda_i^{k+1}-\lambda_i^k}\geq 0$ since $-\sum_{i=1}^NB_i^\top\lambda_i^k\in\partial g(y^k)$.}
Therefore, setting $u=u^{*}$, combining \eqref{update-lambda} and \eqref{opt-y-lambda} 
with the inequality in \eqref{eq:subgrad-ineq}, and using the definition of $\bar{Q}_i^k$ implies that for any $\lambda\in\reals^m$,
{\small
\begin{align}
&F(u^*) - F(u^{k+1})\nonumber\\
& + \sum_{i=1}^N(x_i^*-x_i^{k+1})^\top A_i^\top\lambda_i^{k+1}+(y^*-y^{k+1})^\top \sum_{i=1}^N B_i^\top\lambda_i^{k+1} +\sum_{i=1}^N(\lambda_i-\lambda_i^{k+1})^\top(b_i-A_ix_i^{k+1}-B_iy^{k+1}) \nonumber \\
& + \frac{1}{2}\sum\limits_{i=1}^N \Bigg[ - \frac{2}{\gamma_i}\left( \lambda_i - \lambda_i^{k+1}\right)^\top \left( \lambda_i^{k} - \lambda_i^{k+1} \right)+ \frac{1}{c_i^k}\left\| x_i^*-x_i^k \right\|_{Q_i^k}^2 - \frac{1}{c_i^k}\left\| x_i^*-x_i^{k+1} \right\|_{Q_i^k}^2 \nonumber \\
& + \gamma_i \left\| A_i x_i^* + B_i y^k - b_i \right\|^2 - \gamma_i \left\| A_i x_i^* + B_i y^{k+1} - b_i \right\|^2+2(x^*_i-x_i^k)^\top\delta_i^k+\frac{\norm{\delta_i^k}^2}{1+\sqrt{k}}\Bigg] \nonumber\\
\geq\ & \frac{1}{2}\sum_{i=1}^N \Bigg[{\norm{x_i^{k+1}-x_i^k}_{\bar{Q}_i^k}^2+\gamma_i\norm{B_i(y^{k+1}-y^k)}^2\Bigg]}. \label{inequality-main}
\end{align}
}%
Using \eqref{identity-3} with $Q=I_{m_i}$, we have
{\small
\begin{align*}
\frac{1}{2\gamma_i} \Big \|\lambda_i^{k+1} - \lambda_i^k \Big \|^2 - \frac{1}{\gamma_i}\left(\lambda_i-\lambda_i^{k+1}\right)^{\top}\left( \lambda_i^{k} - \lambda_i^{k+1} \right) = \frac{1}{2\gamma_i}\left(\left\|\lambda_i-\lambda_i^k \right\|^2 - \left\|\lambda_i-\lambda_i^{k+1}\right\|^2\right).
\end{align*}
}%
{Thus, adding $\sum_{i=1}^N\frac{1}{2\gamma_i}\left\| \lambda_i^{k+1}-\lambda_i^k \right\|^2$ to both side of \eqref{inequality-main}, and using the above equality}, we obtain the desired inequality in~\eqref{inequality-lemma}. Indeed, since $u^*$ is a solution to \eqref{prob}, $A_ix_i^*+B_iy^*=b_i$ for $i\in\cN$. Hence, for any $\lambda$ and $\hat{w}$ \vspace*{-2mm}
{\small
\[\sum_{i=1}^N\lambda_i^\top(b_i-A_i\hat{x}_i-B_i\hat{y})=\sum_{i=1}^N(x^*_i-\hat{x}_i)^\top A_i^\top\hat{\lambda}_i+ (y^*-\hat{y})^\top\sum_{i=1}^NB_i^\top\hat{\lambda}_i
+\sum_{i=1}^N(\lambda_i-\hat{\lambda}_i)^\top(b_i-A_i\hat{x}_i-B_i\hat{y}).\]
}%
We use the above equality with $\hat{w}=w^{k+1}$ to simplify \eqref{inequality-main}.
\end{proof}
\begin{theorem}
\label{lem:ergodic}
Under Assumptions~\ref{assumption-0},~\ref{assumption-2} and~\ref{assumption-1}, let $\{w^k\}_{k\geq 1}$ be the SPG-ADMM iterate sequence generated as in \eqref{update-x}-\eqref{update-lambda} starting from given
$w^0=\left(\bx^0,y^0,\lambda^0\right)$ {as in Lemma~\ref{technical-lemma}}.
Define $\bar{u}^t = (\bar{\bx}^t, \bar{y}^t)$ as
$\bar{\bx}^t = \frac{1}{t}\sum_{k=1}^{t}\bx^{k}$, and $\bar{y}^t = \frac{1}{t}\sum_{k=1}^{t} y^{k}$ for $t\geq 1$.
Fix arbitrary $\lambda$ and $u^*=(\bx^*,y^*)\in\chi^*$.

Suppose $\sigma>0$ and $D<\infty$. For $i\in\cN$, when $\{c_i^k\}_{k\geq 0}$ is chosen such that $\frac{1}{c_i^k}=\frac{1}{c_i}+\sqrt{k}$ and $c_i=(L_i + \gamma_i\norm{A_i}^2+1)^{-1}$, then the following bounds hold for $\bar{D}=D$ and for all $t\geq 1$:\vspace*{-2mm}
{\small
\begin{eqnarray}
\lefteqn{{\mathbb{E}\Big[ F\left(\bar{u}^{t}\right) - F(u^*) + \sum\limits_{i=1}^N \lambda_i^\top \left( A_i\bar{x}_i^{t} + B_i\bar{y}^{t} - b_i \right)\Big]}} \label{eq:SPG-key-ineq}\\
&\leq &\frac{N(\bar{D}^2+2\sigma^2)}{2\sqrt{t}}+\frac{1}{2t} \sum\limits_{i=1}^N \left( \frac{1}{\gamma_i}\left\| \lambda_i - \lambda_i^0 \right\|^2 + \frac{1}{c_i}\left\| x_i^*-x_i^0 \right\|_{I- \gamma_i c_i A_i^\top A_i}^2 \right)+\frac{1}{2t} \left\| y^* - y^0 \right\|_{\sum_{i=1}^N\gamma_i B_i^\top B_i}^2. \nonumber
\end{eqnarray}
}%
Suppose $\sigma=0$, i.e., $G_i(x_i)=\grad f_i(x_i)$ w.p.1 for any $x_i$ and $i\in\cN$. When $c_i^k=c_i$ for all $k\geq 0$ for some $c_i\in(0,\frac{1}{L_i + \gamma_i\norm{A_i}^2}]$, the following bound holds w.p.1 for all $t\geq 1$, \vspace*{-2mm}
{\small
\begin{eqnarray}
\lefteqn{F\left(\bar{u}^{t}\right) - F(u^*) + \sum\limits_{i=1}^N \lambda_i^\top \left( A_i\bar{x}_i^{t} + B_i\bar{y}^{t} - b_i \right)} \label{eq:PG-key-ineq} \\
& \leq & \frac{1}{2t} \sum\limits_{i=1}^N \Big( \frac{1}{\gamma_i}\left\| \lambda_i - \lambda_i^0 \right\|^2 + \frac{1}{c_i}\left\| x_i^*-x_i^0 \right\|_{I_n- \gamma_i c_i A_i^\top A_i}^2\Big) + \frac{1}{2t} \left\| y^* - y^0 \right\|_{\sum_{i=1}^N\gamma_i B_i^\top B_i}^2. \nonumber
\end{eqnarray}}%
\end{theorem}
\begin{proof}
Invoking the convexity of $F(\cdot)$ justifies the first inequality below: \vspace*{-1mm}
{\small
\begin{eqnarray*}
\lefteqn{F(\bar{u}^{t})-F(u^*)+\sum\limits_{i\in\cN} \lambda_i^\top \left( A_i\bar{x}_i^{t} + B_i\bar{y}^{t} - b_i \right)} \\
& \leq & -\frac{1}{t}\sum\limits_{k=0}^{t-1}\left[F(u^*) - F(u^{k+1}) + \sum\limits_{i=1}^N \lambda_i^\top \left( b_i-A_ix_i^{k+1}- B_iy^{k+1} \right)\right] \nonumber\\
& \leq & \frac{1}{2t}\sum\limits_{k=0}^{t-1}\sum\limits_{i=1}^N \left[ \frac{1}{\gamma_i}\left\| \lambda_i - \lambda_i^{k}\right\|^2 - \frac{1}{\gamma_i}\left\| \lambda_i - \lambda_i^{k+1} \right\|^2 + \frac{1}{c_i^k}\left\| x_i^*-x_i^{k} \right\|_{Q_i^k}^2 - \frac{1}{c_i^k}\left\| x_i^*-x_i^{k+1} \right\|_{Q_i^k}^2 \right. \nonumber\\
& & \qquad \qquad \qquad \left.+ \gamma_i \left\| B_i (y^{k}-y^*) \right\|^2 - \gamma_i \left\| B_i (y^{k+1}-y^*) \right\|^2+2(x_i^*-x_i^k)^\top\delta_i^k+\frac{\norm{\delta_i^k}^2}{1+\sqrt{k}} \right]\triangleq\Gamma^t, \nonumber
\end{eqnarray*}
}%
where the second inequality follows from Lemma~\ref{technical-lemma} and the fact that $\bar{Q}_i^k\succeq\mathbf{0}$ for $k\geq 0$.

First, consider the PG-ADMM iterate sequence generated using $G_i(x_i^k,\nu_i^k)=\grad f_i(x_i^k)$, i.e., $\delta_i^k=0$ for all $i\in\cN$ and $k\geq 0$. In this setting, according to Lemma~\ref{technical-lemma}, we are allowed to fix a constant step size $c_i$ for each $i\in\cN$. Indeed, $c_i^k=c_i\triangleq\left(L_i + \gamma_i\norm{A_i}^2\right)^{-1}$ for $k\geq 1$. Hence, one obtains the desired result in \eqref{eq:PG-key-ineq} by showing \vspace*{-2mm}
{\small
\begin{align*}
\Gamma^t \leq \frac{1}{2t} \sum\limits_{i=1}^N \left( \frac{1}{\gamma_i}\left\| \lambda_i - \lambda_i^0 \right\|^2 + \frac{1}{c_i}\left\| x_i^*-x_i^0 \right\|_{I- \gamma_i c_i A_i^\top A_i}^2 + \gamma_i \left\| y^*-y^0\right\|_{B_i^\top B_i}^2 \right),
\end{align*}
}%
which follows from dropping the non-positive terms after the telescoping sum in the definition of $\Gamma^t$ is evaluated, and using the fact that $A_i x_i^*+B_i y^*=b_i$ for all $i\in\cN$. 

Next, consider the SPG-ADMM iterate sequence for which $\delta_i^k>0$ with probability 1 for all $i\in\cN$ and $k\geq 0$. In this case, according to Lemma~\ref{technical-lemma}, if one sets $c_i^k=\left(L_i + \gamma_i\norm{A_i}^2+1+\sqrt{k}\right)^{-1}$ for all $i\in\cN$ and $k\geq 0$, then $F(\bar{u}^{t})-F(u^*)+\sum\limits_{i\in\cN} \lambda_i^\top \left( A_i\bar{x}_i^{t} + B_i\bar{y}^{t} - b_i \right)\leq\Gamma^t$ holds for all $t\geq 1$. Moreover, we also have \vspace*{-1mm}
{\small
\begin{eqnarray}
\lefteqn{\frac{1}{2t}\sum_{k=0}^{t-1}\sum_{i=1}^N\frac{1}{c_i^k}\left\| x_i^*-x_i^{k} \right\|_{Q_i^k}^2 - \frac{1}{c_i^k}\left\| x_i^*-x_i^{k+1} \right\|_{Q_i^k}^2}\nonumber\\
& = & \frac{1}{2t}\left[\sum_{i=1}^N\left(\frac{1}{c_i^0}\left\| x_i^*-x_i^0 \right\|_{I- \gamma_i c_i^0 A_i^\top A_i}^2-\frac{1}{c_i^{t-1}}\left\| x_i^*-x_i^{t} \right\|_{I-\gamma_i c_i^{t-1} A_i^\top A_i}^2+\sum_{k=0}^{{t-2}}\left(\sqrt{k+1}-\sqrt{k}\right)\left\| x_i^*-x_i^{k+1} \right\|^2\right)\right]\nonumber\\
&\leq&\frac{ND^2}{2\sqrt{t}}+\frac{1}{2t}\sum_{i=1}^N\frac{1}{c_i^0}\left\| x_i^*-x_i^0 \right\|_{I- \gamma_i c_i^0 A_i^\top A_i}^2. \label{eq:Gamma_ineq}
\end{eqnarray}
}%
Hence, it follows that \vspace*{-1mm}
{\small
\begin{align}
\Gamma^t \leq &\frac{ND^2}{2\sqrt{t}}+\frac{1}{2t} \sum\limits_{i=1}^N \left( \frac{1}{\gamma_i}\left\| \lambda_i - \lambda_i^0 \right\|^2 + \frac{1}{c_i^0}\left\| x_i^*-x_i^0 \right\|_{I- \gamma_i c_i^0 A_i^\top A_i}^2 + \gamma_i \left\| y^*-y^0\right\|_{B_i^\top B_i}^2\right) \nonumber\\
&\quad +\frac{1}{2t} \sum\limits_{i=1}^N\sum_{k=0}^{t-1}\left(2(x_i-x_i^k)^\top\delta_i^k+\frac{\norm{\delta_i^k}^2}{1+\sqrt{k}}\right). \label{eq:gamma_k}
\end{align}
}%
Let $\nu^\ell=[\nu_i^\ell]_{i\in\cN}$ for $\ell\geq 0$. Note from \eqref{GADM-stoc-extension} it follows that $(\bx^{k},y^{k},\lambda^{k})$ are random variables depending only on $\Upsilon^k\triangleq\{\nu^\ell\}_{\ell=0}^{k-1}$ for all $k\geq 1$; hence, $\mathbb{E}_{\nu^k}[\delta_i^k|\Upsilon^{k}]=\mathbf{0}$ and $\mathbb{E}_{\nu^k}\left[\norm{\delta_i^k}^2|\Upsilon^{k}\right]\leq\sigma^2$ because $\delta_i^{k}\triangleq G_i(x_i^k,\nu_i^k) - \grad_i f_i(x_i^k)$. Therefore, for all $k\geq 0$, we have \vspace*{-1mm}
\begin{align*}
\mathbb{E}_{\Upsilon^{k+1}}\left[(x_i^k)^\top\delta_i^k\right]&=\mathbb{E}_{\Upsilon^{k}}\left[\mathbb{E}_{\nu^k}\left[(x_i^k)^\top\delta_i^k|\Upsilon^{k}\right]\right]
=\mathbb{E}_{\Upsilon^{k}}\left[(x_i^k)^\top\mathbb{E}_{\nu^k}[\delta_i^k|\Upsilon^{k}]\right]=0,\\
\mathbb{E}_{\Upsilon^{k+1}}\left[\norm{\delta_i^k}^2\right]&=\mathbb{E}_{\Upsilon^{k}}\left[\mathbb{E}_{\nu^k}\left[\norm{\delta_i^k}^2|\Upsilon^{k}\right]\right]
\leq\sigma^2.
\end{align*}
We obtain the desired result by taking the expectation of both sides in \eqref{eq:gamma_k}
and using the inequality $\sum_{k=0}^{t-1}\frac{1}{1+\sqrt{k}}\leq\int_{0}^{t-1}\frac{1}{1+\sqrt{s}}ds\leq\sqrt{t}$.
\end{proof}
\vspace*{-2mm}
\begin{corollary}
\label{cor:constant-step}
{Suppose $\sigma>0$ and $D=\infty$. Assume that $D^*(\bx^0)<\infty$. For any $t\geq 1$, let $\{c_i^k\}_{0\leq k\leq t}$ be chosen such that $\frac{1}{c_i^k}=\frac{1}{c_i}+\sqrt{t}$ and $c_i=(L_i + \gamma_i\norm{A_i}^2+1)^{-1}$ for $i\in\cN$, then \eqref{eq:SPG-key-ineq} holds for $\bar{D}=D^*(\bx^0)$.}
\end{corollary}
Now we are ready to prove the $O(1/t)$ and $O(1/\sqrt{t})$ convergence rates of PG-ADMM stated in \eqref{GADM-extension} and SPG-ADMM stated in \eqref{GADM-stoc-extension} in an ergodic sense for solving \eqref{prob}.
\begin{theorem}\label{thm-ergodic}
Under the same settings of Theorem~\ref{lem:ergodic}, fix an arbitrary point $(u^*,\lambda_1^*,\ldots,\lambda_N^*) \in\chi^*$.

Suppose $\sigma=0$, i.e., $G_i(x_i)=\grad f_i(x_i)$ w.p.1 for any $x_i$ and $i\in\cN$. 
When $c_i^k=c_i$ for all $k\geq 0$ for some $c_i\in\big(0,~{\frac{1}{L_i + \gamma_i\norm{A_i}^2}}\big)$, the following bounds hold for all $t\geq 1$:
\vspace*{-1mm}
{\small
\begin{equation*}
|F(\bar{u}^t)-F(u^*)| \leq C(c_1,\ldots,c_N)/t,
\quad \mbox{ and } \quad
\sum\limits_{i=1}^N \norm{\lambda_i^*}\left\|  A_i\bar{x}_i^t + B_i\bar{y}^t - b_i \right\|\leq  C(c_1,\ldots,c_N)/t, \vspace*{-3mm}
\end{equation*}
}%
where $C(c_1,\ldots,c_N)\triangleq\sum\limits_{i=1}^N \Big( \frac{1}{\gamma_i}\left(4\|\lambda_i^*\|^2 + \left\| \lambda_i^0\right\|^2 \right) + \frac{1}{2c_i}\left\| x_i^*-x_i^0 \right\|_{Q_i}^2 \Big) + \frac{1}{2}\left\| y^*-y^0\right\|_{\sum_{i=1}^N \gamma_i B_i^\top B_i}^2$.
and $Q_i\triangleq I-c_i\gamma_iA_i^\top A_i$. {Moreover, both $\{w^k\}_{k\geq 0}$ and $\{\bar{w}^k\}_{k\geq 0}$ converge to a primal-dual optimal point if 
$B\triangleq[B_i]_{i\in\cN}$, formed by vertically concatinating $\{B_i\}_{i\in\cN}$, has full column rank.}

Suppose $\sigma>0$. Fix $c_i=(L_i + \gamma_i\norm{A_i}^2+1)^{-1}$ for $i\in\cN$ and set $C\triangleq C(c_1,\ldots,c_N)$. Let $C'\triangleq N(D^2/2+\sigma^2)$, and $D\triangleq \max_{i \in \cN} \sup_{x, x' \in \dom(\xi_i)}\norm{x - x'}$. When $\{c_i^k\}_k$ sequence is chosen as $c_i^k=\left(\frac{1}{c_i}+\sqrt{k}\right)^{-1}$ for $i\in\cN$, the following bounds hold for all $t\geq 1$: \vspace*{-1mm}
{\small
\begin{equation*}
\mathbb{E}[|F(\bar{u}^t)-F(u^*)|] \leq \frac{C'}{\sqrt{t}}+\frac{C}{t}, \quad \mbox{ and } \quad \sum\limits_{i=1}^N \norm{\lambda_i^*}\mathbb{E}\Big[\left\|  A_i\bar{x}_i^t + B_i\bar{y}^t - b_i \right\|\Big]\leq  \frac{C'}{\sqrt{t}}+\frac{C}{t}.
\end{equation*}}%
\end{theorem}
\begin{proof}
Suppose $\sigma=0$. Consider $\bar{Q}_i^k$ defined in Lemma~\ref{technical-lemma}; since $c_i^k=c_i\in(0,(L_i+\gamma_i\norm{A_i}^2)^{-1})$, we have $\bar{Q}_i\triangleq\bar{Q}_i^k=\frac{1}{c_i}Q_i-L_i\succ \mathbf{0}$ for $k\geq 0$ and $i\in\cN$. For $i\in\cN$, let $a_i^k\triangleq \frac{1}{c_i}\left\|x_i^k-x_i^* \right\|_{Q_i}^2+\gamma_i \left\| B_i (y^k - y^*) \right\|^2+\frac{1}{\gamma_i}\left\| \lambda_i^k-\lambda^*_i\right\|^2$ and $b_i^k\triangleq \norm{x_i^{k+1}-x_i^k}_{\bar{Q}_i}^2+ \gamma_i\norm{B_i(y^{k+1}-y^k)}^2+\frac{1}{\gamma_i}\norm{\lambda_i^{k+1}-\lambda_i^{k}}^2$; define $a^k\triangleq \sum_{i=1}^N a_i^k$ and $b^k\triangleq \sum_{i=1}^N b_i^k$ for $k\geq 0$. Hence, Lemma~\ref{technical-lemma} implies that $a^{k+1}+b^k\leq a^k$ for $k\geq 0$, where we used $F(u^*) \leq F(u^{k+1}){+\sum_{i=1}^N\fprod{\lambda^*_i,~A_ix_i^{k+1}+B_iy^{k+1}-b_i}}$ because $\cL(\bx^{k+1},y^{k+1},\lambda^*)\geq\cL(\bx^*,y^*,\lambda^*)$. Therefore, $\lim_k a^k$ exists and $\sum_{k=0}^\infty b^k<\infty$, which implies that $\{a_i^k\}_{k\geq 0}$ is bounded for all $i$. Moreover $Q_i\succ\mathbf{0}$ implies that $\{\lambda_i^k\}$, $\{x_i^k\}$ and $\{B_iy^k\}$ are all bounded sequences for $i\in\cN$. If $B$ has full column rank, $\{y^k\}$ is bounded. Hence, there exists $\{k_n\}_{n\geq 0}$ such that $\lim_{n}\lambda_i^{k_n}\triangleq\bar{\lambda}$, $\lim_{n}x_i^{k_n}\triangleq\bar{x}$ for all $i$ and $\lim_{n}y^{k_n}\triangleq\bar{y}$ exist. Moreover, since $\sum_{k=1}^\infty b^k<\infty$, we have $\sum_{k=1}^\infty b_i^k<\infty$ for $i\in\cN$; thus, $\lim_k\norm{x_i^{k+1}-x_i^k}=\lim_k\norm{\lambda_i^{k+1}-\lambda_i^k}=0$ for all $i$, and $\lim_k\norm{y^{k+1}-y^k}=0$. This implies that $\lim_{n}\lambda_i^{k_n\pm 1}=\bar{\lambda}$, $\lim_{n}x_i^{k_n\pm 1}=\bar{x}$ for $i\in\cN$ and $\lim_{n}y^{k_n\pm 1}=\bar{y}$. Note $A_i\bar{x}_i+B\bar{y}-b_i=\lim_{n}A_ix_i^{k_n}+By^{k_n}-b_i=\frac{1}{\gamma_i}\lim_n \lambda_i^{k_n}-\lambda_i^{k_n-1}=0$. Now, consider the subsequential limit of both sides of the relations in \eqref{alg-opt-conditions} along $k_n$. Since $g$ and $\xi_i$ for all $i$ are proper, closed convex functions, it follows from Theorem~24.4 in~\cite{rockafellar2015convex} that $\mathbf{0}\in \partial\xi_i(\bar{x}_i)+\grad f_i(\bar{x}_i)+A_i^\top\bar{\lambda}$ for all $i$, and $\mathbf{0}\in\partial g(\bar{y})+\sum_{i=1}^NB_i^\top \bar{\lambda}$. Thus $\bar{w}=(\bar{x},\bar{y},\bar{\lambda})$ is a primal-dual point for \eqref{prob}. Now define $r_i^k\triangleq \frac{1}{\gamma_i}\left\| \bar{\lambda}_i - \lambda_i^k\right\|^2+\frac{1}{c_i}\left\| \bar{x}_i-x_i^k \right\|_{Q_i}^2+\gamma_i \left\| B_i (y^k - \bar{y}) \right\|^2$ and $r^k=\sum_{i=1}^Nr_i^k$ for $k\geq 0$. Thus $\lim_k r^k=\lim_n r^{k_n}=0$ and we have $\lim x_i^k=\bar{x}$ (since $\bar{Q}\succ\mathbf{0}$) and $\lim \lambda_i^k=\bar{\lambda}$ for all $i$, and $\lim_k y^k=\bar{y}$.

In the rest, we prove the rate result. Since any $(u^*,\lambda^*) \in\chi^*$ is a saddle point of $\cL$, it satisfies $\cL(u,\lambda^*)\geq\cL(u^*,\lambda^*)$ for any $u=(\bx,y)$, which can be easily checked using the optimality condition \eqref{kkt}. Thus, setting $u=\bar{u}^t=(\bar{\bx}^t,\bar{y}^t)$ and using the fact that $A_ix_i^* + B_iy^*= b_i$ for $i\in\cN$
we obtain \vspace*{-1mm}
{\small
\begin{align}
\label{eq:erg_nonpos}
F(\bar{u}^t)-F(u^*) + \sum\limits_{i=1}^N 2\norm{\lambda_i^*} \left\| A_i\bar{x}^t_i + B_i\bar{y}^t - b_i \right\|\geq F(\bar{u}^t)-F(u^*)+\sum_{i=1}^N {\lambda_i^*}^\top \left( A_i\bar{x}^t_i + B_i\bar{y}^t - b_i \right) \geq 0.
\end{align}
}%

Let $\rho_i\triangleq2 \|\lambda_i^*\|$ and $\bar{\tau}_i \triangleq A_i \bar{x}_i^t + B_i \bar{y}^t - b_i$ for $i=1,2,\ldots,N$. Note \eqref{eq:PG-key-ineq} holds for any $\lambda=[\lambda_i]_{i\in\cN}$; hence, by setting $\lambda_i = \rho_i \bar{\tau}_i/\norm{\bar{\tau}_i}$ in \eqref{eq:PG-key-ineq}, noting that $\|\lambda_i\| = \rho_i$, and using $\norm{\lambda_i-\lambda_i^0}^2/2\leq\norm{\lambda_i}^2+\norm{\lambda_i^0}^2$, we obtain
$F(\bar{u}^t)-F(u^*) + \sum\limits_{i=1}^N \rho_i \left\|  A_i\bar{x}_i^t + B_i\bar{y}^t -b_i \right\| \leq C(c_1,\ldots,c_N)/t$.
From \eqref{eq:erg_nonpos} and the above inequality, it follows that \vspace*{-2mm}
{\small
\begin{align*}
- \sum_{i=1}^N\| \lambda_i^*\| \|\bar{\tau}_i\| &\leq F(\bar{u}^t) - F(u^*) \leq C(c_1,\ldots,c_N)/t - \sum_{i=1}^N \rho_i \|\bar{\tau}_i\|.
\end{align*}
}%
Hence, since $\rho_i\triangleq 2 \norm{\lambda_i^*}$, 
it follows that $\sum_{i=1}^N  \| \lambda_i^*\|   \|\bar{\tau}_i\| \leq \frac{C(c_1,\ldots,c_N)}{t}$.
Therefore, one can conclude that
$\max\{| F(\bar{u}^t) - F(u^*) |,\ \sum_{i=1}^N  \| \lambda_i^*\|   \| A_i \bar{x}_i^t + B_i \bar{y}^t - b_i \|\} \leq \frac{C(c_1,\ldots,c_N)}{t}$.

Using similar arguments as above one can show $\cO(1/\sqrt{t})$ rate for the SPG-ADMM. Indeed, observe that \eqref{eq:SPG-key-ineq} holds for any $\lambda=[\lambda_i]_{i\in\cN}$; hence, in \eqref{eq:SPG-key-ineq} by setting $\lambda_i = \rho_i \bar{\tau}_i/\norm{\bar{\tau}_i}$ as a random vector, noting that $\mathbb{E}[\|\lambda_i\|] = \rho_i$, and using $\norm{\lambda_i-\lambda_i^0}^2/2\leq\norm{\lambda_i}^2+\norm{\lambda_i^0}^2$, we obtain the desired results similarly as above. 
\vspace*{-3mm}
\end{proof}
\subsection{Connections to the existing work}
There is a strong connection between PG-ADMM and the primal-dual algorithms~(PDA) in~\cite{chambolle2015ergodic,condat2013primal} proposed for solving saddle point problems. Let $\Phi=\xi+f$ as in~\eqref{eq:F_i} such that $\grad f$ is Lipschitz with constant $L$; after fixing stepsize $c>0$ and penalty parameter $\gamma>0$, implementing PG-ADMM on $\min_{x,y}\{\Phi(x)+g(y):\ A x- y=\mathbf{0}\}$ generates the following iterate sequence:
{\small
\begin{subequations}\label{eq:pgadmm}
\begin{align}
x^{k+1}&\gets\prox{c\xi}\Big(x^k-c\Big[\grad f(x^k)+A^\top\big(\lambda^k+\gamma~(A x^k-y^k)\big)\Big]\Big), \label{eq:pgadmm_x1}\\
y^{k+1}&\gets\argmin g(y)+\tfrac{\gamma}{2}\norm{A x^{k+1}-y+\gamma^{-1}\lambda^k}^2, \label{eq:pgadmm_w1}\\
\lambda^{k+1}&\gets\lambda^k+\gamma~(A x^{k+1}- y^{k+1}). \label{eq:pgadmm_y1}
\end{align}
\end{subequations}}%
For PG-ADMM iterate sequence, the suboptimality and infeasibility converges to 0 in the ergodic sense
for any $\gamma>0$ when $\frac{1}{c}\geq L+\gamma\norm{A}^2$. Note \eqref{eq:pgadmm_x1} 
can be rewritten as $x^{k+1}=\prox{c\xi}(x^k-c[\nabla f(x^k)+A^\top(2\lambda^k-\lambda^{k-1})])$.
Let $g^*$ denote the convex conjugate of $g$; using Moreau proximal decomposition on $y$-updates in \eqref{eq:pgadmm_w1}, we get
{\small
\begin{align}
y^{k+1}= A x^{k+1}+ \gamma^{-1}\lambda^k-{1\over \gamma}~\prox{\gamma g^*}(\lambda^k+\gamma A x^{k+1}). \label{eq:pgadmm_w2}
\end{align}
}%
Combining \eqref{eq:pgadmm_w2} and \eqref{eq:pgadmm_y1} shows $\lambda^{k+1}=\prox{\gamma g^*}(\lambda^k+\gamma A x^{k+1})$. Thus, \eqref{eq:pgadmm} can be written as
{\small
\begin{subequations}\label{eq:pgadmm-final}
\begin{align}
x^{k+1}&\gets\prox{c\xi}\Big(x^k-c\big[\grad f(x^k)+A^\top(2\lambda^k-\lambda^{k-1})\big]\Big),\\
\lambda^{k+1}&\gets\prox{\gamma g^*}(\lambda^k+\gamma A x^{k+1}).
\end{align}
\end{subequations}
}%
The iterative scheme in \eqref{eq:pgadmm-final} is the same as the PDA proposed in~\cite{condat2013primal}, where Condat only considered the
convergence of the algorithm, and no iteration complexity was given in~\cite{condat2013primal}. The scheme in~\eqref{eq:pgadmm-final} is also a variant of PDA iterations in~\cite{chambolle2015ergodic}. In particular, PG-ADMM as written in \eqref{eq:pgadmm-final} generates the same iterate sequence as $(x^{k+1},\lambda^{k+1})=\cP\cD_{c,\gamma}(x^{k},\lambda^{k},x^{k+1},2\lambda^{k}-\lambda^{k-1})$ in \cite{chambolle2015ergodic} when the Bregman functions $D_x$ and $D_\lambda$ chosen as $\tfrac{1}{2}\norm{\cdot}^2$. Moreover, one can easily prove that Theorem 1 in \cite{chambolle2015ergodic} is still true for this variant of PDA for any $c,\gamma>0$ such that $(\frac{1}{c}-L)\frac{1}{\gamma}\geq \norm{A}^2$. It is worth emphasizing that this equivalence is no longer true on problems $\min_{x,y}\{\Phi(x)+g(y):\ Ax+By=b\}$ with a general $B$, instead of $B=-I_{n_y}$. Note that PG-ADMM is more general than PDAs in~\cite{chambolle2015ergodic,condat2013primal} in the sense that it can also deal with noisy gradients while PDAs cannot.\vspace*{-3mm}
\section{Proximal Gradient Methods for Distributed Optimization}
\label{sec:synch}
In this section, we provide {consensus formulations of the decentralized optimization problem in~\eqref{eq:dist_opt} for unweighted and weighted static (undirected) communication networks}; and these formulations are special cases of~\eqref{prob}. Hence, we develop two different distributed 
{algorithms based on} PG-ADMM in~\eqref{GADM-extension}, one for each formulation; and finally we derive the customized convergence bounds showing the effect of network topology for each implementation. Similarly, one can also implement SPG-ADMM in~\eqref{GADM-stoc-extension} on the two different decentralized formulations of \eqref{eq:dist_opt} to obtain the stochastic gradient variants of these distributed algorithms {based on SPG-ADMM}. The error bounds for these stochastic variants can be driven as we obtain the bounds for the deterministic versions using Theorem~\ref{lem:ergodic}. Due to space considerations, we skip their proofs and only state the error bounds for these variants using stochastic gradients as corollaries of the deterministic error bound results shown in Theorem~\ref{thm:DPGA-II} and Theorem~\ref{thm:DPGA-I}.

In the rest of this paper, we adopt the following notation. Let $\cG=(\cN,\cE)$ be a connected graph, where $\cN\triangleq\{1,\ldots,N\}$ denotes the set of computing nodes, and $\cE\subset\cN\times\cN$ denote the set of (undirected) edges. Without loss of generality, assume that edges in $\cE$ are oriented, i.e., $(i,j)\in\cE$ implies that $i<j$. Let $\cN_i\triangleq\{j\in\cN:\ (i,j)\in\cE\ \mathrm{or}\ (j,i)\in\cE\}$ denote the set of neighboring nodes of $i \in \cN$, and $d_i\triangleq|\cN_i|$ denote the degree of node $i\in\cN$. 
Let $\Omega \in \reals^{|\cN| \times |\cN|}$ denote the Laplacian, 
and $M\in\reals^{|\cE|\times |\cN|}$ denote the oriented edge-node incidence matrix, i.e., for $e=(i,j)\in\cE$ and $k\in\cN$, $M_{ek}$ equals to $1$ if $k=i$, to $-1$ if $k=j$, and to 0, otherwise.

\begin{remark}
\label{rem:laplacian}
Note that $\Omega=M^\top M$. Let $\psi_{\max}\triangleq\psi_1\geq\psi_2\geq\ldots\geq\psi_N$ be the eigenvalues of $\Omega$. Since we assumed that $\cG$ is \emph{connected}, we have $\psi_{N-1}>\psi_N=0$. Hence, $\Rank(M)=\Rank(\Omega)=N-1$. Moreover, $(M\otimes I_n)^\top (M\otimes I_n)=\Omega\otimes I_n$; and from the structure of $\Omega\otimes I_n$, it follows that $\{\psi_i\}_{i=1}^N$ are also the eigenvalues of $\Omega\otimes I_n$, each with \emph{algebraic multiplicity} $n$. Thus, $R\triangleq\Rank(M\otimes I_n)=n(N-1)$. Let $M\otimes I_n=U\Sigma V^\top$ denote the reduced singular value decomposition~(SVD) of $M\otimes I_n$, where $U=[u_1 \ldots u_{R}]\in\reals^{n|\cE|\times R}$, $V=[v_1 \ldots v_{R}]\in\reals^{n |\cN| \times R}$, $\Sigma=\diag(\sigma)$, and $\sigma\in\reals_{++}^{R}$. Note that $\sigma_{\max}(M\otimes I_n)=\max\{\sigma_r:\ r=1,\ldots,R\}=\sqrt{\psi_{1}}$, and $\sigma_{\min}(M\otimes I_n)=\min\{\sigma_r:\ r=1,\ldots,R\}=\sqrt{\psi_{N-1}}$.
\end{remark}
\subsection{DPGA Algorithm}
\label{sec:dpga2}
Using $M\in\reals^{|\cE|\times |\cN|}$, \eqref{eq:dist_opt} can be equivalently written as \vspace*{-1mm}
{\small
\begin{align}
\label{eq:DPGA-II formulation-compact}
\min_{\bx \in \reals^{n|\cN|}}\Big\{F(\bx)\triangleq\sum_{i\in\cN}\Phi_i(x_i):\ (M\otimes I_n)\bx=\mathbf{0}\Big\},
\end{align}
}%
where $\bx^\top=[x_1^\top,\ldots,x_N^\top]^\top$ and $\Phi_i$ is defined in \eqref{eq:F_i}. For each $(i,j)\in\cE$, define new set of primal variables $y_{ij}\in\reals^n$, and let $\by=[y_{ij}]_{(i,j)\in\cE}\in\reals^{n|\cE|}$. Now consider the following reformulation: \vspace*{-1mm}
{\small
\begin{align}
\label{eq:DPGA-II formulation}
\min_{\by,~\{x_i\}_{i\in\cN}}\left\{\sum_{i \in \cN} \Phi_i(x_i):\ x_i - y_{ij} = 0: \alpha_{ij}, \quad x_j - y_{ij} = 0: \beta_{ij},\quad \forall (i,j) \in \cE\right\},
\end{align}
}%
where $\alpha_{ij}\in\reals^n$ and $\beta_{ij}\in\reals^n$ denote the Lagrange multiplier vectors corresponding to the primal constraints $x_i - y_{ij} = 0$, and  $x_j - y_{ij} = 0$ in \eqref{eq:DPGA-II formulation}. 
Define $\alpha=[\alpha_{ij}]_{(i,j)\in\cE}\in\reals^{n|\cE|}$, and $\beta=[\beta_{ij}]_{(i,j)\in\cE}\in\reals^{n|\cE|}$. 
Note \eqref{eq:DPGA-II formulation} is a special case of \eqref{prob} with $g(\by)=0$, and one can employ PG-ADMM or SPG-ADMM to solve \eqref{eq:DPGA-II formulation}. 

Next, we focus on the implementation details of PG-ADMM. Indeed, it can be easily observed that $A_i$ of \eqref{prob} takes the following form for \eqref{eq:DPGA-II formulation}: $A_i=(\mathbf{1}_{d_i}\otimes I_n)\in\reals^{n d_i\times n}$. 
Furthermore, it can also be easily observed that $\sum_{i=1}^N \gamma_i B_i^\top B_i$ in Theorem~\ref{lem:ergodic} is equal to $\diag([\gamma_i+\gamma_j]_{(i,j)\in\cE})\otimes I_{n}\in\reals^{n|\cE|\times n|\cE|}$ for \eqref{eq:DPGA-II formulation}, and $B=[B_i]_{i\in\cN}$, obtained by vertically concatenating $\{B_i\}_{i\in\cN}$, has full column rank. For all $i\in\cN$, we set the stepsize $c_i$ 
according to Theorems~\ref{lem:ergodic} and \ref{thm-ergodic}, i.e., $0<c_i\leq 1/(L_i+\gamma_i\norm{A_i}^2)$. Hence, for the formulation \eqref{eq:DPGA-II formulation}, this corresponds to setting $c_i\leq 1/(L_i+\gamma_i d_i)$ since $\sigma_{\max}(A_i)=\sqrt{d_i}$ for $i\in\cN$. {For the convergence of $\{\bx^k\}_{k\geq 0}$ to a unique limit point, strict inequality is required when choosing $c_i$ for $i\in\cN$.}

Let $\{x_i^0\}_{i\in\cN}$ denote the set of initial primal iterates. The smooth part of the augmented Lagrangian $\phi_\gamma$ corresponding to the formulation \eqref{eq:DPGA-II formulation} can be written as \vspace*{-1mm}
{\small
\begin{align*}
\phi_\gamma(\bx,\by,\alpha,\beta)=\sum_{i\in\cN}f_i(x_i)+\sum_{(i,j)\in\cE}\left(\alpha_{ij}^\top(x_i-y_{ij})+\beta_{ij}^\top(x_j-y_{ij})
\right)+\sum_{(i,j)\in\cE}\left({\frac{\gamma_i}{2}}\norm{x_i-y_{ij}}^2+{\frac{\gamma_j}{2}}\norm{x_j-y_{ij}}^2\right)
\end{align*}
}%
{for node-specific penalty parameters $\{\gamma_i\}_{i\in\cN}\subset\reals_{++}$;} hence, $\grad_{x_i}\phi_\gamma$ can be computed as %
{\small
\begin{align}
\label{eq:DPGA-II-gradient}
\grad_{x_i}\phi_\gamma(\bx^k,y^k,\alpha^k,\beta^k)=\grad f(x_i^k)
+\sum_{j:(i,j)\in\cE}\left(\alpha_{ij}^k+\gamma_i(x_i^k-y_{ij}^k)\right)+\sum_{j:(j,i)\in\cE}\left(\beta_{ji}^k+\gamma_i(x_i^k-y_{ji}^k)\right)
\end{align}
}%
for $k\geq 0$ and the steps of PG-ADMM in \eqref{GADM-extension} take the following form: \vspace*{-1mm}
{\small
\begin{subequations}\label{eq:pre-dpga}
\begin{align}
x_j^{k+1} &= \prox{c_j\xi_j} \left( x_j^k - c_j\grad_{x_j}\phi_\gamma(\bx^k,y^k,\alpha^k,\beta^k) \right), \quad j \in \cN, \label{eq:update-x2} \\
y_{ij}^{k+1} &= \argmin_{y_{ij}} \left\{ -\left(\alpha_{ij}^k+\beta_{ij}^k\right)^\top y_{ij}+
\frac{\gamma_i}{2}\norm{x^{k+1}_i-y_{ij}}^2+\frac{\gamma_j}{2}\norm{x^{k+1}_j-y_{ij}}^2\right\}, \quad (i,j) \in \cE, \label{eq:update-z2}\\
\alpha_{ij}^{k+1} &= \alpha_{ij}^k + \gamma_i \left(x_i^{k+1} - y_{ij}^{k+1} \right), \quad (i,j)\in\cE, \label{eq:update-lambda2-2}\\
\beta_{ij}^{k+1} & = \beta_{ij}^k + \gamma_j \left(x_j^{k+1} - y_{ij}^{k+1} \right), \quad (i,j)\in\cE. \label{eq:update-lambda2}
\end{align}
\end{subequations}
}%
For $k\geq 0$, \eqref{eq:update-z2} can be solved in closed form: \vspace*{-1mm}
{\small
\begin{align}
\label{eq:dummy-primal-y2}
{(\gamma_i+\gamma_j)~y_{ij}^{k+1}=\alpha^k_{ij}+\beta^k_{ij}+\gamma_i x_i^{k+1}+\gamma_j x_j^{k+1}.}
\end{align}
}%
Summing \eqref{eq:update-lambda2-2} and \eqref{eq:update-lambda2}, and using \eqref{eq:dummy-primal-y2}, we get for $k\geq 0$ \vspace*{-1mm}
{\small
\begin{align*}
\alpha^{k+1}_{ij}+\beta^{k+1}_{ij}=\alpha^k_{ij}+\beta^k_{ij}+\gamma_i x_i^{k+1}+\gamma_j x_j^{k+1}-(\gamma_i+\gamma_j)y_{ij}^{k+1}=\mathbf{0}.
\end{align*}
}%
Hence, for each $(i,j)\in\cE$, we have $\alpha^k_{ij}+\beta^k_{ij}=0$ for $k\geq 1$. Suppose we initialize $\alpha^0=\beta^0\triangleq\mathbf{0}$, i.e., $\alpha^0_{ij}=\beta^0_{ij}=0$ for all $(i,j)\in\cE$, and $y_{ij}^0\triangleq(\gamma_i x_i^0+\gamma_j x_j^0)/(\gamma_i+\gamma_j)$ for $(i,j)\in\cE$. Therefore, using \eqref{eq:dummy-primal-y2}, we can conclude that for each $(i,j)\in\cE$, {$y_{ij}^k=(\gamma_i x_i^k+\gamma_j x_j^k)/(\gamma_i+\gamma_j)$} for all $k\geq 0$. Hence, $\alpha_{ij}^{k} = {\frac{\gamma_i\gamma_j}{\gamma_i+\gamma_j}}\sum_{\ell=1}^k \left(x_i^{\ell} - x_j^\ell \right)$, and $\beta_{ij}^{k} = {\frac{\gamma_i\gamma_j}{\gamma_i+\gamma_j}}\sum_{\ell=1}^k \left(x_j^{\ell} - x_i^\ell \right)$ for all $k\geq 1$. Therefore, combining these recursions, \eqref{eq:DPGA-II-gradient} can be computed for all $k\geq 0$ as \vspace*{-1mm}
{\small
\begin{align}
\grad_{x_i}\phi_\gamma(\bx^k,y^k,\alpha^k,\beta^k)
&=\grad f(x_i^k)+\sum_{j\in\cN_i}{\frac{\gamma_i\gamma_j}{\gamma_i+\gamma_j}}\left[(x_i^k-x_j^k)+\sum_{\ell=1}^k(x_i^\ell-x_j^\ell)\right]. \label{eq:DPGA-II-gradient-simple}
\end{align}
}%
Hence, $\grad_{x_j}\phi_\gamma$ in step \eqref{eq:update-x2} can be written in a compact form. Define $\Gamma\in\reals^{|\cN|\times |\cN|}$: for $i\in\cN$
\begin{align}
\Gamma_{ij}\triangleq 0,\ \ j\not\in\cN_i\cup\{i\},\quad \Gamma_{ij}\triangleq-\frac{\gamma_i\gamma_j}{\gamma_i+\gamma_j},\ \ j\in\cN_i,\quad \Gamma_{ii}\triangleq \sum_{j\in\cN_i}\frac{\gamma_i\gamma_j}{\gamma_i+\gamma_j};
\end{align}
and set $s_i^k \triangleq \sum_{j \in \cN_i\cup\{i\}} \Gamma_{ij} x_j^{k}$ for $k\geq 0$; and let $p_i^k\triangleq \sum_{\ell=1}^k s_i^k$ for $k\geq 1$ and $p_i^0=\mathbf{0}$. This initialization and \eqref{eq:DPGA-II-gradient-simple} together imply that $\grad_{x_i}\phi_\gamma(\bx^k,\bz^k,\lambda^k)=\nabla f_i(x_i^k)+ p_{i}^k + s_i^k$ for $k\geq 0$.
Therefore, the steps in \eqref{eq:update-x2}-\eqref{eq:update-lambda2} can be simplified as shown in Figure~\ref{alg:DPGA-II} below. {The algorithm works in a distributed fashion: each node $i\in\cN$ i) sends $\gamma_i$ to and receives $\gamma_j$ from all its neighbors $j\in\cN_i$ once at the beginning -- hence, $\Gamma_{ij}$ for $j\in\cN_i\cup\{i\}$ can be computed; ii) stores three variables in $\reals^n$: $x_i^k, s_i^k$ and $p_i^k$; iii) computes the proximal step; iv) broadcasts the updated variable $x_i^{k+1}$ to all $j\in\cN_i$, and then updates $s_i^{k+1}$; iv) each node updates $p_i^{k+1}$, and then repeats.}
\begin{figure}[htpb]
\centering
\framebox{\parbox{\columnwidth-10pt}{
{\small
\textbf{Algorithm DPGA} ( $ \bx^{0}, \{\gamma_i\}_{i\in\cN}$ ) \\[1.5mm]
Initialization: $c_i<(L_i+\gamma_i d_i)^{-1}$,\quad $s_i^0=\sum_{j \in \cN_i\cup\{i\}} \Gamma_{ij} x_j^{0}$,\quad $p_i^{0} = \mathbf{0}$,\quad $i\in\cN$\\
Step $k$: ($k \geq 0$) For $i\in\cN$ compute\\
\text{ } 1. $x_i^{k+1} = \prox{c_i\xi_i} \Big( x_i^k - c_i\big(\nabla f_i(x_i^k) + p_i^k + s_i^k \big)  \Big), \quad i \in \cN$\\
\text{ } 2. $s_i^{k+1} = \sum_{j \in \cN_i\cup\{i\}} \Gamma_{ij} x_j^{k+1}, \quad i \in \cN$\\
\text{ } 3. $p_i^{k+1} = p_i^{k} + s_i^{k+1}, \quad i \in \cN$\\
 }}}
\caption{\small Distributed Proximal Gradient Algorithm (DPGA)}
\label{alg:DPGA-II}
\end{figure}
\subsubsection{Error Bounds for DPGA \& Effect of Topology}
\label{sec:dpga-errorbounds}
In this section, we examine the effect of network topology on the convergence rate of DPGA, which is nothing but PG-ADMM customized to the decentralized formulation in \eqref{eq:DPGA-II formulation} as discussed in Section~\ref{sec:dpga2}. To obtain simple $\cO(1)$ constants in the error bounds, we set $\alpha_{ij}^0=\beta_{ij}^0=\mathbf{0}$ for all $(i,j)\in\cE$.
\begin{theorem}
\label{thm:DPGA-II}
Suppose a solution to \eqref{eq:central_opt}, $x^*\in\reals^n$, exists and {$\ri(\cap_{i\in\cN}\dom \xi_i)\neq\emptyset$}. Given arbitrary $\{x^0_i,\gamma_i\}_{i\in\cN}$, the DPGA iterate sequence $\{\bx^k\}_{k\geq 1}$, generated as shown in Figure~\ref{alg:DPGA-II}, {converges to an optimal solution} to \eqref{eq:dist_opt}. Moreover, the average sequence $\{\bar{\bx}^t\}_{t\geq 1}$, defined as $\bar{x}^t_i=(\sum_{k=1}^tx_i^k)/t$ for $i\in\cN$ and $t\geq 1$, satisfies the following 
bounds for all $t\geq 1$, \vspace*{-1mm}
{\small
\begin{align*}
-\frac{1}{t}\left( \frac{2\norm{Q}}{\sigma_{\min}(\Omega)}\sum_{i\in\cN}\kappa_i^2 + \sum_{i \in \cN} \frac{1}{2c_i} \norm{x^* - x_{i}^0  }^2\right)\leq F(\bar{\bx}^t)-F^*\leq \frac{1}{t}\left(\sum_{i \in \cN} \frac{1}{2c_i} \norm{x^* - x_{i}^0  }^2\right),\\ 
\left(\sum_{(i,j)\in\cE}\norm{\bar{x}_i^t-\bar{x}^t_j}^2\right)^{\tfrac{1}{2}}\leq \frac{1}{t}\left[\norm{Q}\left(\sum_{i\in\cN}\frac{\kappa_i^2}{\sigma_{\min}(\Omega)}+1\right)+ \sum_{i \in \cN} \frac{1}{2c_i} \norm{x^*-x_i^0}^2\right],
\end{align*}}%
when the step-size $c_i\leq(L_i+\gamma_i d_i)^{-1}$ for $i\in\cN$, where $\kappa_i>0$ denotes a bound on the elements of $\partial \Phi_i(x^*)$, i.e., if $q\in\partial \Phi_i(x^*)$, then $\norm{q}\leq \kappa_i$, for each $i\in\cN$, and {$Q\triangleq\diag([\frac{1}{\gamma_i}+\frac{1}{\gamma_j}]_{(i,j)\in\cE})$.}
\end{theorem}
\begin{proof}
$x^*$ is an optimal solution to \eqref{eq:central_opt} if and only if $\bx^*\in\reals^{n|\cN|}$ such that $x_i^*=x^*$ for $i\in\cN$ is optimal to \eqref{eq:DPGA-II formulation-compact}; hence, $(\bx^*,\by^*)$ is optimal to \eqref{eq:DPGA-II formulation} for $\by^*\in\reals^{n|\cE|}$ such that $y^*_{ij}=x^*$ for all $(i,j)\in\cE$. The formulation \eqref{eq:DPGA-II formulation-compact}, equivalent to~\eqref{eq:DPGA-II formulation}, can be written as \vspace*{-1mm}
{\small
\begin{flalign}
\label{formulation3-II}
&\min_{\bx}\{F(\bx)\triangleq\sum_{i \in \cN} \Phi_i( x_i ):\ x_i - x_j = \mathbf{0}  : \theta_{ij},\ (i,j) \in \cE\},
\end{flalign}
}%
where $\theta_{ij}\in\reals^n$ denotes the dual variable corresponding to the primal constraint $x_i - x_j = \mathbf{0}$. Since $\ri\big(\bigcap_{i\in\cN}\dom \xi_i\big)\neq\emptyset$, Assumption~\ref{assumption-0} implies the Slater's condition for both \eqref{formulation3-II} and \eqref{eq:DPGA-II formulation}; hence, dual optimal solutions to \eqref{formulation3-II} and \eqref{eq:DPGA-II formulation} exist. Let $\theta^*\triangleq[\theta^{*}_{ij}]_{(i,j)\in\cE}$ and $(\alpha^*,\beta^*)$ denote some dual solutions to \eqref{formulation3-II} and \eqref{eq:DPGA-II formulation}, respectively. Note \eqref{eq:DPGA-II formulation} is a special case of \eqref{prob}, and \eqref{eq:DPGA-II formulation} satisfies Assumption~\ref{assumption-1}; moreover, $\{\bx^k\}_{k\geq 1}$ generated by DPGA in Fig.~\ref{alg:DPGA-II} is the same as with the sequence produced by PG-ADMM applied to \eqref{eq:DPGA-II formulation}. Thus, convergence of the DPGA iterate sequence $\{\bx^k\}_{k\geq 0}$ to an optimal solution follows from Theorem~\ref{thm-ergodic} when $c_i<(L_i+\gamma_i d_i)^{-1}$ for $i\in\cN$; moreover, the bound \eqref{eq:PG-key-ineq} given in Theorem~\ref{lem:ergodic} is valid for the DPGA iterates $\{\bx^k\}$ whenever $c_i\leq(L_i+\gamma_i d_i)^{-1}$ for $i\in\cN$. In the rest, we analyze the error bounds.

From the optimality conditions for \eqref{formulation3-II}, $\theta^*$ is an optimal dual solution to \eqref{formulation3-II} if and only if\vspace*{-1mm}
{\small
\begin{align}
\label{eq:opt-conditions-theta_ij}
\mathbf{0}\in\partial \Phi_i(x^*)+\sum_{j:(i,j)\in\cE}\theta^*_{ij}-\sum_{j:(j,i)\in\cE}\theta^*_{ji},\quad \forall\ i\in\cN,
\end{align}
}%
i.e., $-(M\otimes I_n)^\top\theta^*\in \partial F(\bx^*)$ such that $x^*_i=x^*$ for $i\in\cN$. Similarly, according to first order optimality conditions for \eqref{eq:DPGA-II formulation}, $\alpha^*\in\reals^{n|\cE|}$ and $\beta^*\in\reals^{n|\cE|}$ are dual optimals to \eqref{eq:DPGA-II formulation} if and only if\vspace*{-1mm}
{\small
\begin{align*}
&\mathbf{0} \in\partial \Phi_i(x^*)+\sum_{j:(i,j)\in\cE}\alpha^*_{ij}+\sum_{j:(j,i)\in\cE}\beta^*_{ji},\quad \forall\ i\in\cN,
\end{align*}
}%
and $\alpha^*_{ij}+\beta^*_{ij} =\mathbf{0}$ for all $(i,j)\in\cE$. Therefore, given an optimal dual solution to \eqref{formulation3-II}, say $\theta^*$, one can construct a dual optimal solution $(\alpha^*,\beta^*)$ to \eqref{eq:DPGA-II formulation} by simply setting $\alpha^*=\theta^*$ and $\beta^*=-\theta^*$. 
In the rest of the proof, we fix $(\bx^*,\by^*,\alpha^*,\beta^*)$ as a primal-dual optimal solution to \eqref{eq:DPGA-II formulation} such that $x_i^*=x^*$ for $i\in\cN$, $y^*_{ij}=x^*$, $\alpha^*_{ij}=\theta^*_{ij}$ and $\beta^*_{ij}=-\theta^*_{ij}$ for all $(i,j)\in\cE$, where $\theta^*$ is some dual optimal solution to \eqref{formulation3-II}.

For $i\in\cN$, $c_i=1/(L_i+\gamma_i d_i+\vartheta_i)$ for some $\vartheta_i\geq 0$, and $A_i=(\mathbf{1}_{d_i}\otimes I_n)\in\reals^{n d_i\times n}$; hence,  $A_i^\top A_i = d_i I_n$
and $\frac{1}{c_i}\norm{x_i^0-x^*}_{I_n-\gamma_i c_i A_i^\top A_i}=(L_i+\vartheta_i)\norm{x_i^0-x^*}^2$. Also $\sum_{i=1}^N \gamma_i  B_i^\top B_i=\diag([\gamma_i+\gamma_j]_{(i,j)\in\cE})\otimes I_{n}$; and since we initialize $\by^0$ such that $y_{ij}^0=(\gamma_i x_i^0+\gamma_j x_j^0)/(\gamma_i+\gamma_j)$ and $y_{ij}^*=x^*$ for all $(i,j)\in\cE$, it trivially follows from the convexity of $\norm{\cdot}^2$ that \vspace*{-1mm}
{\small
$$\norm{\by^*-\by^0}^2_{\sum_{i=1}^N\gamma_iB_i^\top B_i}\leq \sum_{(i,j)\in\cE}\left(\gamma_i\norm{x^*-x_i^0}^2+\gamma_j\norm{x^*-x_j^0}^2\right)=\sum_{i\in\cN}d_i\gamma_i \norm{x^*-x_i^0}^2.$$
}%
Recall that we initialize $\alpha^0=\beta^0=\mathbf{0}$. Furthermore, Theorem~\ref{lem:ergodic} holds for all $\alpha=[\alpha_{ij}]_{(i,j)\in\cE}\in\reals^{n|\cE|}$ and $\beta=[\beta_{ij}]_{(i,j)\in\cE}\in\reals^{n|\cE|}$, where $\alpha_{ij},\beta_{ij}\in\reals^{n}$ for $(i,j)\in\cE$. Hence, given $\theta=[\theta_{ij}]_{(i,j)\in\cE}$ for some $\theta_{ij}\in\reals^n$ for all $(i,j)\in\cE$, we set $\alpha_{ij}=\theta_{ij}$ and $\beta_{ij}=-\theta_{ij}$ for all $(i,j)\in\cE$, and invoke Theorem~\ref{lem:ergodic} for this specific choice of $\alpha=\theta$ and $\beta=-\theta$ to obtain the following inequality, customized for \eqref{eq:DPGA-II formulation}: for all $\theta$ and $t\geq 1$, \vspace*{-1mm}
{\small
\begin{align}
F(\bar{\bx}^t) - F^* + \theta^\T (M\otimes I_n)\bar{\bx}^t
&=F(\bar{\bx}^t) - F^* + \sum_{(i,j) \in \cE} \alpha_{ij}^\T (  \bar{x}_i^t - \bar{y}_{ij}^t) + \beta_{ij}^\T (\bar{x}_j^t - \bar{y}_{ij}^t), \nonumber\\
&\leq \; \frac{1}{2t} \Big[\sum_{i\in\cN}\Big(\frac{1}{\gamma_i}\sum_{j\in\cO_i} \norm{\theta_{ij}}^2\Big) + \sum_{i \in \cN} (L_i+\vartheta_i) \norm{x^* - x_{i}^0  }^2 + \norm{\by^*-\by^0}^2_{\sum_{i=1}^N\gamma_iB_i^\top B_i}\Big], \nonumber \\
&\leq \; \frac{1}{2t}\Big(\sum_{(i,j)\in\cE}\Big(\frac{1}{\gamma_i}+\frac{1}{\gamma_j}\Big)\norm{\theta_{ij}}^2 + \sum_{i \in \cN} \frac{1}{c_i} \norm{x^* - x_{i}^0  }^2\Big). \label{eq:ergodic-II}
\end{align}
}%
Hence, setting $\theta=\mathbf{0}$ in \eqref{eq:ergodic-II} leads to \vspace*{-1mm}
{\small
\begin{align}
\label{eq:upper-subopt-II}
F(\bar{\bx}^t) & - F^*
\leq \; \frac{1}{t} \sum_{i \in \cN} \frac{1}{2c_i} \norm{x^* - x_{i}^0  }^2.
\end{align}
}%
From convexity of $\Phi_i$, \eqref{eq:opt-conditions-theta_ij} and the fact that $(M\otimes I_n)\bx^*=\mathbf{0}$, it follows that \vspace*{-1mm}
{\small
\begin{align}
\label{subgradient-II}
0 \leq F( \bar{\bx}^t ) - F^* + {\theta^*}^\top (M\otimes I_n)(\bar{\bx}^t-\bx^*)=F( \bar{\bx}^t ) - F^* + {\theta^*}^\top (M\otimes I_n)\bar{\bx}^t.
\end{align}
}%
Adding 
${\theta^*}^\top (M\otimes I_n)\bar{\bx}^t$ to both sides, and invoking \eqref{eq:ergodic-II} for $\theta$ such that $\theta_{ij}=2\theta_{ij}^*$ for $(i,j)\in\cE$, we get \vspace*{-1mm}
{\small
\begin{align}
{\theta^*}^\T (M\otimes I_n)\bar{\bx}^t \leq F( \bar{\bx}^t ) - F^* + (2\theta^*)^\T (M\otimes I_n)\bar{\bx}^t \leq \frac{1}{t}\left( 2 \norm{\theta^*}_Q^2 + \sum_{i \in \cN} \frac{1}{2c_i} \norm{x^* - x_{i}^0  }^2\right). \label{eq:lower-subopt-II}
\end{align}
}%
Invoking \eqref{eq:ergodic-II} once again for $ \theta = \theta^* + (M\otimes I_n) \bar{\bx}^t / \norm{ (M\otimes I_n) \bar{\bx}^t }$ and using \eqref{subgradient-II}, we get 
{\small
\begin{align}
\left(\sum_{(i,j)\in\cE}\norm{\bar{x}_i^t-\bar{x}^t_j}^2\right)^{\tfrac{1}{2}}=
\norm{(M\otimes I_n) \bar{\bx}^t } &\leq F( \bar{x}^t ) - F^* + {\theta^*}^\T (M\otimes I_n) \bar{\bx}^t + \norm{(M\otimes I_n) \bar{\bx}^t }, \nonumber \\
& \leq \frac{1}{2t} \left( \norm{\theta^*+\frac{(M\otimes I_n) \bar{\bx}^t}{\norm{(M\otimes I_n) \bar{\bx}^t}}}_Q^2 + \sum_{i \in \cN} \frac{1}{c_i}\norm{x_i^* - x_{i}^0}^2\right),\nonumber \\ 
&\leq  \frac{1}{t}\left(\sigma_{\max}(Q)+\norm{\theta^*}_Q^2 + \sum_{i\in\cN}\frac{1}{2c_i}\norm{x^* - x_{i}^0}^2\right), \label{consensus_ineq-II}
\end{align}
}%
where the last inequality follows from 
$\norm{a+b}_Q^2\leq 2\norm{a}_Q^2+2\norm{b}_Q^2$. Next, by appropriately choosing a dual optimal \eqref{formulation3-II}, we bound $\norm{\theta^*}_Q$, which appears in \eqref{eq:lower-subopt-II} and \eqref{consensus_ineq-II}.

According to Remark~\ref{rem:laplacian}, since $(M\otimes I_n)^\top \theta^*$ lies in the row space, there exists $\{ \lambda_r \}_{r=1}^{R}$ such that $-(M\otimes I_n)^\top \theta^* = \sum_{r = 1}^{R} \lambda_r v_r$. Following~\cite{makhdoumi2014broadcast}, choose $\bar{\theta} \triangleq - \sum_{r = 1}^{R} u_r\lambda_r/\sigma_r=U\diag(\lambda/\sigma)$, where $\diag(\lambda/\sigma)\in\reals^{R\times R}$ is a diagonal matrix with its $r$-th diagonal entry is $\lambda_r/\sigma_r$ for $r=1,\ldots,R$. Hence, $\bar{\theta}\in\reals^{n|\cE|}$ satisfies $-(M\otimes I_n)^\T \bar{\theta} = V \Sigma U^\T U \diag(\lambda/\sigma) = - (M\otimes I_n)^\T \theta^*$; therefore, $-(M\otimes I_n)^\T \bar{\theta}\in\partial F(\bx^*)$ as well. Recall that the optimality conditions stated in \eqref{eq:opt-conditions-theta_ij} 
can be written more compactly as $\mathbf{0}\in\partial F(\bx^*)+(M\otimes I_n)^\top \theta^*$; hence, $\bar{\theta}$ is an optimal dual solution to \eqref{formulation3-II} as well. Finally, using the local bounds $\kappa_i$ on the subdifferential of $\Phi_i$ at $x^*\in\reals^n$, we can bound $\norm{\bar{\theta}}$ from above as follows: \vspace*{-1mm}
{\small
\begin{align}
\norm{\bar{\theta}}^2 &= \sum_{r = 1}^{R} \frac{\lambda_r^2}{\sigma_r^2} \leq \frac{1}{\psi_{N-1}} \sum_{r = 1}^{R} \lambda_r^2 = \frac{1}{\sigma_{\min}(\Omega)} \norm{(M\otimes I_n)^\top \theta^*}^2. \nonumber
\end{align}
}%
Hence, $\norm{\bar{\theta}}_Q^2\leq\frac{\sigma_{\max}(Q)}{\sigma_{\min}(\Omega)}\sum_{i\in\cN}\kappa_i^2$. Thus, using this bound within \eqref{eq:lower-subopt-II} and \eqref{consensus_ineq-II}, and combining the resulting inequalities together with \eqref{eq:upper-subopt-II} and \eqref{subgradient-II} implies the desired bounds.
\end{proof}
Suppose 
$x_i^0=x^0$, $c_i=(L_i+\gamma d_i)^{-1}$ and $\gamma_i=\gamma$ for $i\in\cN$ for some $x^0\in\reals^n$ and $\gamma>0$. The bounds in Theorem~\ref{thm:DPGA-II} can be simplified further. Indeed, observe that $\sum_{i \in \cN} \frac{1}{2c_i} \norm{x^*-x_i^0}^2
=(\gamma|\cE|+\sum_{i\in\cN}L_i/2)\norm{x^*-x^0}^2$. Therefore, for all $t\geq 1$, \vspace*{-1mm}
{\small
\begin{align*}
\max \Big\{ |F(\bar{\bx}^t)-F^*|, \ \Big(\sum_{(i,j)\in\cE} \norm{\bar{x}_i^t-\bar{x}^t_j}^2 \Big)^{\tfrac{1}{2}} \Big\}
\leq  \frac{1}{t}\left[\frac{4}{\gamma}\left(\sum_{i\in\cN}\frac{\kappa_i^2}{\sigma_{\min}(\Omega)}+1\right)+ \left(\gamma |\cE| +\sum_{i \in \cN} \frac{L_i}{2}\right) \norm{x^*-x^0}^2\right].
\end{align*}
}%
Note that $\gamma^*=\frac{2}{\norm{x^0-x^*}}\sqrt{\Big(\sum_{i\in\cN}\frac{\kappa_i^2}{\sigma_{\min}(\Omega)}+1\Big)/|\cE|}$ is an optimal choice for constant penalty.
\subsection{DPGA-W Algorithm for weighted communication networks}
\label{sec:dpga1}
A formulation for \emph{weighted} networks that is suitable for implementing PG-ADMM in~\eqref{GADM-extension} or SPG-ADMM in~\eqref{GADM-stoc-extension} follows from~\cite{makhdoumi2014broadcast} using communication matrices. In particular, $W\in\reals^{|\cN|\times |\cN|}$ is called a \emph{communication matrix} if for all $i\in\cN$, $W_{ij}=0$ for all $j\not\in\cN_i\cup\{i\}$, $W_{ij}<0$ for all $i\in\cN_i$, and $W_{ii}=-\sum_{j\in\cN_i}W_{ij}$. It is easy to show that for  $\bx=[x_i]_{i\in\cN}\in\reals^{n|\cN|}$ satisfying $(W\otimes I_n)~\bx = \mathbf{0}$, we have $x_i=\bar{x}$ for all $i\in\cN$ for some $\bar{x}\in\reals^n$. Therefore, given $W$ with properties above,
\eqref{eq:dist_opt} can be equivalently written as \vspace*{-1mm}
{\small
\begin{align}
\label{eq:DPGA-I formulation-compact}
\min_{\bx \in \reals^{n|\cN|}}\Big\{F(\bx)\triangleq\sum_{i\in\cN}\Phi_i(x_i):\ (W\otimes I_n)\bx=\mathbf{0}\Big\},
\end{align}}%
where 
$\Phi_i$ is defined in \eqref{eq:F_i}. Note that the Laplacian $\Omega$ of the graph $\cG$ is also a \emph{communication matrix}, and can also be used to model unweighted networks.
In the rest of this section, assume that \eqref{eq:central_opt} has a solution, and \eqref{eq:DPGA-I formulation-compact} satisfies Assumption~\ref{assumption-1}.

For each $i\in\cN$, define new set of primal variables $y_{ij}\in\reals^n$ for $j\in\cN_i\cup\{i\}$, and form $\by_i\triangleq[y_{ij}]_{j\in\cN_i\cup\{i\}}\in\reals^{(d_i+1)n}$. Let $\cY_i\triangleq\{\by_i:\ \sum_{j\in\cN_i\cup\{i\}}y_{ij}=\mathbf{0}\}$ for $i\in\cN$, and define $g(\by)\triangleq \sum_{i \in \cN}  \mathds{1}_{\cY_i}(\by_i)$, where $\mathds{1}_{\cY_i}$ denotes the indicator function of the set $\cY_i$, i.e., $\mathds{1}_{\cY_i}(\by_i)$ is equal to $0$ if $\by_i\in\cY_i$, and to $+\infty$ otherwise, where $\by^\T=[\by_1^\T,\ldots,\by_N^\T]$. Hence, consider the following equivalent formulation proposed in~\cite{makhdoumi2014broadcast}: \vspace*{-1mm}
{\small
\begin{align}
\label{eq:DPGA-I formulation}
\min_{\by,\{x_i\}_{i\in\cN}} \quad &g(\by)+\sum_{i\in\cN} \Phi_i(x_i) \quad
\text{s.t.} \quad W_{ij}x_j-y_{ij}=0:\ \lambda_{ij},\quad \forall~j\in\cN_i\cup\{i\},\ \forall~i\in\cN,
\end{align}}%
where $\lambda_{ij}\in\reals^n$ denotes the Lagrange multiplier corresponding to the primal constraint $W_{ij}x_j-y_{ij}=0$ in \eqref{eq:DPGA-I formulation}. 
Define $\lambda_i=[\lambda_{ij}]_{j\in\cN_i\cup\{i\}}\in\reals^{(d_i+1)n}$, and $\lambda=[\lambda_i]_{i\in\cN}\in\reals^{n(2|\cE|+|\cN|)}$.

Clearly 
\eqref{eq:DPGA-I formulation} is a special case of \eqref{prob}, and one can employ PG-ADMM or SPG-ADMM to solve \eqref{eq:DPGA-I formulation}. In the rest, we focus on the implementation details of PG-ADMM. Indeed, it can be easily observed that $A_i$ of \eqref{prob} takes the following form for \eqref{eq:DPGA-I formulation}: $A_i=(\mathbf{\omega}_i\otimes I_n)\in\reals^{n (d_i+1)\times n}$, where $\mathbf{\omega}_i\in\reals^{d_i+1}$ is obtained from the $i$-th column of $W$ after $0$ entries are removed without changing the order of non-zero entries. Furthermore, for the constraint matrix 
in \eqref{eq:DPGA-I formulation},
{$\norm{\by^*-\by^0}^2_{\sum_{i=1}^N\gamma_iB_i^\top B_i}$ in Theorem~\ref{lem:ergodic} is equal to $\sum_{i\in\cN}\gamma_i\sum_{j\in\cN_i\cup\{i\}}(y^0_{ji}-y^*_{ji})^2$, and $B=[B_i]_{i\in\cN}$, obtained by vertically concatenating $\{B_i\}_{i\in\cN}$, has full column rank.}
For all $i\in\cN$, we set the stepsize $c_i$ according to Theorems~\ref{lem:ergodic} and \ref{thm-ergodic}.
Since $\sigma_{\max}(A_i)=\norm{\omega_i}$ for $i\in\cN$ for the formulation \eqref{eq:DPGA-I formulation}, this corresponds to setting $c_i\leq 1/(L_i+\gamma_i \norm{\omega_i}^2)$.

The smooth part of the augmented Lagrangian $\phi_\gamma$ 
for the formulation \eqref{eq:DPGA-I formulation} can be written as \vspace*{-1mm}
{\small
\begin{align*}
\phi_\gamma(\bx,\by,\lambda)=\sum_{i\in\cN}\left[f_i(x_i)+\sum_{j\in\cN_i\cup\{i\}}\lambda_{ij}^\top(W_{ij}x_j-y_{ij})
+\sum_{j\in\cN_i\cup\{i\}}\frac{\gamma_j}{2}\norm{W_{ij}x_j-y_{ij}}^2\right],
\end{align*}
}%
for node-specific penalty parameters $\{\gamma_i\}_{i\in\cN}\subset\reals_{++}$; hence, $\grad_{x_j}\phi_\gamma$ can be computed as \vspace*{-1mm}
{\small
\begin{align*}
\grad_{x_j}\phi_\gamma(\bx^k,\by^k,\lambda^k)=\grad f(x_j^k)+\sum_{i \in \cN_j\cup\{j\}}W_{ij} \Big[ \lambda_{ij}^k + \gamma_j (W_{ij} x_j^k - y_{ij}^k ) \Big]
\end{align*}
}%
and the steps of PG-ADMM in \eqref{GADM-extension} take the following form: \vspace*{-1mm}
{\small
\begin{subequations}\label{eq:DPGA-W-updates}
\begin{align}
x_j^{k+1} &= \prox{c_j\xi_j} \left( x_j^k - c_j\grad_{x_j}\phi_\gamma(\bx^k,\by^k,\lambda^k) \right), \quad j \in \cN, \label{eq:update-x1} \\
\by_i^{k+1} &= \argmin_{\by_i} \left\{ 
\sum\limits_{j \in \cN_i\cup\{i\}}\frac{\gamma_j}{2} \left\|  W_{ij} x_j^{k+1} - y_{ij} + \frac{1}{\gamma_j} \lambda_{ij}^k \right\|^2:\ \sum_{j\in\cN_i\cup\{i\}}y_{ij}=0\right\}, \quad i \in \cN, \label{eq:update-z1}  \\
\lambda_{ij}^{k+1} &= \lambda_{ij}^k + \gamma_j \left( W_{ij} x_j^{k+1} - y_{ij}^{k+1} \right), \quad j\in\cN_i\cup\{i\},\ i\in \cN. \label{eq:update-lambda1}
\end{align}
\end{subequations}
}%
Except for $x$-step in \eqref{eq:update-x1}, the $y$-step in \eqref{eq:update-z1} and $\lambda$-step in \eqref{eq:update-lambda1} are exactly the same as those in~\cite{makhdoumi2014broadcast} when for all $i\in\cN$, $\gamma_i=\gamma$ for some $\gamma>0$. Instead of taking proximal gradient step,  $x_j^{k+1}$ is computed in~\cite{makhdoumi2014broadcast} by solving $\min_{x_j\in\reals^n}\Phi_j(x_j) 
+\frac{\gamma}{2}\sum_{i\in\cN_j\cup\{j\}}\norm{W_{ij}x_j-y^k_{ij}+\frac{\lambda^k_{ij}}{\gamma}}^2$, which is equivalent to computing $\prox{\xi_j+f_j}$. Note that even if both $\xi_j$ and $f_j$ have simple prox maps, the prox map of the sum is not necessarily simple and it can be impractical to compute. 

Since $y$-step and $\lambda$-step are the same as those in~\cite{makhdoumi2014broadcast}, the results of this paragraph directly follow from Section~III of~\cite{makhdoumi2014broadcast}. Let $\{x_i^0\}_{i\in\cN}$ denote the set of initial primal iterates. For $k\geq 0$, let $p_i^{k+1}$ be the optimal Lagrange multiplier corresponding to $\by_i\in\cY_i$ constraint in \eqref{eq:update-z1}; hence, $y_{ij}^{k+1}=W_{ij}x_j^{k+1}+\frac{1}{\gamma_j}(\lambda_{ij}^k-p_i^{k+1})$. On the other hand, combining this equality with \eqref{eq:update-lambda1}, we conclude that $\lambda_{ij}^{k+1}=p_i^{k+1}$ for all $k\geq 0$. Moreover, since $\by_i^{k+1}\in\cY_i$, the optimal dual $p_i^{k+1}$ can be computed using the recursion: {$p_i^{k+1}=p_i^k+ (\sum_{j\in\cN_i\cup\{i\}}\frac{1}{\gamma_j})^{-1}\sum_{j\in\cN_i\cup\{i\}}W_{ij}x_j^{k+1}$} for $k\geq 0$. Suppose that we initialize $\lambda_{ij}^0=p_i^0$ for all $j\in\cN_i\cup\{i\}$ for some given $p_i^0$ for all $i\in\cN$.
Finally, by defining $s_i^0=\mathbf{0}$ and {$s_i^k \triangleq (\sum_{j\in\cN_i\cup\{i\}}\frac{1}{\gamma_j})^{-1}\sum_{j\in\cN_i\cup\{i\}}W_{ij}x_j^{k}$} for $k\geq 1$, and initializing 
$y_{ij}^0\triangleq W_{ij}x_j^0$ for all $j\in\cN_i\cup\{i\}$ and $i\in\cN$, the computation of $\grad_{x_j}\phi_\gamma$ in~\eqref{eq:update-x1} can be simplified. Indeed, for any $j\in\cN$, $\lambda_{ij}^k + \gamma_j (W_{ij} x_j^k - y_{ij}^k )=2p_{i}^k - p_{i}^{k-1}$ for all $i\in\cN_j\cup\{j\}$; hence, \vspace*{-1mm}
{\small
\[\grad_{x_j}\phi_\gamma(\bx^k,\by^k,\lambda^k)-\nabla f_j(x_j^k)=\sum_{i \in \cN_j\cup\{j\}} W_{ij} \left(\lambda_{ij}^k + \gamma_j (W_{ij} x_j^k - y_{ij}^k ) \right)
=\sum_{i \in \cN_j\cup\{j\}} W_{ij} \left(p_{i}^k + s_i^k\right).\]
}%
holds for $k\geq 0$. Note that this is true for $k=0$ because of how we initialize $\lambda^0$ and $\by^0$. Therefore, the steps in \eqref{eq:update-x1}-\eqref{eq:update-lambda1} can be simplified as shown in Figure~\ref{alg:DPGA-I} below.
\begin{figure}[htpb]
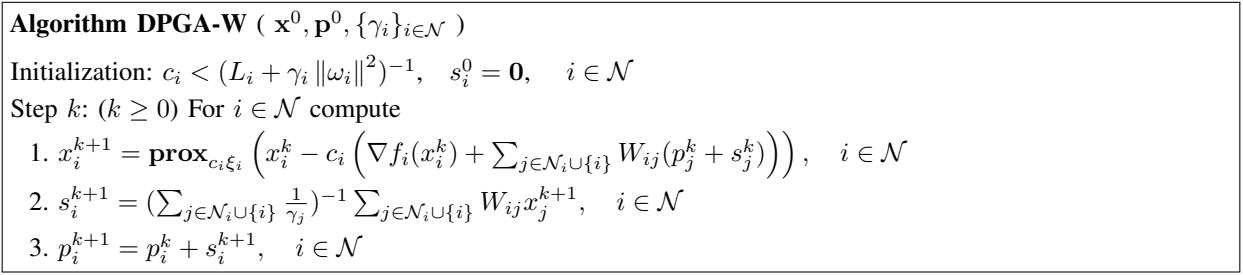

\centering
\framebox{\parbox{\columnwidth - 8pt}{
{\small
\textbf{Algorithm DPGA-W} ( $\bx^{0},\bp^0,\{\gamma_i\}_{i\in\cN}$ ) \\[1.5mm]
Initialization: $c_i<(L_i+\gamma_i \norm{\omega_i}^2)^{-1}$,\quad $s_i^0=\mathbf{0}$, \quad $i\in\cN$\\
Step $k$: ($k \geq 0$) For $i\in\cN$ compute\\[1mm]
\text{ } 1. $x_i^{k+1} = \prox{c_i\xi_i}\left( x_i^k - c_i \left(\nabla f_i(x_i^k) + \sum_{j \in \cN_i\cup\{i\}}  W_{ij} ( p_j^k + s_j^k   ) \right) \right), \quad i \in \cN$  \\[1mm]
\text{ } 2. $s_i^{k+1} =  (\sum_{j\in\cN_i\cup\{i\}}\frac{1}{\gamma_j})^{-1}\sum_{j\in\cN_i\cup\{i\}}W_{ij}x_j^{k+1}, \quad i \in \cN$\\[1mm]
\text{ } 3. $p_i^{k+1} = p_i^{k} + s_i^{k+1}, \quad i \in \cN$
 }}}
\caption{\small Distributed Proximal Gradient Algorithm for Weighted Networks (DPGA-W)}
\label{alg:DPGA-I}
\end{figure}

As in DPGA, to be able compute $s_i^k$ updates DPGA-W requires each node $i\in\cN$ to send $\gamma_i$ to and receive $\gamma_j$ from all its neighbors $j\in\cN_i$ once at the beginning. Moreover, stepsize $c_i$ depends on $\omega_i$, which is formed by the weights $W_{ji}$ assigned to $i$ by all its neighbors $j\in\cN_i$; therefore, assigned weights are exchanged among neighbors once at the beginning as well. Note that while DPGA-W can be applied to more general weighted communication networks, it requires each node to communicate two times with its neighbors per iteration, in contract to \emph{one} time for DPGA.
\subsubsection{Error Bounds for DPGA-W \& Effect of Network Topology}
In this section, we examine the effect of network topology on the convergence rate of DPGA-W, which is nothing but PG-ADMM customized to the decentralized formulation in \eqref{eq:DPGA-I formulation} as discussed in Section~\ref{sec:dpga1}. 
To obtain simple $\cO(1)$ constants in the error bounds, we set $p_i^0=\mathbf{0}$ for all $i\in\cN$.
\begin{theorem}
\label{thm:DPGA-I}
Suppose a solution to \eqref{eq:central_opt}, $x^*\in\reals^n$, exists and {$\ri(\cap_{i\in\cN}\dom \xi_i)\neq\emptyset$}. Given arbitrary $\{x^0_i,\gamma_i\}_{i\in\cN}$, let $p_i^0=\mathbf{0}$ for all $i\in\cN$, and $\{\bx^k\}_{k\geq 1}$ be the DPGA-W iterate sequence, generated as shown in Figure~\ref{alg:DPGA-I}, {converges to an optimal solution} to \eqref{eq:dist_opt}. Moreover, the average sequence $\{\bar{\bx}^t\}_{t\geq 1}$ 
satisfies the following bounds for all $t\geq 1$, \vspace*{-1mm}
{\small
\begin{align*}
-\frac{1}{t}\left( \frac{2\tau_{\max}}{\sigma^2_{\min}(W)}\sum_{i\in\cN}\kappa_i^2 + \sum_{i \in \cN} \frac{1}{2c_i} \norm{x_{i}^0-x^*}^2\right) \leq F(\bar{\bx}^t)-F^*\leq \frac{1}{t}\left(\sum_{i \in \cN} \frac{1}{2c_i} \norm{x_{i}^0-x^*}^2\right),\\ 
\norm{(\Omega\otimes I_n) \bar{\bx}^t }\leq \frac{1}{t}\left(\tau_{\max}\left(\sum_{i\in\cN}\frac{\kappa_i^2}{\sigma^2_{\min}(W)}+1\right)+ \sum_{i \in \cN} \frac{1}{2c_i} \norm{x_i^0-x^*}^2\right),
\end{align*}}%
where $\kappa_i>0$ denotes an upper bound on the elements of $\partial \Phi_i(x^*)$, i.e., if $q\in\partial \Phi_i(x^*)$, then $\norm{q}\leq \kappa_i$, for each $i\in\cN$, and $\tau_{\max}\triangleq\max_{i\in\cN}\sum_{j\in\cN_i}\frac{1}{\gamma_j}$.
\end{theorem}
\begin{proof}
$x^*\in\reals^n$ is an optimal solution to \eqref{eq:central_opt} if and only if $\bx^*\in\reals^{n|\cN|}$ such that $x_i^*=x^*$ for all $i\in\cN$ is optimal to \eqref{eq:DPGA-I formulation-compact}; hence, $(\bx^*,\by^*)$ is optimal to \eqref{eq:DPGA-I formulation} for $\by^*\in\reals^{n(2|\cE|+|\cN|)}$ such that $y^*_{ij}=W_{ij}x^*$ for all $j\in\cN_i\cup\{i\}$ for $i\in\cN$. 
\eqref{eq:DPGA-I formulation-compact}, equivalent to~\eqref{eq:DPGA-I formulation}, can be written as \vspace*{-1mm}
{\small
\begin{align}
\label{formulation3-I}
\min_{\bx} \left\{\sum_{i \in \cN} \Phi_i( x_i ):\ \sum_{j\in\cN_i\cup\{i\}}W_{ij}x_j = \mathbf{0}  : p_{i} \quad \forall i \in \cN \right\},
\end{align}
}%
where $p_{i}\in\reals^n$ denotes the dual variable corresponding to the primal constraint for $i\in\cN$. Since $\ri(\bigcap_{i\in\cN}\dom \xi_i)\neq\emptyset$, Assumption~\ref{assumption-0} implies the Slater's condition for both \eqref{formulation3-I} and \eqref{eq:DPGA-I formulation};  hence, dual optimal solutions to \eqref{formulation3-I} and \eqref{eq:DPGA-I formulation} exist. Let $\bp^*\triangleq[p^{*}_{i}]_{i\in\cN}\in\reals^{n|\cN|}$ and $\lambda^*\in\reals^{n(2|\cE|+|\cN|)}$ denote some dual solutions to \eqref{formulation3-I} and \eqref{eq:DPGA-I formulation}, respectively. Note \eqref{eq:DPGA-I formulation} is a special case of \eqref{prob}, and \eqref{eq:DPGA-I formulation} satisfies Assumption~\ref{assumption-1}; moreover, $\{\bx^k\}_{k\geq 1}$ generated by DPGA-W in Fig.~\ref{alg:DPGA-I} is the same as with the sequence produced by PG-ADMM applied to \eqref{eq:DPGA-I formulation}. {Note $B=[B_i]_{i\in\cN}$, obtained by vertically concatenating $\{B_i\}_{i\in\cN}$, has full column rank, and $-\sum_{i=1}^NB_i^\top \lambda_i^0\in\partial g(\by^0)$ since $\lambda^0_{ij}=p_i^0=\mathbf{0}$ for $j\in\cN_i\cup\{i\}$ and $i\in\cN$}. Thus, convergence of the DPGA-W iterate sequence $\{\bx^k\}_{k\geq 0}$ 
follows from Theorem~\ref{thm-ergodic} when $c_i<(L_i+\gamma_i \norm{\omega_i}^2)^{-1}$ for $i\in\cN$; moreover, the bound \eqref{eq:PG-key-ineq} given in Theorem~\ref{lem:ergodic} is valid for the DPGA-W iterates $\{\bx^k\}$ whenever $c_i\leq(L_i+\gamma_i \norm{\omega_i}^2)^{-1}$ for $i\in\cN$. In the rest, we analyze the error bounds.

From the 
optimality conditions for \eqref{formulation3-I}, $\bp^*$ is an optimal dual solution to \eqref{formulation3-I} if and only if \vspace*{-1mm}
{\small
\begin{align}
\label{eq:opt-condition-I}
\mathbf{0}\in\partial \Phi_j(x^*)+\sum_{i\in\cN_j\cup\{j\}}W_{ij}p^*_{i},\quad \forall\ j\in\cN,
\end{align}
}%
i.e., $-(W\otimes I_n)^\top \bp^*\in\partial F(\bx^*)$ such that $x^*_i=x^*$ for $i\in\cN$.
Moreover, 
from the first-order optimality conditions for \eqref{eq:DPGA-I formulation}, $\lambda^*$ is dual optimal to \eqref{eq:DPGA-I formulation} if and only if there exists some $\bp=[p_i]_{i\in\cN}$ such that $\lambda^*_{ij}=p_i$ for all $j\in\cN_i\cup\{i\}$ and $i\in\cN$, and $\mathbf{0}\in\partial \Phi_j(x^*)+\sum_{i\in\cN_j\cup\{j\}}W_{ij}\lambda^*_{ij}$ for all $j\in\cN$. Therefore, given an optimal dual solution to \eqref{formulation3-I}, say $\bp^*$ , one can construct a dual optimal $\lambda^*$ to \eqref{eq:DPGA-I formulation} by simply setting $\lambda_{ij}^*=p_{i}^*$ for all $j\in\cN_i\cup\{i\}$ and $i\in\cN$. In the rest, 
we fix $(\bx^*,\by^*,\lambda^*)$ as a primal-dual optimal solution to \eqref{eq:DPGA-I formulation} such that $x_i^*=x^*$ for $i\in\cN$, $y^*_{ij}=W_{ij}x^*$ and $\lambda^*_{ij}=p^*_i$ for $j\in\cN_i\cup\{i\}$ and $i\in\cN$, where $\bp^*$ is some dual optimal solution to \eqref{formulation3-I}.

For $i\in\cN$, $c_i=1/(L_i+\gamma_i \norm{\omega_i}^2+\vartheta_i)$ for some $\vartheta_i\geq 0$, and $A_i=(\mathbf{\omega}_i\otimes I_n)\in\reals^{n (d_i+1)\times n}$; hence,
$A_i^\top A_i = \norm{\omega_i}^2 I_n$ and
$\frac{1}{c_i}\norm{x_i^0-x^*}_{I_n-\gamma_i c_i A_i^\top A_i}=(L_i+\vartheta_i)\norm{x_i^0-x^*}^2$. 
Moreover, since we initialize $\by^0$ such that $y_{ij}^0=W_{ij}x_j^0$, and $y_{ij}^*=W_{ij}x^*$, for all $j\in\cN_i\cup\{i\}$ and $i\in\cN$, it trivially follows that $\norm{\by^*-\by^0}^2_{\sum_{i=1}^N\gamma_i B_i^\top B_i}=\sum_{i\in\cN}\gamma_i\sum_{j\in\cN_i\cup\{i\}}W_{ji}^2\norm{x^*-x_i^0}^2=\sum_{i\in\cN}\gamma_i\norm{\omega_i}^2\norm{x_i^0-x^*}^2$.

Furthermore, Theorem~\ref{lem:ergodic} holds for all $\lambda=[\lambda_i]_{i\in\cN}\in\reals^{n(2|\cE|+|\cN|)}$, where $\lambda_i=[\lambda_{ij}]_{j\in\cN_i\cup\{i\}}\in\reals^{n(d_i+1)}$ for $i\in\cN$. Hence, given $\bp^\top=[p_1^\top,\ldots,p_N^\top]^\top$ for some $p_i\in\reals^n$ for all $i\in\cN$, we set $\lambda_{ij}=p_i$ for all $j\in\cN_i\cup\{i\}$ and $i\in\cN$, and invoke Theorem~\ref{lem:ergodic} for this specific choice of $\lambda$ to obtain the following inequality, by customizing \eqref{eq:PG-key-ineq} for \eqref{eq:DPGA-I formulation}, for all $t\geq 1$: 
{\small
\begin{align}
F(\bar{\bx}^t)  - F^* + \bp^\top (W\otimes I_n)\bar{\bx}^t  &= F(\bar{\bx}^t)  - F^* + \sum_{i \in \cN} \sum_{j \in \cN_i\cup\{i\}} p_i^\top (W_{ij} \bar{x}_j^t  - \bar{y}_{ij}^t) \nonumber\\
&\leq \; \frac{1}{2t} \sum_{i \in \cN} \Big[\tau_i\norm{p_i}^2 + \frac{1}{c_i}\norm{x_{i}^0-x^*}^2\Big], \label{eq:ergodic-I}
\end{align}
}%
where $\tau_i\triangleq \sum_{j\in\cN_i\cup\{i\}}\frac{1}{\gamma_j}$ for $i\in\cN$, and the above inequality follows from the facts that $\sum_{j\in\cN_i\cup\{i\}}\bar{y}_{ij}^t=\mathbf{0}$ and $c_i=(L_i+\gamma_i \norm{\omega_i}^2+\vartheta_i)^{-1}$. Hence, setting $\bp=\mathbf{0}$ in \eqref{eq:ergodic-I} leads to \vspace*{-1mm}
{\small
\begin{align}
\label{eq:upper-subopt-I}
F(\bar{\bx}^t)  - F^* \leq \frac{1}{t} \sum_{i \in \cN} \frac{1}{2c_i}\norm{x_{i}^0-x^*}^2.
\end{align}
}%
From the convexity of $\Phi_i$, \eqref{eq:opt-condition-I} and the fact $(W\otimes I_n)\bx^*=\mathbf{0}$, it follows that \vspace*{-1mm}
{\small
\begin{align}
\label{subgradient-I}
0 \leq F( \bar{\bx}^t ) - F^* + {\bp^*}^\T (W\otimes I_n)(\bar{\bx}^t-\bx^*)=F( \bar{\bx}^t ) - F^* + {\bp^*}^\T (W\otimes I_n)\bar{\bx}^t.
\end{align}
}%
Adding the last term to both sides, and invoking \eqref{eq:ergodic-I} for $\bp$ such that $p_i=2p_i^*$ for $i\in\cN$, we get \vspace*{-1mm}
{\small
\begin{align}
{\bp^*}^\T (W\otimes I_n)\bar{\bx}^t
\leq F( \bar{\bx}^t ) - F^* + (2\bp^*)^\T (W\otimes I_n)\bar{\bx}^t  
\leq \frac{1}{t} \left(2\tau_{\max}\norm{\bp^*}^2 + \sum_{i\in\cN}\frac{1}{2c_i}\norm{x_{i}^0-x^*}^2\right), \label{func_ineq2}
\end{align}
}%
where $\tau_{\max}\triangleq\max_{i\in\cN}\tau_i$. Invoking \eqref{eq:ergodic-I} once again for $ \bp = \bp^* + (W\otimes I_n) \bar{\bx}^t / \norm{ (W\otimes I_n) \bar{\bx}^t }$ and using \eqref{subgradient-I}, we get \vspace*{-1mm}
{\small
\begin{align}
\norm{(W\otimes I_n) \bar{\bx}^t } &\leq F( \bar{\bx}^t ) - F^* + {\bp^*}^\T (W\otimes I_n) \bar{\bx}^t + \norm{(W\otimes I_n) \bar{\bx}^t } \nonumber \\
& \leq \frac{1}{2t} \sum_{i \in \cN} \left(\tau_i\norm{p_i^*+\frac{\sum_{j\in\cN_i\cup\{i\}}W_{ij}\bar{x}_j^t}{\norm{(W\otimes I_n) \bar{\bx}^t}}}^2 + \frac{1}{c_i}\norm{x_{i}^0-x^*}^2\right)\nonumber \\ 
&\leq  \frac{1}{t}\left(\tau_{\max}(1+\norm{\bp^*}^2) + \sum_{i\in\cN}\frac{1}{2c_i}\norm{x^* - x_{i}^0}^2\right), \label{consensus_ineq}
\end{align}
}%
where the last inequality follows from 
$\norm{a+b}^2\leq 2\norm{a}^2+2\norm{b}^2$.

Next, appropriately choosing a dual optimal to \eqref{formulation3-I}, we bound $\norm{\bp^*}$, which appears in \eqref{func_ineq2} and \eqref{consensus_ineq}.
From the definition of $W\in\reals^{N\times N}$, we have $\Rank(W)=N-1$ and $W\succeq \mathbf{0}$ -- it is diagonally dominant. Let $\{\bar{\psi}_i\}_{i=1}^N$ denote the eigenvalues of $W$ such that $\bar{\psi}_1\geq\ldots\bar{\psi}_{N-1}>\bar{\psi}_N=0$. Suppose $\bar{V} \bar{\Sigma} \bar{V}^\top$ represent the eigenvalue decomposition of $(W\otimes I_n)$, where $\bar{V}=[\bar{v}_1 \ldots \bar{v}_{R}]\in\reals^{n |\cN|\times R}$, $\bar{\Sigma}=\diag(\bar{\sigma})$, 
and $R\triangleq\Rank(W\otimes I_n)=n(N-1)$. Note that $\max\{\bar{\sigma}_r:\ r=1,\ldots,R\}=\bar{\psi}_{1}$, and $\min\{\bar{\sigma}_r:\ r=1,\ldots,R\}=\bar{\psi}_{N-1}>0$. Since $(W\otimes I_n)^\T \bp^*$ is in the row space of $(W\otimes I_n)$, there exists $\{ \lambda_r \}_{r =1}^{R}$ such that $-(W\otimes I_n)^\T \bp^* = \sum_{r=1}^R \lambda_r \bar{v}_r$. Following~\cite{makhdoumi2014broadcast}, choose $\bar{\bp} \triangleq \sum_{r=1}^R \bar{v}_r \lambda_r/\bar{\sigma}_r = V\diag(\lambda/\bar{\sigma})$, where $\diag(\lambda/\bar{\sigma})\in\reals^{R\times R}$ is a diagonal matrix with its $r$-th diagonal entry is $\lambda_r/\bar{\sigma}_r$ for $r=1,\ldots,R$. Hence, $\bar{\bp}\in\reals^{n|\cN|}$ satisfies $-(W\otimes I_n)^\T \bar{\bp} = V \bar{\Sigma} V^\T V \diag(\lambda/\bar{\sigma}) = - (W\otimes I_n)^\T \bp^*$; therefore, $-(W\otimes I_n)^\T \bar{\bp}\in\partial F(\bx^*)$ as well. Using the local bounds $\kappa_i$ on the subdifferential of each $\Phi_i$ at $x^*\in\reals^n$, we can bound $\norm{\bar{\bp}}$ from above: 
\vspace*{-1mm}
{\small
\begin{align*}
\norm{\bar{\bp}}^2 &= \sum_{r = 1}^{R} \left(\frac{\lambda_r}{\bar{\sigma}_r}\right)^2 \leq \frac{1}{\bar{\psi}^2_{N-1}} \sum_{r = 1}^{R} \lambda_r^2 = \frac{1}{\sigma_{\min}^2(W)} \norm{(W\otimes I_n)^\top \bp^*}^2.
\end{align*}
}%
Hence, we can conclude that \vspace*{-2mm}
{\small
\begin{align}
\norm{\bar{\bp}}^2 \leq \frac{1}{\sigma^2_{\min}(W)}\sum_{i\in\cN}\kappa_i^2.
\label{theta_bound-I}
\end{align}
}%
Using \eqref{theta_bound-I} within \eqref{func_ineq2} and \eqref{consensus_ineq}, and combining the resulting inequalities together with \eqref{subgradient-I} and \eqref{eq:upper-subopt-I} implies the desired bounds.
\end{proof}

Suppose $W=\Omega$. When $x_i^0=x^0$, $c_i=(L_i+\gamma\norm{\omega_i}^2)^{-1}$ and $\gamma_i=\gamma$ for $i\in\cN$ for some $x^0\in\reals^n$ and $\gamma>0$, the bounds in Theorem~\ref{thm:DPGA-I} can be simplified. Observe $\sum_{i \in \cN} \frac{1}{c_i} \norm{x^*-x^0}^2=(\gamma\norm{\Omega}_F^2+\sum_{i\in\cN}L_i)\norm{x^0-x^*}^2$, and $\tau_{\max}=(d_{\max}+1)/\gamma$. Therefore, for $t\geq 1$, \vspace*{-1mm}
{\small
\begin{eqnarray*}
\max\left\{|F(\bar{\bx}^t)-F^*|,\ \norm{(\Omega\otimes I_n) \bar{\bx}^t }\right\}\leq \frac{1}{2t} \Big[\frac{4(d_{\max}+1)}{\gamma}\left(\sum_{i\in\cN}\frac{\kappa_i^2}{\sigma^2_{\min}(\Omega)}+1\right)+ \Big(\gamma\norm{\Omega}_F^2+\sum_{i \in \cN} L_i \Big)  \norm{x^*-x^0}^2 \Big].
\end{eqnarray*}
}
\subsection{Stochastic gradient variants of DPGA and DPGA-W}
Using Theorem~\ref{lem:ergodic}, one can provide error bounds for the stochastic gradient variants of DPGA and DPGA-W as corollaries of Theorem~\ref{thm:DPGA-II} and Theorem~\ref{thm:DPGA-I}. 
Both SDPGA and SDPGA-W employ SFO $G_i$ defined in Definition~\ref{eq:def:SFO-II} for $i\in\cN$ instead of accessing to $\grad f_i$. Due to space constraints, we only provide the result for SDPGA, that said the bounds for SDPGA-W immediately follows from the same arguments. In Fig.~\ref{alg:DPGA-II} simply replace $\grad f_i(x_i^k)$ with $G_i(x_i^k,\nu_i^k)$ and set the stepsize at the $k$-th iteration as $c_i^k=(L_i+\gamma_i d_i+1+\sqrt{k})^{-1}$. Then under Assumption~\ref{assumption-2}, slightly modifying the proof of Theorem~\ref{thm:DPGA-II} by invoking the result of Theorem~\ref{lem:ergodic} for the case $\sigma>0$, we immediately obtain the bounds for SDPGA in Corollary~\ref{cor:SDPGA-II}.
\begin{corollary}
\label{cor:SDPGA-II}
Suppose a solution to \eqref{eq:central_opt}, $x^*\in\reals^n$, exists, {$\ri(\cap_{i\in\cN}\dom \xi_i)\neq\emptyset$}, and $D \triangleq \max_{i \in \cN} \sup_{x, x' \in \dom(\xi_i)  }  \norm{x - x'}<+\infty$. Given arbitrary $\{x^0_i,\gamma_i\}_{i\in\cN}$, when $\{c_i^k\}_{k\geq 0}$ is chosen such that $\frac{1}{c_i^k}=\frac{1}{c_i}+\sqrt{k}$ and $c_i=(L_i + \gamma_i\norm{A_i}^2+1)^{-1}$, the SDPGA average sequence $\{\bar{\bx}^t\}_{t\geq 1}$ satisfies the following 
bounds for $\bar{D}=D$ and all $t\geq 1$, \vspace*{-1mm}
{\small
\begin{subequations}\label{coreq:sdpga-bounds}
\begin{align}
\mathbb{E}\left[ |F(\bar{\bx}^t)-F^*|\right]\leq \frac{1}{t}\left( \frac{2\norm{Q}}{\sigma_{\min}(\Omega)}\sum_{i\in\cN}\kappa_i^2 + \sum_{i \in \cN} \frac{1}{2c_i} \norm{x^* - x_{i}^0  }^2\right)+\frac{N}{2\sqrt{t}}(\bar{D}^2+2\sigma^2),\\ 
\mathbb{E}\left[\Big(\sum_{(i,j)\in\cE}\norm{\bar{x}_i^t-\bar{x}^t_j}^2\Big)^{\tfrac{1}{2}}\right]\leq \frac{1}{t}\left[\norm{Q}\left(\sum_{i\in\cN}\frac{\kappa_i^2}{\sigma_{\min}(\Omega)}+1\right)+ \sum_{i \in \cN} \frac{1}{2c_i} \norm{x^*-x_i^0}^2\right]+\frac{N}{2\sqrt{t}}(\bar{D}^2+2\sigma^2),
\end{align}
\end{subequations}}%
where $\kappa_i>0$ denotes a bound on the elements of $\partial \Phi_i(x^*)$, i.e., if $q\in\partial \Phi_i(x^*)$, then $\norm{q}\leq \kappa_i$, for each $i\in\cN$, and {$Q\triangleq\diag([\frac{1}{\gamma_i}+\frac{1}{\gamma_j}]_{(i,j)\in\cE})$.}
\end{corollary}
Note that even if $D=\infty$, one can still achieve $\cO(1/\sqrt{t})$ rate in case $D^*(\bx^0)<\infty$ by using constant stepsize. Indeed, as in Corollary~\ref{cor:constant-step}, for any $t\geq 1$, let $\{c_i^k\}_{0\leq k\leq t}$ be chosen such that $\frac{1}{c_i^k}=\frac{1}{c_i}+\sqrt{t}$ and $c_i=(L_i + \gamma_id_i+1)^{-1}$ for $i\in\cN$, \eqref{coreq:sdpga-bounds} holds for $\bar{D}=D^*(\bx^0)$.
\subsection{Adaptive step-size strategy}
\label{sec:adaptive-step}
One important property of DPGA methods is their ability to adopt an adaptive step-size sequence for each node. Note $L_i$, Lipschitz constants of $\grad f_i$, may not be known in advance or may be too large for certain nodes -- leading to very small steps since $c_i=\cO(L_i^{-1})$, i.e., $c_i\leq(L_i+\gamma_i d_i)^{-1}$ for DPGA and $c_i\leq(L_i+\gamma_i \norm{\omega_i}^2)^{-1}$ for DPGA-W -- see Fig.~\ref{alg:DPGA-II} and Fig.~\ref{alg:DPGA-I}. On the other hand, it is elementary to check that all the proofs given above still go through if node $i\in\cN$ uses the step size $c_i^k\leq(L_i^k+\gamma_i d_i)^{-1}$ for DPGA and $c_i^k\leq(L_i^k+\gamma_i \norm{\omega_i}^2)^{-1}$ for DPGA-W at the $k$-th iteration such that the following inequality holds: \vspace*{-1mm}
{\small
\begin{align}
\label{eq:adaptive-check}
f_i(x_i^{k+1})\leq f_i(x_i^{k})+\fprod{\grad f_i(x_i^{k}), \Delta_i^k}+\frac{L_i^k}{2}\norm{\Delta_i^k}^2,
\end{align}
}%
where $\Delta_i^k\triangleq x_i^{k+1}-x_i^{k}$ and $x_i^{k+1}$ is computed using $c_i^k$ instead of $c_i$. Clearly, $L_i^k\leq L_i$. Since this condition can be checked locally, one can possibly take longer steps compared to $c_i$ and still has a convergence guarantee. In contrast, distributed algorithms that use constant step size $c>0$ for all nodes, e.g., PG-EXTRA~\cite{shi2015proximal}, cannot take advantage of this trick. We adopted the following rule in our numerical tests: let $\upsilon>1$, for $k\geq 1$ we set $L_i^{k}=L_i^{k-1} \upsilon^{\ell_k-1}$ where $\ell_k\geq 0$ is the smallest integer such that \eqref{eq:adaptive-check} holds, and $L_i^{0}=L_i$ for $i\in\cN$. {Note that \eqref{eq:adaptive-check} is usually called the descent lemma. In proximal gradient methods and ADMM, it is now a common practice to perform backtracking on $L_i^k$ such that \eqref{eq:adaptive-check} holds. For more details on adaptive step size using backtracking, we refer the interested readers to \cite{Beck09} and \cite{Goldfarb-Scheinberg-fastlinesearch2011}.}
\vspace*{-3mm}
\section{Numerical results}
\label{sec:numerical}
We compared DPGA with PG-EXTRA, distributed ADMM and 
its variant proposed in~\cite{shi2015proximal}, \cite{makhdoumi2014broadcast} and~\cite{icml2015_aybat15}, respectively, on the \emph{sparse group LASSO} problem with Huber loss:
{
\begin{equation}
\min_{x \in \reals^n} \sum_{i\in\cN}\left[ \beta_1 \norm{x}_1 + \beta_2 \norm{x}_{G_i}+h_{\delta}\left( A_i x- b_i \right)\right], \label{eq:problem} \vspace{-0.15cm}
\end{equation}
}%
where $\beta_1, \beta_2>0$, for each $i\in\cN$, $A_i\in\reals^{m_i\times n}$, $b_i\in\reals^{m_i}$, and $\norm{x}_{G_i} \triangleq \sum_{k=1}^{K}\norm{x_{g_i(k)}}_2$ denotes the group norm with respect to the partition $G_i$ of $[1,n]\triangleq\{1,\cdots,n\}$, i.e., $G_i=\{g_i(k)\}_{k=1}^{K}$ such that $\bigcup_{k=1}^{K}g_i(k)=[1,n]$, and $g_i(j)\cap g_i(k) =\emptyset$ for all $j \neq k$; and $h_{\delta}$ denotes the Huber loss function, i.e., for any $m\geq 1$ let $h_{\delta}:\reals^m\rightarrow\reals$ such that $h_{\delta}(y) \triangleq \max\{z^\top y-\tfrac{1}{2}\norm{z}^2:\ \norm{z}_\infty\leq \delta\}$. In this case, $f_i(x)\triangleq h_{\delta}\left( A_i x- b_i \right)$ and $\xi_i(x)\triangleq\beta_1\norm{x}_1+\beta_2\norm{x}_{G_i}$.

Next, we briefly describe the competitive algorithms: PG-EXTRA~\cite{shi2015proximal}, the distributed ADMM algorithm in~\cite{makhdoumi2014broadcast}; and a more efficient variant of the ADMM that exploits the problem structure in \eqref{eq:problem}.
Recall that $\Omega\in\reals^{N\times N}$ denotes the Laplacian of the graph $\cG=(\cN,\cE)$.
\subsection{Distributed ADMM Algorithms and PG-EXTRA}
As discussed in Section~\ref{sec:dpga1}, \eqref{eq:dist_opt} can be equivalently written as in \eqref{eq:DPGA-I formulation}.
Makhdoumi \& Ozdaglar~\cite{makhdoumi2014broadcast} establish that when an ADMM algorithm with penalty parameter $\gamma>0$ is implemented on \eqref{eq:DPGA-I formulation}, 
the subproblems can be simplified as shown in Figure~\ref{fig:admm}.
It is shown in~\cite{makhdoumi2014broadcast} that suboptimality and consensus violation converge to $0$ with a rate $\cO(1/k)$, and in each iteration every node communicates $2n$ scalars, i.e., $x_i\in\reals^n$ and $p_i+s_i\in\reals^n$. Moreover, each node stores $3n$ scalars at each iteration, i.e., $x_i,s_i,p_i\in\reals^n$. Note that ADMM in Fig.~\ref{fig:admm} is a special case of DPGA-W in Fig.~\ref{alg:DPGA-I}. Indeed, DPGA-W iterations reduces to ADMM when $\gamma_i=\gamma$ for all $i\in\cN$ for some $\gamma>0$, and when we set $\xi_i\gets\Phi_i$ and $f_i\gets 0$ in DPGA-W, i.e., when we treat $\Phi_i=\xi_i+f_i$ as the non-smooth component in DPGA-W.

\begin{figure}[htbp]
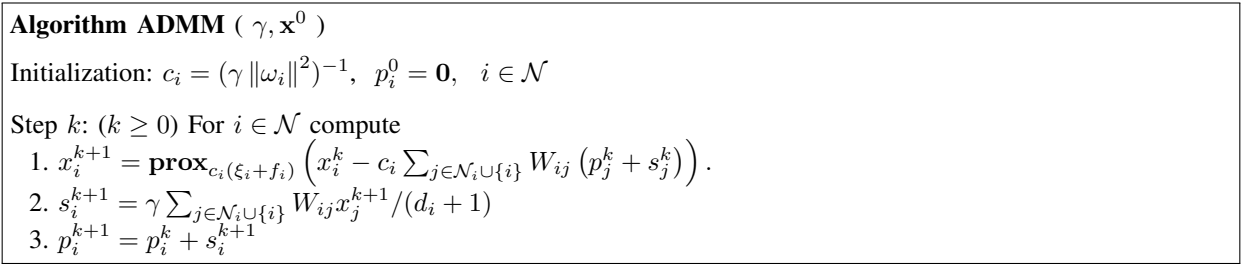

\centering
\framebox{\parbox{\columnwidth-8pt}{
{\small
\textbf{Algorithm ADMM} ( $\gamma,\bx^{0}$ ) \\[1.5mm]
Initialization: $c_i=(\gamma \norm{\omega_i}^2)^{-1}$,\ \ $p_i^{0} = \mathbf{0},\ \ i\in\cN$\\[2mm]
Step $k$: ($k \geq 0$) For $i\in\cN$ compute\\
\text{ } 1. $x_i^{k+1}=  \prox{c_i(\xi_i + f_i)}  \left( x_i^{k} - c_i \sum_{j \in \cN_i\cup\{i\}}  W_{ij} \left(p_j^{k}+s_j^{k}\right)  \right).$ \\
\text{ } 2. $s_i^{k+1}=\gamma \sum_{j\in\cN_i\cup\{i\}}W_{ij}x_j^{k+1}/(d_i+1)$\\
\text{ } 3. $p_i^{k+1}=p_i^{k}+s_i^{k+1}$
 }}}
\caption{\small ADMM algorithm}
\vspace*{-0.25cm}
\label{fig:admm}
\end{figure}
From now on, we refer to this algorithm that directly works with $\Phi_i=\xi_i + f_i$ as \admm -- see~Fig.~\ref{fig:admm}. Computing  $\prox{\Phi_i}$ for each $i \in \cN$ is the computational bottleneck in each iteration of \admm. Note that computing $\prox{\Phi_i}$ for \eqref{eq:problem} is almost as \emph{hard} as solving the problem. To deal with this issue, Aybat et al.~\cite{icml2015_aybat15} considered the following reformulation: \vspace*{-1mm}
{\small
\begin{align*}
\min_{\substack{x_i,y_i \in \mathbb{R}^n,\\ z_i,\tilde{z}_i\in\cZ_i}}
\left\{\sum_{i\in\cN}\xi_i(x_i)+f_i(y_i):\
\begin{array}{ll}
W_{ij}x_j=z_{ij},\quad &W_{ij}y_j=\tilde{z}_{ij},\quad j\in\cN_i\cup\{i\},\ i\in\cN \\
x_i=q_i,\ &y_i=q_i,\quad i\in\cN,
\end{array}
\right\}
\end{align*}
}%
where $\cZ_i\triangleq\{z_i=[z_{ij}]_{j\in\cN_i\cup\{i\}}:\ \sum_{j\in\cN_i\cup\{i\}}z_{ij}=0\}$, and proposed a split ADMM algorithm (SADMM). 
Steps of \sadmm~can be derived by minimizing the augmented Lagrangian alternatingly in $(\bx,\by)$, and in $(\bz,\mathbf{\tilde{z}},\bq)$ while fixing the other. As in~\cite{makhdoumi2014broadcast}, computing $(\bz,\mathbf{\tilde{z}},\bq)$ can be avoided by exploiting the structure of optimality conditions. Convergence of SADMM with $\cO(1/k)$ rate follows immediately from the results on the convergence of ADMM~\cite{Yuan11_2J}. In each iteration of SADMM, every node communicates $4n$ scalars, i.e., $x_i\in\reals^n$, $y_i\in\reals^n$, $p_i+ s_i\in\reals^n$ and $\tilde{p}_i+\tilde{s}_i\in\reals^n$. Moreover, each node stores $7n$ scalars, i.e., $x_i,y_i,s_i,\tilde{s}_i,p_i,\tilde{p}_i,r_i\in\reals^n$.

Given two doubly stochastic, symmetric mixing matrices, $W=[W_{ij}]$, $\tilde{W}=[\tilde{W}_{ij}]\in\reals^{|\cN|\times |\cN|}$ satisfying Assumption~1 in~\cite{shi2015proximal}, Shi et al. show that PG-EXTRA, displayed in Fig.~\ref{fig:pgextra}, can solve \eqref{eq:central_opt} with $\Phi_i=\xi_i+ f_i$ as in \eqref{eq:F_i} ($\xi_i$ and $f_i$ convex and $\grad f_i$ is Lipchitz continuous with constant $L_i$) in a distributed fashion with $\cO(1/k)$ rates on sub-optimality and consensus violation in terms of \emph{squared residuals} of KKT violation and consensus violation, respectively.
In our experiments, the mixing matrices are chosen as $W \triangleq I - \frac{\Omega}{d_{\max}+1}$ and $\tilde{W} \triangleq \frac{I + W}{2} = I - \frac{\Omega}{2(d_{\max+1})}$ (see Section 2.3 in~\cite{shi2015extra}). According to \cite{shi2015proximal}, $\bx^*=[x_i^*]_{i\in\cN}$ is an optimal solution to \eqref{eq:dist_opt} with $\Phi_i=\xi_i+ f_i$ as in \eqref{eq:F_i} if and only if there exist $\bq^*\in\reals^{n|\cN|}$ and $\bg^*=[g_i^*]_{i\in\cN}\in\reals^{n|\cN|}$ such that $\bq^*=(U\otimes I_n)\bp$ for some $\bp\in\reals^{n|\cN|}$ and $g_i^*\in\partial \xi_i(x_i^*)$ satisfying $(U\otimes I_n)\bx^*=\mathbf{0}$ and $(U\otimes I_n)\bq^*+c (\bg^*+\grad f(\bx^*))=\mathbf{0}$, where $c>0$ is the given step size parameter, $U\triangleq(\tilde{W}-W)^{1/2}$, and $\grad f(\bx)\triangleq[\grad f_i(x_i)]_{i\in\cN}\in\reals^{n|\cN|}$ for any $\bx=[x_i]_{i\in\cN}$. Assumption~1 in~\cite{shi2015proximal} implies that null space of $U$ only contains $\mathbf{1}\in\reals^{|\cN|}$; therefore, $(U\otimes I_n)\bx=\mathbf{0}$ implies $x_1=\ldots=x_N$.

Let $\{\bx^k\}$ be the PG-EXTRA iterate sequence generated as in Fig.~\ref{fig:pgextra}, and sequence $\{\bq^k\}$ be defined as $\bq^k=\sum_{t=0}^k(U\otimes I_n)\bx^t$. According to Theorem 1 and Theorem~2 in \cite{shi2015proximal}, when the step size $c\in\Big(0 , \frac{2\lambda_{\min}(\tilde{W})}{L_{\max}} \Big)$, where $L_{\max}=\max_{i\in\cN}L_i$, then $\{\bx^k\}$ satisfies \vspace*{-2mm}
{\small
\begin{align}
\label{eq:pgextra-rate}
\frac{1}{t}\sum_{k=0}^t\norm{(U\otimes I_n)\bq^k+c(\grad f(\bx^k)+\bg^{k+1})}_{\tilde{W}}^2=\cO\left(\frac{1}{t}\right),\quad \frac{1}{t}\sum_{k=0}^t\norm{(U\otimes I_n)\bx^k}^2=\cO\left(\frac{1}{t}\right), \vspace{-4mm}
\end{align}
}
where $g_i^{k+1}\in\partial \xi_i(x_i^{k+1})$. As also pointed out in the introduction, we consider this rate result as $\cO(1/\sqrt{t})$ because \eqref{eq:pgextra-rate} can only guarantee $\norm{(U\otimes I_n)\bar{\bx}^t}=\cO(1/\sqrt{t})$, where $\bar{\bx}^t\triangleq\sum_{k=1}^t\bx^k/t$. On the other hand, according to Theorems~\ref{thm:DPGA-II} and \ref{thm:DPGA-I}, DPGA and DPGA-W iterate sequences satisfy  $(\sum_{(ij)\in\cE}\norm{\bar{x}_i^t-\bar{x}_i^t}^2)^{1/2}=\cO(1/t)$ and $\norm{(W\otimes I_n)\bar{\bx}^t}=\cO(1/t)$, respectively.

Adopting mixing matrices $(W,\tilde{W})$ to be able to set the stepsize $c>0$ independent of the global topology of $\cG=(\cN,\cE)$ may still require certain parameters, determined by the global topology of $\cG$, to be in the common knowledge of all nodes $\cN$. In particular, $W = I - \frac{\Omega}{2d_{\max}}\succeq\mathbf{0}$ and $\tilde{W}=\frac{I+W}{2}\succeq\frac{1}{2}I$; hence, $c$ can be chosen $c\in(0,\frac{1}{L_{\max}})$, which is independent of the global topology, and only depends on $L_{\max}$; however, all the nodes need to know $d_{\max}$, and $L_{\max}$ which can be computed using some max-consensus algorithm. This assumption may not be attainable for very large scale fully distributed networks, and computing parameters such as $L_{\max}$ and $d_{\max}$ may violate the privacy requirements of the nodes. Also note that since the stepsize $c>0$ is the same for all nodes, PG-EXTRA cannot take advantage of the adaptive step size strategy discussed in Section~\ref{sec:adaptive-step}.
\begin{figure}[htpb]
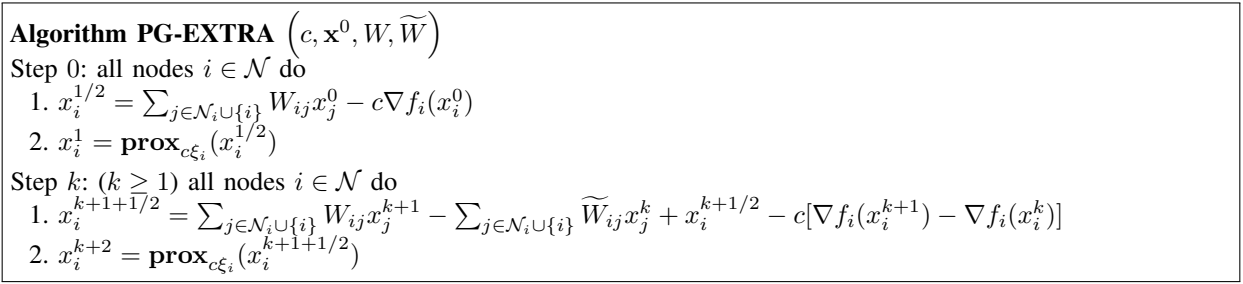

\centering
\framebox{\parbox{\columnwidth-8pt}{\small
\textbf{Algorithm PG-EXTRA} $\left( c, \bx^0, W, \widetilde{W}\right)$ \\
Step $0$: all nodes $i\in\cN$ do\\
\text{ } 1. $x_{i}^{1/2} = \sum_{j\in\cN_i\cup\{i\}} W_{ij} x_{j}^0 - c \nabla f_i( x_{i}^0 ) $ \\
\text{ } 2. $x_{i}^{1} = \prox{c \xi_i}( x_{i}^{1/2} ) $ \\[1mm]
Step $k$: ($k \geq 1$) all nodes $i\in\cN$ do \\
\text{ } 1. $x_{i}^{k+1+1/2} = \sum_{j\in\cN_i\cup\{i\}} W_{ij} x_{j}^{k+1} - \sum_{j\in\cN_i\cup\{i\}} \widetilde{W}_{ij} x_{j}^{k}+ x_{i}^{k+1/2} - c [ \nabla f_i( x_{i}^{k+1} ) - \nabla f_i( x_{i}^k ) ]$ \\
\text{ } 2. $x_{i}^{k+2} = \prox{c \xi_i}( x_{i}^{k+1+1/2} ) $
}}
\caption{\footnotesize Proximal Gradient Exact First-order Algorithm (PG-EXTRA) }
\vspace*{-1cm}
\label{fig:pgextra}
\end{figure}
\subsection{Implementation details and numerical results}
\label{sec:experimental-setting}
In Lemma~\ref{lem:prox}, we show that $\prox{\xi_i}$ can be computed in closed form. On the other hand, when \admm, and \sadmm~are implemented on \eqref{eq:problem}, one needs to compute $\prox{\Phi_i}$ and $\prox{f_i}$, respectively; and these proximal operations do not assume closed form solutions. To be fair, we computed them using an efficient interior point solver MOSEK~(ver. 7.1.0.12).
\begin{lemma}
\label{lem:prox}
Let $\xi(x) = \beta_1 \| x \|_1 + \beta_2 \| x \|_G$. For $t>0$ and $\bar{x}\in\reals^{n}$, $x^p=\prox{t\xi}(\bar{x})$ is given by $ x^{p}_{g(k)} = \eta'_{g(k)} \max \left\{ 1 - \frac{t\beta_2}{\norm{\eta_{g(k)}}_2},~ 0  \right\}$ for $1\leq k\leq K$, where $\eta' = \sgn(\bar{x}) \odot \max \{ |\bar{x}| - t\beta_1, 0 \}$.
\end{lemma}
\begin{proof}
This result is shown in \cite{icml2015_aybat15}.
\end{proof}
\begin{table*}[t!]
\centering
\caption{\small Comparison of \textbf{DPGA, PG-EXTRA}, 
\admm, and \sadmm~ (Termination time T=1800 sec)}
\resizebox{0.95\textwidth}{!}{
\label{tab:1}
\begin{tabular}{cl|cc|cc|cc|cc|c}
\toprule
Size  & Alg. & \multicolumn{2}{c|}{Rel. Suboptimality}  & \multicolumn{2}{c|}{Consensus Violation~(V)} & \multicolumn{2}{c|}{Walltime (sec.)} & \multicolumn{2}{c|}{\# of communication rounds} & Solved? \\
\midrule
  & & Case 1 & Case 2 & Case 1 & Case 2 & Case 1 & Case 2 & Case 1 & Case 2& \\ \midrule
&SDPT3 & 0  & 0  & 0  & 0  & 26  & 74  & N/A  & N/A& \\
&FISTA &1E-3&N/A&0&N/A&10&N/A&2173&N/A& \\
$N=5$ & DPGA~(CS) & 1E-3, 1E-3 & 1E-3, 1E-3 & 9E-5, 2E-5 & 9E-5, 1E-5 & 27, 28 & 28, 30 & 7596, 7597 & 7829, 7804 & yes\\
$n_g=100$
& DPGA~(AS) & 1E-3, 1E-3 & 1E-3, 1E-3 & 4E-5, 4E-5 & 4E-5, 3E-5 & \textbf{11, 11} & \textbf{11, 11} & \textbf{2926, 2906} & \textbf{3021, 2976} & yes\\
& PG-EXTRA & 1E-3, 1E-3 & 1E-3, 1E-3 & 6E-8, 2E-8 & 5E-8, 2E-8 & 43, 46 & 46, 49 & 25246, 25244 & 25948, 25946 & yes\\
& ADMM & 4E-5, 3E-5 & 3E-5, 3E-5 & 1E-4, 1E-4 & 1E-4, 1E-4 & 996, 712 & 987, 709 & 618, 442 & 638, 458 & yes\\
& SADMM & 1E-4, 2E-4 & 1E-4, 2E-4 & 1E-4, 1E-4 & 1E-4, 1E-4 & 655, 713 & 668, 741 & 1984, 2188 & 2020, 2268 & yes\\
\midrule
&SDPT3 & 0  & 0 & 0  & 0  & 26  & 82  & N/A  & N/A  & \\
&FISTA &1E-3&N/A&0&N/A&10&N/A&2173&N/A& \\
 $N=10$ & DPGA~(CS) & 4E-4, 1E-3 & 4E-4, 1E-3 & 1E-4, 2E-5 & 1E-4, 2E-5 & 90, 84 & 94, 88 & 15479, 12281 & 15717, 12622 & yes\\
$n_g=100$
& DPGA~(AS) & 1E-3, 1E-3 & 1E-3, 1E-3 & 3E-5, 3E-5 & 3E-5, 3E-5 & \textbf{30, 34} & \textbf{32, 34} & \textbf{4834, 4790} & \textbf{5015, 4926} & yes \\
& PG-EXTRA & 1E-3, 1E-3 & 1E-3, 1E-3 & 8E-8, 1E-8 & 6E-8, 9E-9 & 123, 150 & 131, 158 & 43346, 43340 & 44628, 44624 & yes\\
& ADMM & 9E-3, 2E-2 & 8E-3, 2E-2 & 6E-4, 2E-4 & 6E-4, 2E-4 & T, T & T, T & 696, 690 & 734, 726 & no\\
& SADMM & 2E-4, 7E-3 & 2E-4, 7E-3 & 1E-4, 7E-4 & 1E-4, 7E-4 & T, T & T, T & 3368, 3460 & 3404, 3452 & no\\
\midrule
& SDPT3 & 0  & 0  & 0  & 0  & 691  & 1381  & N/A  & N/A  \\
&FISTA&1E-3&N/A&0&N/A&253&N/A&8663&N/A & \\
$N=5$ & DPGA~(CS) & 1E-3, 1E-3 & 1E-3, 1E-3 & 7E-5, 8E-6 & 7E-5, 9E-6 & 131, 137 & 136, 141 & 11274, 11336 & 11419, 11482 & yes\\
$n_g=300$
& DPGA~(AS) & 1E-3, 1E-3 & 1E-3, 1E-3 & 3E-5, 2E-5 & 3E-5, 3E-5 & \textbf{65, 66} & \textbf{62, 63} & \textbf{4268, 4242} & \textbf{4064, 4028} & yes \\
& PG-EXTRA & 1E-3, 1E-3 & 1E-3, 1E-3 & 5E-8, 2E-8 & 5E-8, 2E-8 & 181, 190 & 186, 197 & 32644, 32642 & 33018, 33016 & yes \\
& ADMM & 4E-2, 9E-4 & 4E-2, 8E-4 & 4E-3, 1E-3 & 5E-3, 1E-3 & T, T & T, T & 230, 244 & 228, 242 & no\\
& SADMM & 5E-3, 2E-2 & 5E-3, 2E-2 & 2E-3, 3E-3 & 2E-3, 3E-3 & T, T & T, T & 1080, 1108 & 1080, 1100 & no\\
\midrule
& SDPT3 & 0  & 0  & 0  & 0  & 706  & 1463  & N/A  & N/A  & \\
&FISTA &1E-3&N/A&0&N/A&253&N/A&8663&N/A & \\
$N=10$ & DPGA~(CS) & 9E-4, 1E-3 & 9E-4, 1E-3 & 7E-5, 7E-6 & 7E-5, 7E-6 & 246, 286 & 255, 294 & 18874, 18673 & 19124, 18831 & yes\\
$n_g=300$
& DPGA~(AS) & 1E-3, 1E-3 & 1E-3, 1E-3 & 2E-5, 2E-5 & 2E-5, 2E-5 & \textbf{113, 123} & \textbf{111, 122} & \textbf{7128, 7066} & \textbf{6716, 6641} & yes\\
& PG-EXTRA & 1E-3, 1E-3 & 1E-3, 1E-3 & 5E-8, 8E-9 & 6E-8, 8E-9 & 351, 424 & 361, 435 & 56492, 56488 & 56934, 56930 & yes\\
& ADMM & 5E-2, 7E-2 & 5E-2, 7E-2 & 6E-3, 6E-3 & 6E-3, 6E-3 & T, T & T, T & 236, 246 & 232, 242 & no\\
& SADMM & 5E-2, 6E-1 & 5E-2, 7E-1 & 4E-3, 2E-2 & 4E-3, 2E-2 & T, T & T, T & 1020, 1060 & 1020, 1064 & no \\
\bottomrule
\end{tabular}
}
\end{table*}
In our experiments, the network was either a \emph{star tree} or a \emph{clique} with either $5$ or $10$ nodes. The remaining problem parameters defining $\{\xi_i,f_i\}_{i\in\cN}$ were set as follows. We set $\beta_1=\beta_2=\frac{1}{N}$, $\delta=1$, and $K=10$. Let $n=K n_g$ for $n_g\in\{100,300\}$, i.e., $n\in\{1000, 3000\}$. We generated partitions $\{G_i\}_{i\in\cN}$ in two different ways. For test problems in \textbf{CASE 1}, we created a single partition $G=\{g(k)\}_{k=1}^K$ by generating $K$ groups uniformly at random such that $|g(k)| = n_g$ for all $k$; and set $G_i=G$ for all $i\in\cN$, i.e., $\xi_i(x)\triangleq\beta_1\norm{x}_1+\beta_2\norm{x}_G$ for all $i\in\cN$. For the test problems in \textbf{CASE 2}, we created a different partition $G_i$ for each node $i$, in the same manner as in \textbf{Case 1}. For all $i\in\cN$, $m_i = \frac{n}{2N}$, $\reals^{m_i}\ni b_i = A_i \bar{x}$ for $\bar{x}_j = (-1)^j e^{-(j-1)/n_g}$ for $j\in[1,n]$, and $A_i\in\reals^{m_i\times n}$ is set as $A_i=0.5^{\pi_i}\bar{A}_i$, where the elements of $\bar{A}_i\in\reals^{m_i\times n}$ are i.i.d. with standard Gaussian, and $\{\pi_i\}_{i\in\cN}$ are i.i.d. Bernoulli random variables with success probability $\tfrac{1}{2}$. Our choice of $\{A_i\}_{i\in\cN}$ 
leads to a significant deviation among $\{L_i\}$, i.e., $\max_{i\in\cN}L_i/\min_{i\in\cN}L_i\approx 4$ since $L_i=\norm{A_i}^2$. This type of setting is expected to \emph{adversely} affect constant step algorithms, e.g., $c=\cO(1/L_{\max})$ for PG-EXTRA. For all the algorithms, we initialize the iterate sequence from the origin. For ADMM methods, the penalty parameter was chosen specifically for each problem setup by searching for the best penalty over a line segment where the total number of ADMM iterations to terminate exhibits a convex behavior -- similar to Section~4.2.1 in \cite{Aybat15_2J}.

We solved the distributed optimization problem \eqref{eq:dist_opt} using \textbf{DPGA, PG-EXTRA}, \admm, and \sadmm~for both cases, on both star trees, and cliques, and for $N\in\{5,10\}$ and $n_g\in\{100,300\}$.  For each problem setting, we randomly generated 5 instances.
Note that for \textbf{Case 1}, $\sum_{i\in\cN}\xi_i(x)=\norm{x}_1+\norm{x}_G$ and its prox map can be computed efficiently, while for \textbf{Case 2}, $\sum_{i\in\cN}\xi_i(x)$ does not assume a simple prox map. Therefore, for \textbf{Case 1} we were also able to use FISTA~\cite{Beck09,nesterov2013gradient,Tseng08} to solve the central problem \eqref{eq:problem} by exploiting the result of Lemma~\ref{lem:prox}. All the algorithms are terminated when the relative suboptimality, $|F^{k}-F^*|/|F^*|$, is less than $10^{-3}$, and consensus violation, $\mathrm{V}^{k}$, is less than $10^{-4}$, where $F^{k}$ equals $\sum_{i\in\cN}\Phi_i(x_i^{k})$ for \textbf{DPGA, PG-EXTRA} and \admm, and to $\sum\limits_{i\in\cN}\Phi_i\left(\frac{x_i^{k}+y_i^{k}}{2}\right)$ for \sadmm; $\mathrm{V}^{k}$ equals to $\max_{(ij)\in\cE}\norm{x_i^{k}-x_j^{k}}_2/\sqrt{n}$ for \textbf{DPGA, PG-EXTRA}, and \admm, and to $\max\{\max_{(ij)\in\cE}\norm{x_i^{k}-x_j^{k}}_2,~\max_{i\in\cN}\norm{x_i^{k}-y_i^{k}}_2\}/\sqrt{n}$ for \sadmm. If the stopping criteria are not satisfied in $T=1800$ seconds (30 min.), we terminated the algorithm and report the statistics corresponding to the iterate at the termination.

We solved the \emph{central} problem \eqref{eq:problem} with SDPT3 and FISTA for benchmarking. We run DPGA on the \emph{decentralized} problem both with constant step and adaptive step rules - see Section~\ref{sec:adaptive-step}. In Table~\ref{tab:1}, '(CS)' and '(AS)' stand for
constant step and adaptive step rules, respectively. We used PG-EXTRA, ADMM, and SADMM with suggested parameters. 
For the results separated by comma, the left and right ones are for the star tree and clique, respectively. Table~\ref{tab:1} displays the means over 5 replications for each case. 
Table~\ref{tab:1} shows that DPGA and PG-EXTRA finish the jobs much faster than ADMM and SADMM -- as expected due to not so simple $\prox{\Phi_i}$ and $\prox{f_i}$ operations required for ADMM and SADMM, respectively. PG-EXTRA
runs slower than DPGA, mainly because it uses a stepsize that is the same for all the nodes. Moreover, adaptive step-size strategy worked very well in our tests, and it lead to speedup for DPGA by a factor of at least 2 when compared to constant step-size strategy. It is worth mentioning that run-times reported do not include the effect of communication. However, in real life, transmitting information also takes time. The number of communication rounds per iteration are 1 for DPGA, 2 for PG-EXTRA, 2 for ADMM, and 4 for SADMM - see Table~\ref{tab:str_commu}. Thus, we expect the result to be more in favor of DPGA as the communication time is also taken into consideration when implemented in real networks.
\subsection{Numerical Tests on the Effect of Network Topology and Noisy Gradients}
In this section, we study the effect of network topology on the convergence of DPGA -- we used constant step version, i.e., DPGA (CS). 
In the experiment, we have three types of network topologies, circle, small-world and complete graph. 
 Circle is constructed by forming a cycle connecting all the nodes; the small-world networks are constructed by adding random edges after forming a cycle~\cite{Ling15_1J}. 
Both the problem setting and the stopping criteria are the same with those in CASE 1 of Section~\ref{sec:experimental-setting}.
According to discussion at the end of Section~\ref{sec:dpga-errorbounds}, $\gamma^*=\cO(\sqrt{\frac{|\cN|}{|\cE|~\sigma_{\min}(\Omega)}})$ minimizes the error bounds; hence, the penalty parameter $\gamma$ was chosen as $\sqrt{\frac{c|\cN|}{ |\cE| \min_{i \in \cN} d_i}}$, where $c$ is set to $2.6$. 
This empirical rule has worked fairly well in our tests.

\begin{figure}[htbp]
\centering
	\begin{subfigure}[b]{0.45\textwidth}
        \centering
        \includegraphics[width = \textwidth]{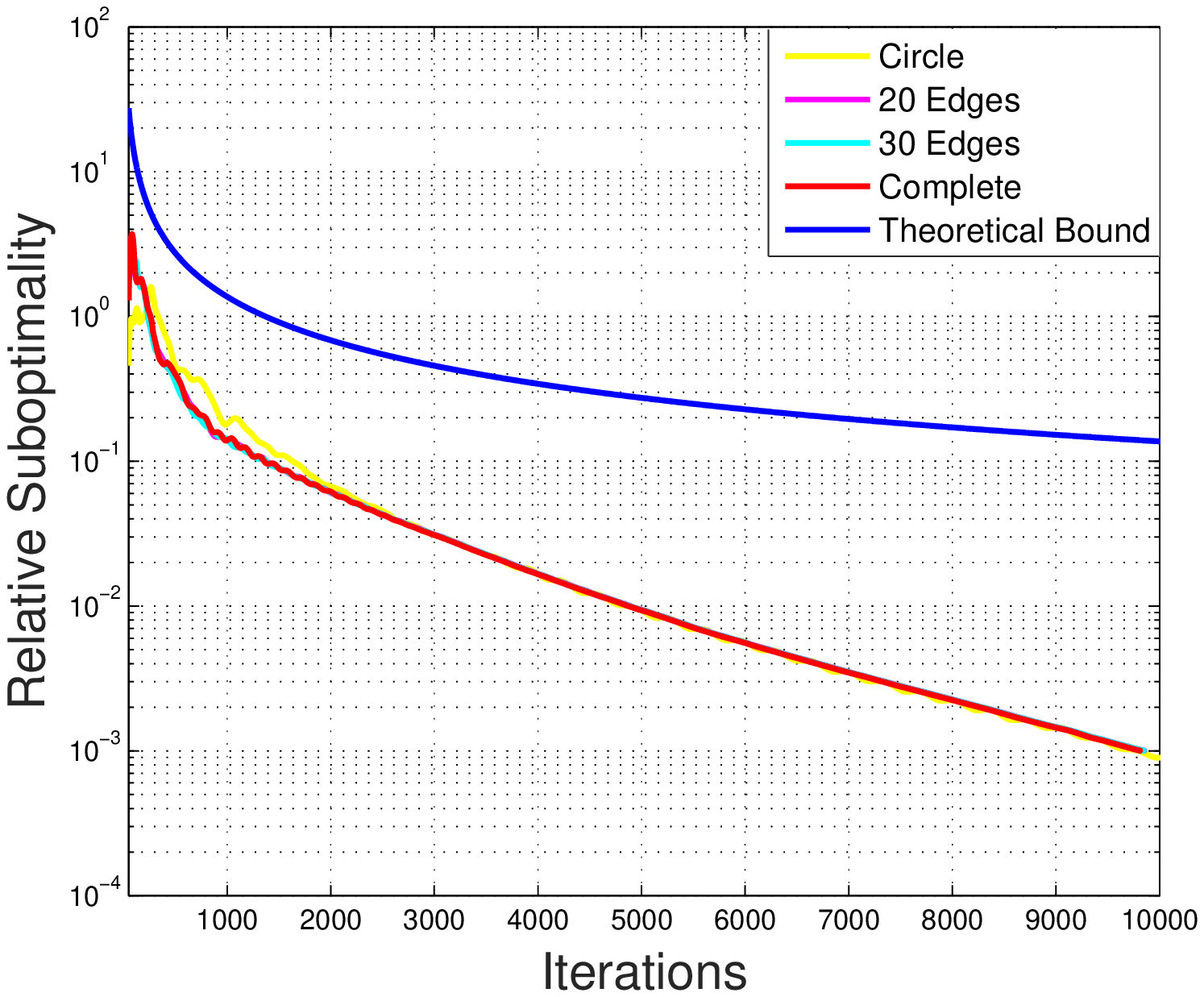}
        \caption{Relative suboptimality, $N=10$}
        \label{fig:relfun_sm}
   	\end{subfigure}%
	\begin{subfigure}[b]{0.45\textwidth}
        \centering
        \includegraphics[width = \textwidth]{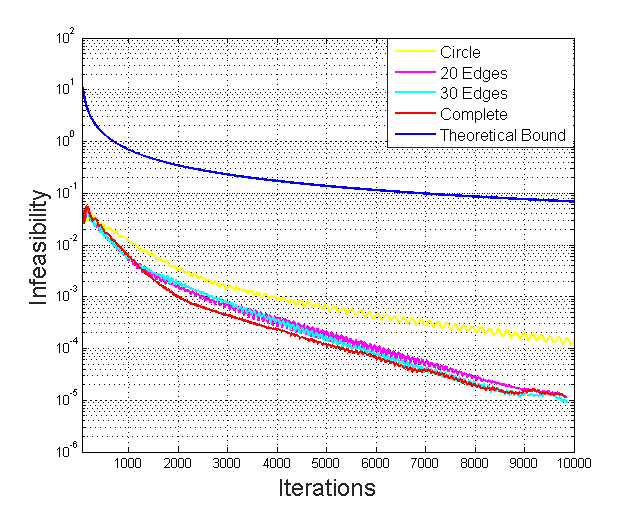}
        \caption{Infeasibility, $N=10$}
        \label{fig:infeas_sm}
	\end{subfigure}
    \begin{subfigure}[b]{0.45\textwidth}
        \centering
        \includegraphics[width = \textwidth]{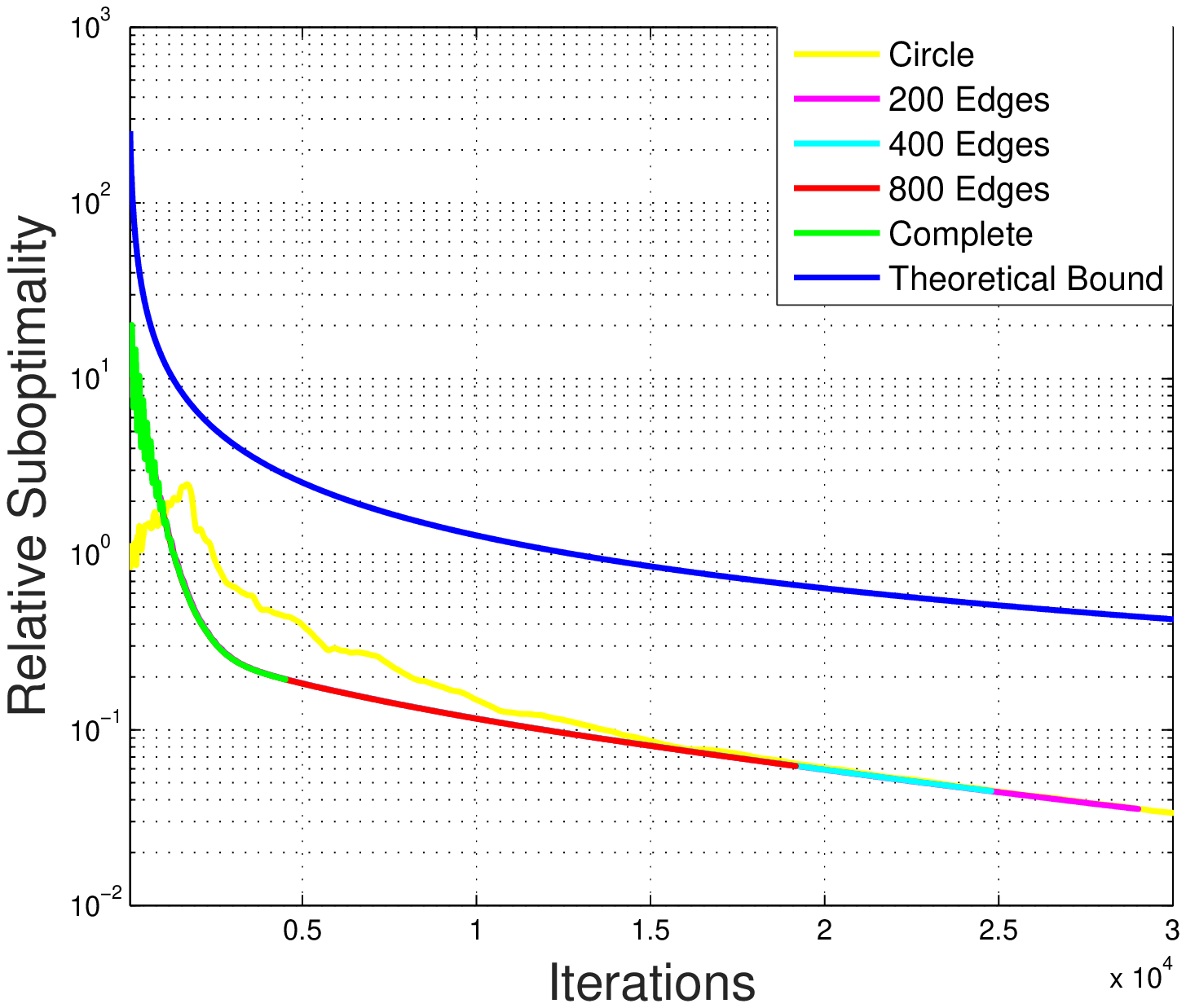}
        \caption{Relative suboptimality, $N=100$}
        \label{fig:relfun_lg}
   	\end{subfigure}%
	\begin{subfigure}[b]{0.45\textwidth}
        \centering
        \includegraphics[width = \textwidth]{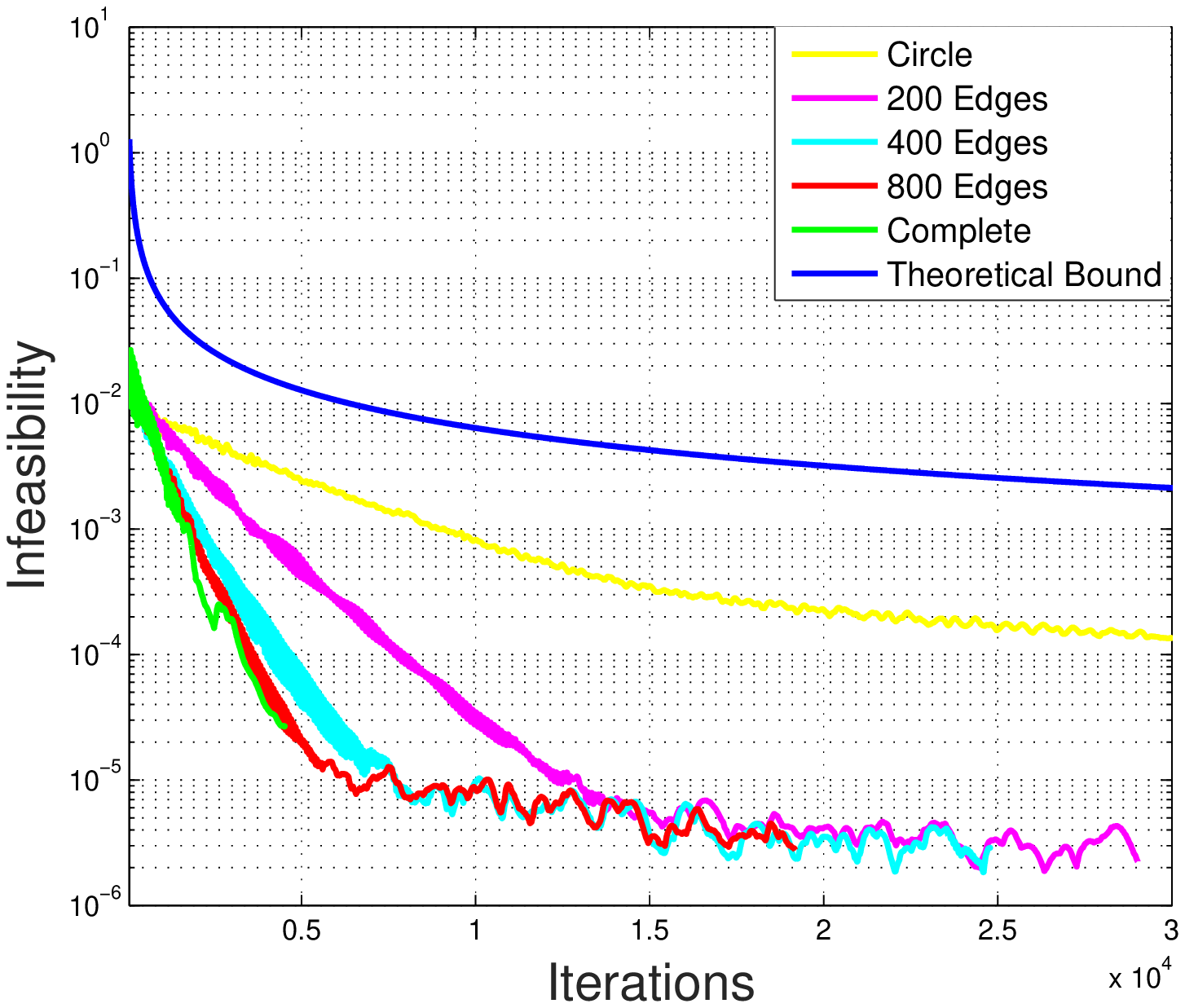}
        \caption{Infeasibility, $N=100$}
        \label{fig:infeas_lg}
	\end{subfigure}
\caption{The effect of increasing connectivity}
\end{figure}
In Fig.~\ref{fig:relfun_sm} and~\ref{fig:infeas_sm} we display the topology effect on convergence of relative optimality and consensus violation when $N=10$; and in Fig.~\ref{fig:relfun_lg} and~\ref{fig:infeas_lg} for networks with $N=100$.
In all these figures, we also plot the theoretical error bounds for the circle network in both Fig~\ref{fig:relfun_sm} and Fig.~\ref{fig:infeas_sm} -- to avoid crowding the figures, we only show the curve for the circle network as it is the loosest one among the others. 
In Fig.~\ref{fig:infeas_sm} and Fig.~\ref{fig:infeas_lg}, we observe that as more edges are added to the network, the convergence rate improves as expected -- improvement in consensus violation is more noticeable than that in suboptimality.

Fig.~\ref{fig:relfun_sm_den} and~\ref{fig:infeas_sm_den} compare convergence rates of DPGA as $|\cE|/|\cN|$, the density of edges in the network changes. We tested with $2$ and $3$ average number of edges per node for small-world networks with $N=10$ and $N=50$. It is worth noting that the network size has more impact on convergence rate than average edge density, i.e., the smaller the network faster the convergence is. On the other hand, for fixed size network, higher the density faster the convergence is.
\begin{figure}[h]
\centering
	\begin{subfigure}[b]{0.45\textwidth}
        \centering
        \includegraphics[width = \textwidth]{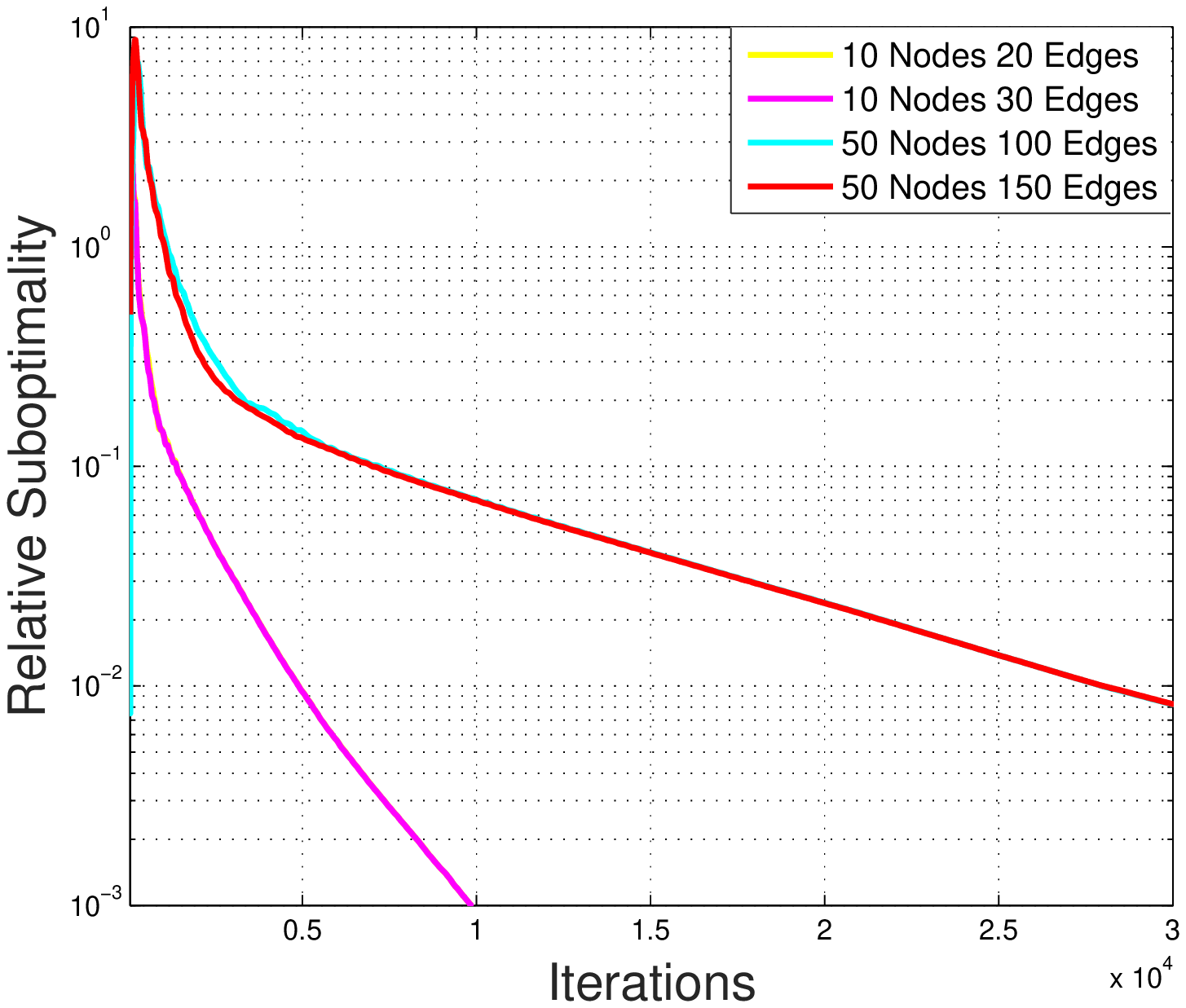}
        \caption{Relative sub-optimality}
        \label{fig:relfun_sm_den}
   	\end{subfigure}%
    ~
    	\begin{subfigure}[b]{0.45\textwidth}
        \centering
        \includegraphics[width = \textwidth]{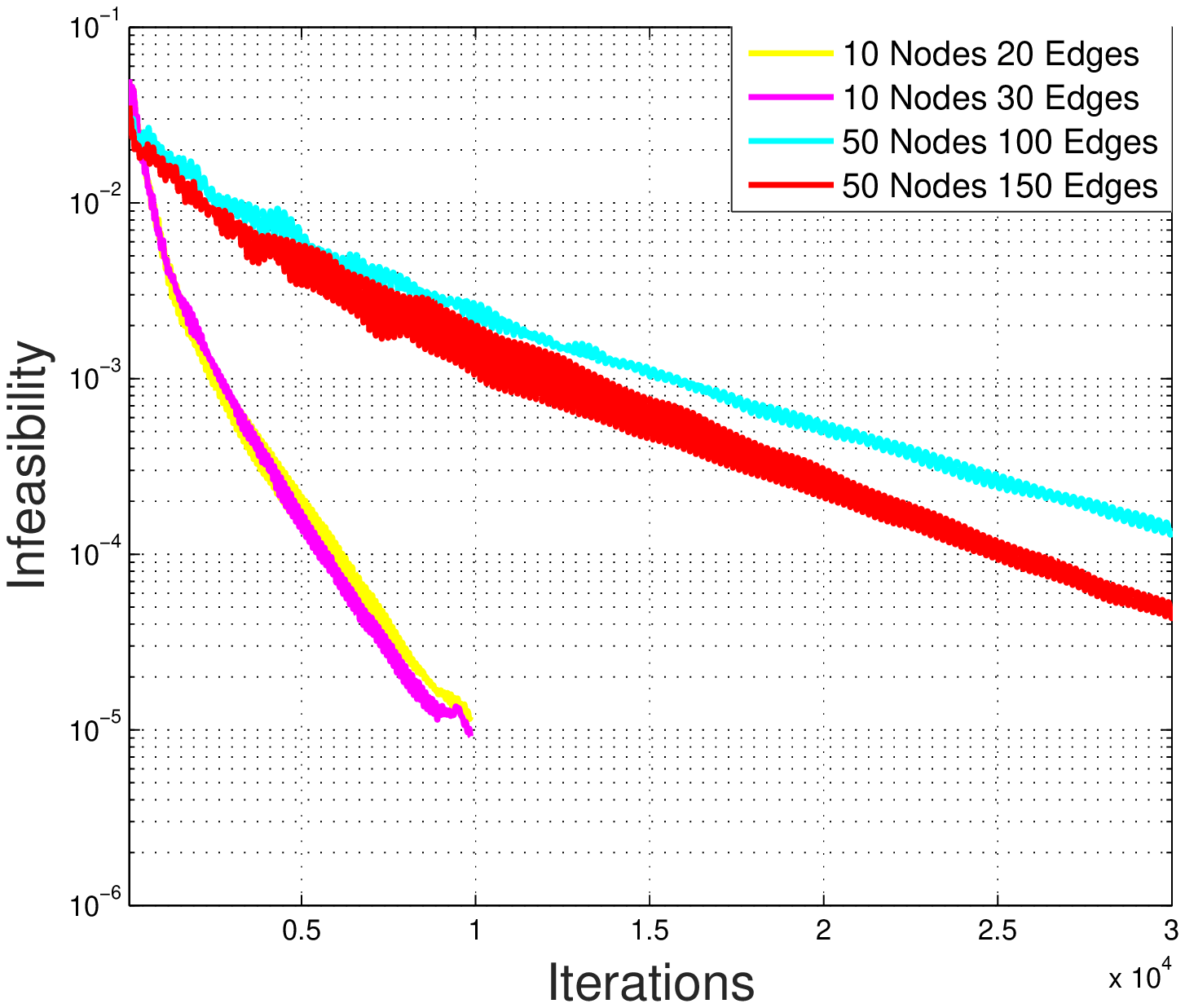}
        \caption{Infeasibility}
        \label{fig:infeas_sm_den}
    	\end{subfigure}
\caption{The effect of connectivity vs network size}
\vspace*{-5mm}
\end{figure}

Finally, we compared DPGA with SDPGA when the noise variance $\sigma\in\{0.01, 0.1, 1\}$ on a random small-world network with $N=10$ and $|\cE|=20$. Although DPGA is clearly faster than SDPGA, it turns out that the theoretical $\cO(1/\sqrt{t})$ rate for SDPA is not tight and empirically SDPA performs much better even though diminishing stepsize of $\cO(1/\sqrt{k})$ is used.
\begin{figure}[h]
\centering
	\begin{subfigure}[b]{0.45\textwidth}
        \centering
        \includegraphics[width = \textwidth]{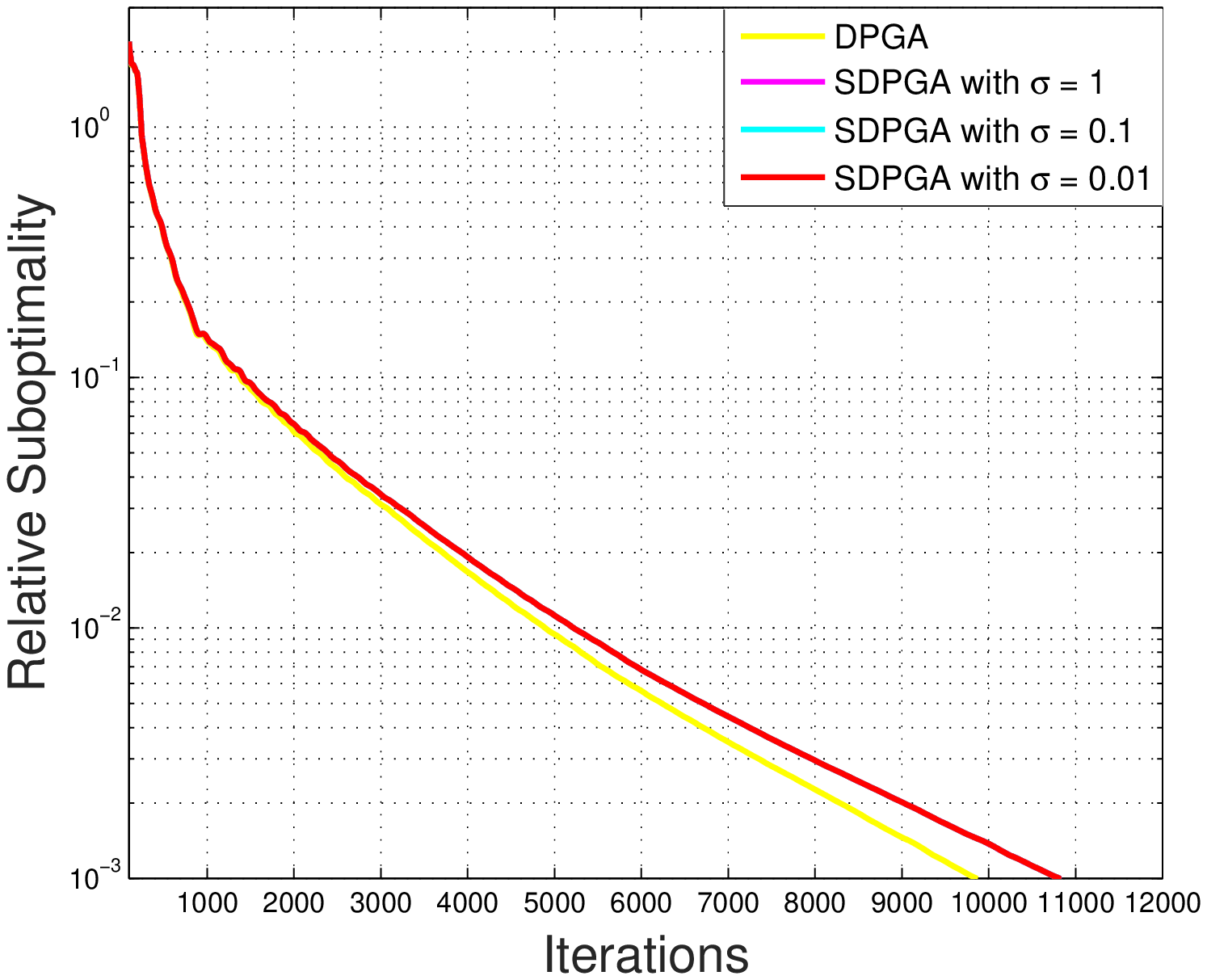}
        \caption{Relative sub-optimality}
        \label{fig:relfun_sdpga}
   	\end{subfigure}%
    ~
    	\begin{subfigure}[b]{0.45\textwidth}
        \centering
        \includegraphics[width = \textwidth]{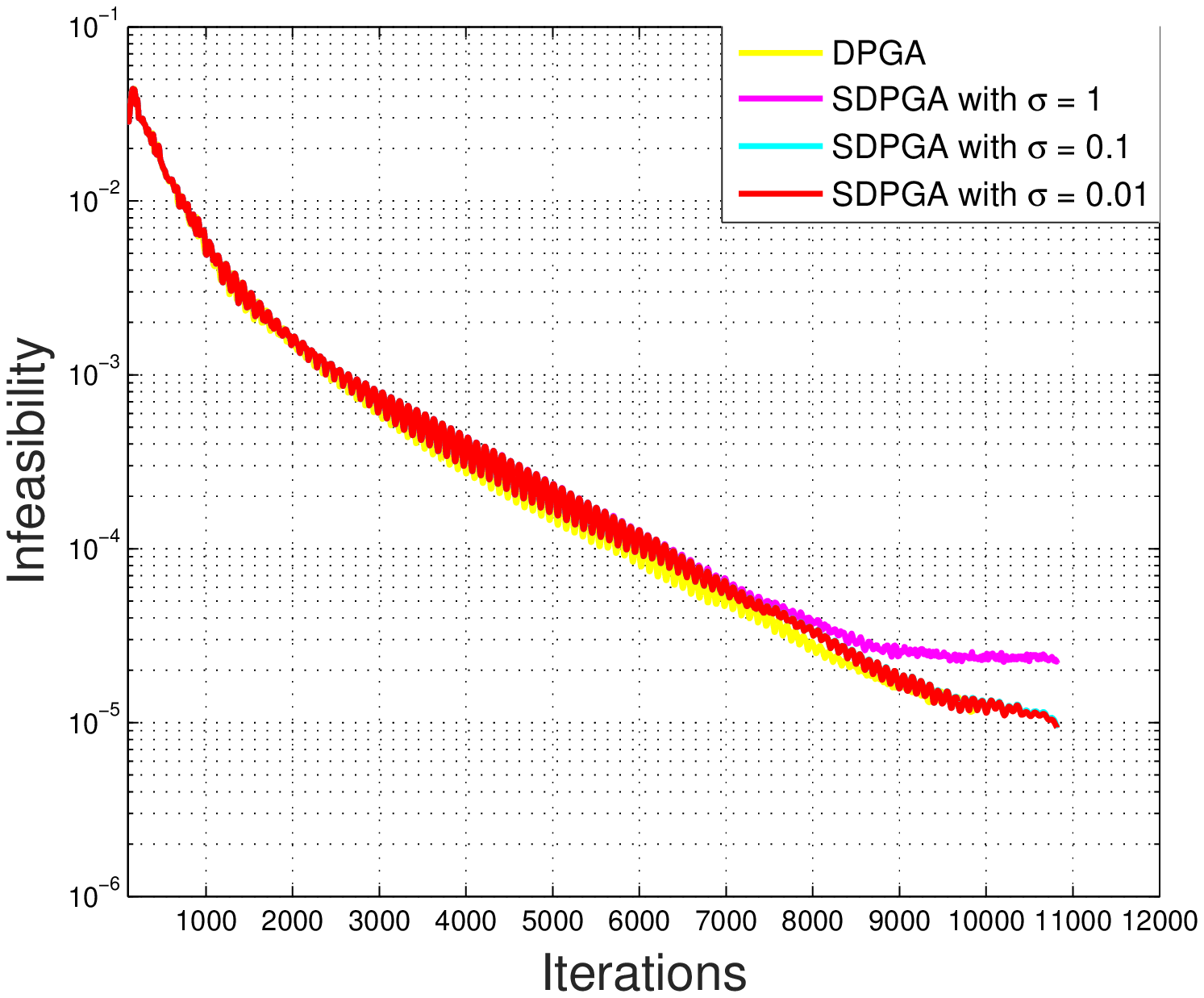}
        \caption{Infeasibility}
        \label{fig:infeas_sdpga}
    	\end{subfigure}
\caption{The effect of connectivity vs network size}
\end{figure}
\section{Conclusion}
In this paper, we studied distributed proximal gradient ADMM and its stochastic counterpart for distributed minimization of composite convex functions over connected networks. The convergence rates of these methods were analyzed. Comparing with existing works, the advantages of our methods are as follows: DPGA, DPGA-W, SDPGA and SDPGA-W are fully distributed, i.e., the agents are \emph{not} required to know any global parameters depending on the entire network topology, e.g., the second smallest eigenvalue of the Laplacian; instead, we only assume that agents know who their neighbors are. Using only local communication, our \emph{node-based} distributed algorithms require less communication burden and memory storage compared to edge-based distributed algorithms. The proposed algorithms consist of a \emph{single loop}, i.e., there are no outer and inner iteration loops; therefore, they are easy and practical to be implemented over distributed networks. To sum up, there are many practical problems where one can compute the prox map for $\xi_i$ efficiently; however, computing the prox map for $\Phi_i=\xi_i+f_i$ is not easy. The methods proposed in this paper can compute an $\epsilon$-optimal $\epsilon$-feasible solution in
$\cO(\epsilon^{-1})$ iterations \emph{without} assuming bounded $\grad f_i$ for any $i\in\cN$, where each iteration requires computing $\prox{\xi_i}$ and $\grad f_i$ for $i\in\cN$, and one or two communication rounds among the neighbors -- hence, $\cO(\epsilon^{-1})$ communications per node in total.

\ifCLASSOPTIONcaptionsoff
  \newpage
\fi

\bibliography{paper}
\bibliographystyle{unsrt}
\end{document}

%% file: defs.tex
\def\grad{\nabla}

\def\bg{\mathbf{g}}

\def\bp{\mathbf{p}}
\def\bq{\mathbf{q}}

\def\bx{\mathbf{x}}  
\def\by{\mathbf{y}}
\def\bz{\mathbf{z}}

\def\cD{\mathcal{D}}
\def\cE{\mathcal{E}}

\def\cG{\mathcal{G}}

\def\cL{\mathcal{L}}

\def\cN{\mathcal{N}}
\def\cO{\mathcal{O}}
\def\cP{\mathcal{P}}

\def\cY{\mathcal{Y}}
\def\cZ{\mathcal{Z}}

\definecolor{LightCyan}{rgb}{0.88,1,1}

\def\smskip{\smallskip}

\def\texitem#1{\par\smskip\noindent\hangindent 25pt
               \hbox to 25pt {\hss #1 ~}\ignorespaces}

\def\norm#1{\left\|#1\right\|}

\newcommand{\BEAS}{\begin{eqnarray*}}
\newcommand{\EEAS}{\end{eqnarray*}}
\newcommand{\BEA}{\begin{eqnarray}}
\newcommand{\EEA}{\end{eqnarray}}
\newcommand{\BEQ}{\begin{eqnarray}}
\newcommand{\EEQ}{\end{eqnarray}}
\newcommand{\BIT}{\begin{itemize}}
\newcommand{\EIT}{\end{itemize}}
\newcommand{\BNUM}{\begin{enumerate}}
\newcommand{\ENUM}{\end{enumerate}}

\newcommand{\BA}{\begin{array}}
\newcommand{\EA}{\end{array}}

\newcommand{\reals}{\mathbb{R}}
\newcommand{\integers}{\mathbb{Z}}

\newcommand{\Rank}{\mathop{\bf rank}}

\newcommand{\diag}{\mathop{\bf diag}}

\newcommand{\argmin}{\mathop{\rm argmin}}

\newcommand{\ri}{\mathop{\bf ri}}

\newcommand{\dom}{\mathop{\bf dom}}

\newif\ifpagenumbering
\pagenumberingtrue

\pagenumberingfalse

\newsavebox{\theorembox}
\newsavebox{\lemmabox}
\newsavebox{\assbox}
\savebox{\theorembox}{\noindent\bf Theorem}
\savebox{\lemmabox}{\noindent\bf Lemma}
\savebox{\assbox}{\noindent\bf Assumption}
\newtheorem{assumption}{\usebox{\assbox}}

%% file: IEEE_v12_One_Col_final_Aybat_arxiv.bbl
\begin{thebibliography}{10}

\bibitem{Tsitsiklis84_1T}
J.~N. Tsitsiklis.
\newblock {\em Problems in Decentralized Decision Making and Computation}.
\newblock PhD thesis, Massachusetts Institute of Technology, 1984.

\bibitem{chang2015multi}
T.-H. Chang, M.~Hong, and X.~Wang.
\newblock Multi-agent distributed optimization via inexact consensus {ADMM}.
\newblock {\em Signal Processing, IEEE Transactions on}, 63(2):482--497, 2015.

\bibitem{lesser2003distributed}
Victor Lesser, Charles~L Ortiz~Jr, and Milind Tambe.
\newblock {\em Distributed sensor networks: A multiagent perspective},
  volume~9.
\newblock Springer, 2003.

\bibitem{ling2010decentralized}
Qing Ling and Zhi Tian.
\newblock Decentralized sparse signal recovery for compressive sleeping
  wireless sensor networks.
\newblock {\em Signal Processing, IEEE Transactions on}, 58(7):3816--3827,
  2010.

\bibitem{ravazzi2015distributed}
Chiara Ravazzi, Sophie~Marie Fosson, and Enrico Magli.
\newblock Distributed iterative thresholding for $\ell_0/\ell_1$-regularized
  linear inverse problems.
\newblock {\em Information Theory, IEEE Transactions on}, 61(4):2081--2100,
  2015.

\bibitem{schizas2008consensus}
Ioannis~D Schizas, Alejandro Ribeiro, and Georgios~B Giannakis.
\newblock Consensus in ad hoc {WSN}s with noisy links - {Part I}: Distributed
  estimation of deterministic signals.
\newblock {\em Signal Processing, IEEE Transactions on}, 56(1):350--364, 2008.

\bibitem{forero2010consensus}
Pedro~A Forero, Alfonso Cano, and Georgios~B Giannakis.
\newblock Consensus-based distributed support vector machines.
\newblock {\em The Journal of Machine Learning Research}, 11:1663--1707, 2010.

\bibitem{mcdonald2010distributed}
Ryan McDonald, Keith Hall, and Gideon Mann.
\newblock Distributed training strategies for the structured perceptron.
\newblock In {\em Human Language Technologies: The 2010 Annual Conference of
  the North American Chapter of the Association for Computational Linguistics},
  pages 456--464. Association for Computational Linguistics, 2010.

\bibitem{yan2013distributed}
F.~Yan, S.~Sundaram, S.~Vishwanathan, and Y.~Qi.
\newblock Distributed autonomous online learning: Regrets and intrinsic
  privacy-preserving properties.
\newblock {\em Knowledge and Data Engineering, IEEE Transactions on},
  25(11):2483--2493, 2013.

\bibitem{lopuhaa1991breakdown}
Hendrik~P Lopuhaa and Peter~J Rousseeuw.
\newblock Breakdown points of affine equivariant estimators of multivariate
  location and covariance matrices.
\newblock {\em The Annals of Statistics}, pages 229--248, 1991.

\bibitem{mateos2010distributed}
Gonzalo Mateos, Juan~Andr{\'e}s Bazerque, and Georgios~B Giannakis.
\newblock Distributed sparse linear regression.
\newblock {\em Signal Processing, IEEE Transactions on}, 58(10):5262--5276,
  2010.

\bibitem{chambolle2015ergodic}
Antonin Chambolle and Thomas Pock.
\newblock On the ergodic convergence rates of a first-order primal--dual
  algorithm.
\newblock {\em Mathematical Programming}, pages 1--35, 2015.

\bibitem{condat2013primal}
Laurent Condat.
\newblock A primal--dual splitting method for convex optimization involving
  lipschitzian, proximable and linear composite terms.
\newblock {\em Journal of Optimization Theory and Applications},
  158(2):460--479, 2013.

\bibitem{Duchi12}
J.~C. Duchi, A.~Agarwal, and M.~J. Wainwright.
\newblock Dual averaging for distributed optimization: Convergence analysis and
  network scaling.
\newblock {\em IEEE Trans. Automat. Contr.}, 57(3):592--606, 2012.

\bibitem{nedic2009distributed}
Angelia Nedic and Asuman Ozdaglar.
\newblock Distributed subgradient methods for multi-agent optimization.
\newblock {\em Automatic Control, IEEE Transactions on}, 54(1):48--61, 2009.

\bibitem{wei2012_1}
Ermin Wei and Asuman Ozdaglar.
\newblock Distributed alternating direction method of multipliers.
\newblock In {\em Decision and Control (CDC), 2012 IEEE 51st Annual Conference
  on}, pages 5445--5450. IEEE, 2012.

\bibitem{makhdoumi2014broadcast}
Ali Makhdoumi and Asuman Ozdaglar.
\newblock Broadcast-based distributed alternating direction method of
  multipliers.
\newblock In {\em Communication, Control, and Computing (Allerton), 2014 52nd
  Annual Allerton Conference on}, pages 270--277. IEEE, 2014.

\bibitem{wei20131}
Ermin Wei and Asuman Ozdaglar.
\newblock On the o (1/k) convergence of asynchronous distributed alternating
  direction method of multipliers.
\newblock {\em arXiv preprint arXiv:1307.8254}, 2013.

\bibitem{jakovetic2011fast}
Dusan Jakovetic, Joao Xavier, and Jose~MF Moura.
\newblock Fast distributed gradient methods.
\newblock {\em Automatic Control, IEEE Transactions on}, 59(5):1131--1146,
  2014.

\bibitem{chen2012fast}
Annie~I Chen and Asuman Ozdaglar.
\newblock A fast distributed proximal-gradient method.
\newblock In {\em Communication, Control, and Computing (Allerton), 2012 50th
  Annual Allerton Conference on}, pages 601--608. IEEE, 2012.

\bibitem{shi2015proximal}
Wei Shi, Qing Ling, Gang Wu, and Wotao Yin.
\newblock A proximal gradient algorithm for decentralized composite
  optimization.
\newblock {\em Submitted to IEEE Transactions on Signal Processing}, 2015.

\bibitem{icml2015_aybat15}
N.~S. Aybat, Z.~Wang, and G.~Iyengar.
\newblock An asynchronous distributed proximal gradient method for composite
  convex optimization.
\newblock In {\em Proceedings of the 32nd International Conference on Machine
  Learning (ICML-15)}, pages 2454--2462. JMLR Workshop and Conference
  Proceedings, 2015.

\bibitem{shi2015extra}
Wei Shi, Qing Ling, Gang Wu, and Wotao Yin.
\newblock Extra: An exact first-order algorithm for decentralized consensus
  optimization.
\newblock {\em SIAM Journal on Optimization}, 25(2):944--966, 2015.

\bibitem{bianchi2014stochastic}
Pascal Bianchi, Walid Hachem, and Franck Iutzeler.
\newblock A stochastic primal-dual algorithm for distributed asynchronous
  composite optimization.
\newblock In {\em Proceedings of the 2nd Global Conference on Signal and
  Information Processing (GlobalSIP)}, pages 732--736, 2014.

\bibitem{Ling15_1J}
Q.~Ling, W.~Shi, G.~Wu, and A.~Ribeiro.
\newblock {DLM}: Decentralized linearized alternating direction method of
  multipliers.
\newblock {\em Signal Processing, IEEE Transactions on}, 63(15):4051--4064,
  2015.

\bibitem{hong2015stochastic}
Mingyi Hong and Tsung-Hui Chang.
\newblock Stochastic proximal gradient consensus over random networks.
\newblock {\em arXiv preprint arXiv:1511.08905}, 2015.

\bibitem{tsianos2012push}
Konstantinos~I Tsianos, Sean Lawlor, and Michael~G Rabbat.
\newblock Push-sum distributed dual averaging for convex optimization.
\newblock In {\em 2012 IEEE 51st IEEE Conference on Decision and Control
  (CDC)}, pages 5453--5458. IEEE, 2012.

\bibitem{nedic2015distributed}
Angelia Nedi{\'c} and Alex Olshevsky.
\newblock Distributed optimization over time-varying directed graphs.
\newblock {\em IEEE Transactions on Automatic Control}, 60(3):601--615, 2015.

\bibitem{lobel2011distributed}
Ilan Lobel and Asuman Ozdaglar.
\newblock Distributed subgradient methods for convex optimization over random
  networks.
\newblock {\em IEEE Transactions on Automatic Control}, 56(6):1291--1306, 2011.

\bibitem{Ma-Zhang-EGADM-2013}
T.~Lin, S.~Ma, and S.~Zhang.
\newblock An extragradient-based alternating direction method for convex
  minimization.
\newblock {\em Accepted in Foundations of Computational Mathematics}, 2015.

\bibitem{Gao-Jiang-Zhang-2014}
X.~Gao, B.~Jiang, and S.~Zhang.
\newblock On the information-adaptive variants of the {ADMM}: an iteration
  complexity perspective.
\newblock {\em preprint}, 2014.

\bibitem{rockafellar2015convex}
Ralph~Tyrell Rockafellar.
\newblock {\em Convex analysis}.
\newblock Princeton university press, 1997.

\bibitem{Beck09}
A.~Beck and M.~Teboulle.
\newblock A fast iterative shrinkage-thresholding algorithm for linear inverse
  problems.
\newblock {\em SIAM J. Img. Sci.}, 2(1):183--202, March 2009.

\bibitem{Goldfarb-Scheinberg-fastlinesearch2011}
K.~Scheinberg, D.~Goldfarb, and X.~Bai.
\newblock Fast first-order methods for composite convex optimization with
  backtracking.
\newblock {\em Foundations of Computational Mathematics}, 14(3):389--417, 2014.

\bibitem{Yuan11_2J}
Bingsheng He and Xiaoming Yuan.
\newblock On the $\mathcal{O}(1/n)$ convergence rate of the douglas-rachford
  alternating direction method.
\newblock {\em SIAM Journal on Numerical Analysis}, 50(2):700--709, 2012.

\bibitem{Aybat15_2J}
N.S. Aybat and G.~Iyengar.
\newblock An alternating direction method with increasing penalty for stable
  principal component pursuit.
\newblock {\em Computational Optimization and Applications}, 61(3):635--668,
  2015.

\bibitem{nesterov2013gradient}
Yu~Nesterov.
\newblock Gradient methods for minimizing composite functions.
\newblock {\em Mathematical Programming}, 140(1):125--161, 2013.

\bibitem{Tseng08}
Paul Tseng.
\newblock On accelerated proximal gradient methods for convex-concave
  optimization.
\newblock {\em submitted to SIAM Journal on Optimization}, 2008.

\end{thebibliography}
